\documentclass[12pt]{amsart}

\newtheorem{theorem}{Theorem}[section] 
\newtheorem{corollary}[theorem]{Corollary} 
 
\newtheorem{lemma}[theorem]{Lemma} 
\newtheorem{proposition}[theorem]{Proposition}

\theoremstyle{definition}

\theoremstyle{remark}

  {\end{list}}

\usepackage{amssymb} 
\usepackage{verbatim}

\newcommand{\ord}{\operatorname{ord}}
\newcommand{\Aut}{\operatorname{Aut}}

\renewcommand{\deg}{{\operatorname{deg}}}

\newcommand{\GL}{\operatorname{GL}}

\newcommand{\Res}{\operatorname{Res}}

\newcommand{\Adj}{\operatorname{Adj}}

\newcommand{\diam}{\operatorname{diam}}

\newcommand{\MinRes}{\operatorname{MinResLoc}}
\newcommand{\Aff}{\operatorname{Aff}}
\newcommand{\new}{\operatorname{new}}

\def\ord{{\mathop{\rm ord}}}
\def\ordRes{{\mathop{\rm ordRes}}}

\def\deg{{\mathop{\rm deg}}}

\def\Aut{{\mathop{\rm Aut}}}
\def\GL{{\mathop{\rm GL}}}

\def\Berk{{\mathop{\rm Berk}}}
\def\Res{{\mathop{\rm Res}}}

\def\Fix{{\mathop{\rm Fix}}}
\def\Repel{{\mathop{\rm Repel}}}

\def\vv{{\vec{v}}}

\def\v1{{\vec{1}}}
\def\vbb1{\vec{{\mathbf 1}}}

\def\BF1{{\mathbf 1}}

\def\CC{{\mathbb C}}

\def\FF{{\mathbb F}}

\def\HH{{\mathbb H}}

\def\PP{{\mathbb P}}
\def\QQ{{\mathbb Q}}
\def\RR{{\mathbb R}}

\def\ZZ{{\mathbb Z}}

\def\cF{{\mathcal F}}

\def\cO{{\mathcal O}}

\def\abar{{\overline{a}}}
\def\bbar{{\overline{b}}}
\def\cbar{{\overline{c}}}

\def\fbar{{\overline{f}}}
\def\gbar{{\overline{g}}}
\def\hbar{{\overline{h}}}

\def\betabar{{\overline{\beta}}}

\def\fM{{\mathfrak M}}

\def\fR{{\mathfrak R}}

\def\fa{{\mathfrak{a}}}

\def\fp{{\mathfrak{p}}}

\def\tk{{\widetilde{k}}}

\def\tvarphi{{\widetilde{\varphi}}}

\def\tF{{\widetilde{F}}}
\def\tG{{\widetilde{G}}}

\def\tL{{\widetilde{L}}}

\def\tS{{\widetilde{S}}}

\def\tX{{\widetilde{X}}}
\def\tY{{\widetilde{Y}}}

\def\talpha{{\widetilde{\alpha}}}

\def\tphi{{\widetilde{\varphi}}}

\def\Berk{{\rm Berk}}

\def\<{{\langle }}
\def\>{{\rangle }}

\def\<<{{\langle \! \langle}}
\def\>>{{\rangle \! \rangle}} 

\def\({(\!(}
\def\){)\!)}

\def\[{[\![}
\def\]{]\!]}

\DeclareMathSymbol{\varnothing} {\mathord}{AMSb}{"3F} 
 
\theoremstyle{definition} 
\theoremstyle{remark} 

\begin{document}
\title{The Minimal Resultant Locus}

\author{Robert Rumely}
\address{Robert Rumely\\ 
Department of Mathematics\\
University of Georgia\\
Athens, Georgia 30602\\
USA}
\email{rr@math.uga.edu}

\date{April 5, 2013}
\thanks{Work carried out in part during the ICERM Program ``Complex and $P$-adic dynamics'', Spring 2012.}

\begin{abstract}
Let $K$ be a complete, algebraically closed,
nonarchimedean  valued field, and let $\varphi(z) \in K(z)$ have degree $d \ge 2$.  
We give an algorithm to determine whether $\varphi$ has potential good reduction over $K$,   
based on a geometric reformulation of the problem using the Berkovich Projective Line. 
We show the minimal resultant is is either achieved at a single point in $\PP^1_{\Berk}$, or on a segment,
and that minimal resultant locus is contained in the tree in $\PP^1_\Berk$ 
spanned by the fixed points and poles of $\varphi$. 
When $\varphi$ is defined over $\QQ$ the algorithm runs in probabilistic polynomial time.
If $\varphi$ has potential good reduction, and is defined over a subfield $H \subset K$, 
we show there is an $L \subset K$ with $[L:H] \le (d+1)^2$ such that $\varphi$ has good reduction over $L$. 
\end{abstract}

\maketitle

Let $K$ be a complete, algebraically closed nonarchimedean valued field with absolute value $| \cdot |$
and associated valuation $\ord(\cdot) = -\log(| \cdot |)$. 
Write $\cO$ for the ring of integers of $K$, $\fM$ for its maximal ideal, and $\tk$ for its residue field.

Let $\varphi(z)  \in K(z)$  be a rational function with $\deg(\varphi) = d \ge 1$.  
Then there are homogeneous polynomials $F(X,Y), G(X,Y) \in K[X,Y]$ of degree $d$,  
having no common factor, such that the map $[X:Y] \mapsto [F(X,Y):G(X,Y)]$ gives the action of $\varphi$ on $\PP^1$.    
After scaling $F$ and $G$ appropriately, one can arrange that $F$ and $G$ belong to $\cO[X,Y]$ 
and that at least one of their coefficients is a unit in $\cO$.  Such a pair $(F,G)$ 
is called a {\em normalized representation} of $\varphi$; it is unique up to scaling by a unit in $\cO$.
Writing $F(X,Y) = f_d X^d +f_{d-1} X^{d-1}Y + \cdots + f_0 Y^d$ 
and  $G(X,Y) = g_d X^d + g_{d-1} X^{d-1}Y + \cdots + g_0 Y^d$,  
the resultant of $F$ and $G$ is 
\begin{equation} \label{FGResultant} 
\Res(F,G) \ = \ \det\Bigg( \ \left[ \begin{array}{ccccccc} 
                                     f_d & f_{d-1} & \cdots &  f_0   &         &        &     \\
                                         & f_d & f_{d-1}    & \cdots &    f_0           &     \\
                                         &     &        &         & \vdots &        &     \\
                                         &     &        &    f_d  &   f_{d-1}  & \cdots & f_0 \\
                                     g_d & g_{d-1} & \cdots &  g_0   &         &        &     \\
                                         & g_d & g_{d-1}    & \cdots &    g_0           &     \\
                                         &     &        &         & \vdots &        &     \\
                                         &     &        &    g_d  &   g_{d-1}  & \cdots & g_0 
                             \end{array} \right] \ \Bigg) \ ,
\end{equation} 
and the quantity 
\begin{equation} \label{ResPhi}
\ordRes(\varphi) \ := \ \ord(\Res(F,G))
\end{equation}
is 
independent of the choice of normalized representation.  By construction, it is non-negative.  

The {\em reduction} $\tphi$ is the map $[\tX:\tY] \mapsto [\tF(\tX,\tY):\tG(\tX,\tY)]$ on $\PP^1(\tk)$ 
obtained by reducing $F$ and $G$ $\pmod{\fM}$ and eliminating common factors.  
If $\tphi$ has degree $d$, then $\varphi$ is said to have
{\em good reduction}.  
Likewise, $\varphi$ is said to have {\em potential good reduction} if after a change of coordinates
by some $\gamma \in \GL_2(K)$, 
the map $\varphi^{\gamma} = \gamma^{-1} \circ \varphi \circ \gamma$ 
has good reduction. It is well known (see e.g. \cite{Sil}, Theorem 2.15) 
that $\varphi$ has good reduction if and only if $\ordRes(\varphi) = 0$. 
\smallskip

It has been a long-standing problem to find an algorithm to decide
whether or not a given $\varphi$ has potential good reduction.  
When $\varphi$ is defined over a local field $H_v$,
Bruin and Molnar (\cite{BM}) recently gave an algorithm that determines when $\varphi$
has potential good reduction over $H_v$.
Their algorithm involves a recursive search, and depends
on the fact that $H_v$ is discretely valued.

In this paper we solve the problem by reformulating 
it in terms of the Berkovich projective line $\PP^1_{\Berk} = \PP^1_{\Berk}/K$.  We show that the map 
$\gamma \rightarrow \ordRes(\varphi^{\gamma})$ 
factors through a function $\ordRes_{\varphi}(\cdot)$ on $\PP^1_\Berk$ 
which is is continuous, piecewise affine, and convex upwards on each path.  
It takes on a minimum value. We study the properties of $\ordRes_\varphi(\cdot)$ and the set 
$\MinRes(\varphi) \subset \PP^1_{\Berk}$, the {\em Minimal Resultant Locus},  
where its minimal value is attained. 
We use this to give an algorithm that decides whether $\varphi$ has potential good reduction   
and finds a $\gamma$ for which $\varphi^{\gamma}$ has a minimal resultant.  
When $\varphi$ is defined over a 
subfield $H \subset K$, 
we obtain an \`{a} priori bound  of $(d+1)^2$
for the degree of an extension $L/H$ such that there is a $\gamma \in \GL_2(L)$ 
for which $\ordRes(\varphi^{\gamma})$ is minimal.  

\smallskip
Recall that $\PP^1_{\Berk}$ 
is a path-connected Hausdorff space containing $\PP^1(K)$.
By Berkovich's classification theorem (see for example \cite{B-R}, p.5), 
$\PP^1_{\Berk}$ can be viewed as a space whose points 
correspond to discs in $K$.  There are four types of points:   
type I points are the points of $\PP^1(K)$, which we regard as discs of radius $0$.  
Type II and III points correspond to discs $D(a,r) = \{z \in K : |z-a| \le r\}$,   
with type II points corresponding to discs $D(a,r)$ with $r$ in the value group $|K^{\times}|$,
and type III points corresponding to those with $r \notin |K^{\times}|$. 
The point $\zeta_G$ corresponding to $D(0,1)$ is called the {\em Gauss point}.
Type IV points serve to complete $\PP^1_{\Berk}$; they correspond to (cofinal equivalence classes of) 
sequences of nested discs with empty intersection.  
Paths in $\PP^1_{\Berk}$ correspond to ascending or descending chains of discs, 
or unions of chains sharing an endpoint.  
For example the path from $0$ to $1$
in $\PP^1_{\Berk}$ corresponds to the chains 
$\{D(0,r) : 0 \le r \le 1\}$ and $\{D(1,r) : 1 \ge r \ge 0\}$;  here $D(0,1) = D(1,1)$. 
Topologically, $\PP^1_{\Berk}$ is a tree:     
there is a unique path $[x,y]$ between any two points $x, y \in \PP^1_{\Berk}$.

The set $\HH_{\Berk} = \PP^1_{\Berk} \backslash \PP^1(K)$ is called the {\em Berkovich upper halfspace}; 
it carries a metric $\rho(x,y)$ called the {\em logarithmic path distance}, for which the length of the 
path corresponding to $\{D(a,r) : R_1 \le r \le R_2\}$ is $\log(R_2/R_1)$. 
There are two natural topologies on $\PP^1_{\Berk}$, called the {\em weak}\, and {\em strong} topologies. 
The weak topology on $\PP^1_{\Berk}$ is the coarsest one which makes the evaluation functionals 
$z \rightarrow |f(z)|$ continuous for all $f(z) \in K(z)$;   
under the weak topology, $\PP^1_{\Berk}$ is compact and $\PP^1(K)$ is dense in it.
The basic open sets for the weak topology are the path-components 
of $\PP^1_{\Berk} \backslash \{P_1, \ldots, P_n\}$
as $\{P_1, \ldots, P_n\}$ ranges over finite subsets of $\HH_{\Berk}$. 
The strong topology on $\PP^1_{\Berk}$ (which is finer than the weak topology) 
restricts to the topology on $\HH_{\Berk}$ induced by $\rho(x,y)$.  
The basic open sets for the strong topology 
are the $\rho(x,y)$-balls in $\HH_{\Berk}$, 
together with the basic open sets from the weak topology.  
Type II points are dense in $\PP^1_{\Berk}$ for both topologies.
The action of $\varphi$ on $\PP^1(K)$ extends functorially to an action on $\PP^1_{\Berk}$, 
which  is continuous for both topologies, and takes points of a given type to points of the same type.    
Similarly, the action of  $\GL_2(K)$ on $\PP^1(K)$ extends 
to an action on $\PP^1_{\Berk}$, which is continuous for both topologies, 
and preserves the type of each point.  The action of $\GL_2(K)$ also preserves the logarithmic path distance:  
$\rho(\gamma(x),\gamma(y)) = \rho(x,y)$ for all $x, y \in \HH_\Berk$ and all $\gamma \in \GL_2(K)$.   
For these and other facts, see (\cite{B-R}) and  
(\cite{BIJL}, \cite{Berk1}, \cite{Fab}, \cite{F-RL2}, \cite{FRLErgodic}, \cite{R-L1}).  

\smallskip
It follows from standard formulas for the resultant (see for example (Silverman \cite{Sil}, Exercise 2.7, p.75))
that for each $\gamma \in \GL_2(K)$ and each $\tau \in K^{\times} \cdot \GL_2(\cO)$,  
we have
\begin{equation*}
\ordRes(\varphi^{\gamma}) \ = \ \ordRes(\varphi^{\gamma \tau}) \ . 
\end{equation*}  
On the other hand, $\GL_2(K)$ acts transitively on type II points, and  
$K^{\times} \cdot \GL_2(\cO)$ is the stabilizer of the Gauss point.  
This means there is a well-defined function \ $\ordRes_{\varphi}(\cdot)$ \ 
on the type II points in $\PP^1_\Berk$, given by 
\begin{equation} \label{ordRes_varphiDef}
\ordRes_{\varphi}(\gamma(\zeta_G)) \ := \ \ordRes(\varphi^{\gamma}) \ .
\end{equation}

This observation is the key to our investigation.  Our main result is 

\begin{theorem}[Main Theorem] \label{MainThm}  Suppose $d = \deg(\varphi) \ge 2$.  
The function $\ordRes_{\varphi}(\cdot)$ on type {\rm II} points 
extends uniquely to a function $\ordRes_{\varphi} :\PP^1_{\Berk} \rightarrow [0,\infty]$
continuous with respect to the strong topology.  
On each path in $\PP^1_{\Berk}$, 
it is piecewise affine and convex upwards with respect to the logarithmic path distance.
It is finite on $\HH_{\Berk}$ and $\infty$ on $\PP^1(K)$.
It achieves a minimum on $\PP^1_{\Berk}$.  
The set $\MinRes(\varphi)$  where $\ordRes_{\varphi}(\cdot)$ takes on its minimum 
is contained in the tree $\Gamma_{\Fix,\varphi^{-1}(\infty)}$ spanned by the 
fixed points and poles of $\varphi$ in $\PP^1(K)$,   
and lies in  
$\{z \in \HH_\Berk : \rho(\zeta_G,z) \le \frac{2}{d-1} \ordRes(\varphi) \}$.
$\MinRes(\varphi)$  consists of a single type {\rm II} point if $d$ is even, 
and is a type {\rm II} point or a segment with type {\rm II} endpoints if $d$ is odd.  
If the minimum value of $\ordRes_{\varphi}(\cdot)$ is $0$ 
$($that is, if $\varphi$ has potential good reduction$)$,
then $\MinRes(\varphi)$ consists of a single point.
\end{theorem}

In the proof of Theorem \ref{MainThm}, one sees that each affine piece of $\ordRes_{\varphi}(\cdot)$ 
has an integer integer slope $m \equiv d^2+d \pmod{2d}$ 
with $-d^2-d \le m \le d^2+d$, and that breaks between affine pieces occur at type II points. 
By Proposition \ref{GeneralTreeProp}, in Theorem \ref{MainThm} the tree $\Gamma_{\Fix,\varphi^{-1}(\infty)}$ 
can be replaced by the tree $\Gamma_{\Fix,\varphi^{-1}(a)}$ 
spanned by the fixed points and the preimages of $a$, for any $a \in \PP^1(K)$.
The Theorem has the following consequences:  

\smallskip
(1) Relative to computations in $K$, there is an algorithm (Algorithm A) to determine 
whether or not $\varphi$ has potential good reduction. 
If it does, one can find a $\gamma \in \GL_2(K)$ such that $\varphi^{\gamma}$ has good reduction. 

Indeed, the algorithm is as follows.  
First, find the fixed points $\{P_0, \ldots, P_d\}$ and poles $\{Q_1, \ldots, Q_d\}$ 
of $\varphi$.  Choose one of the fixed points, say $P_0$, and restrict  $\ordRes_{\varphi}(\cdot)$ in turn 
to each of the $2d$ paths $[P_0,P_k]$ and $[P_0,Q_k]$ for $k = 1, \ldots, d$.  The resulting piecewise affine
functions can be computed and their minima found.  If the minimum value on some path is $0$,
then $\varphi$ has good reduction at the corresponding point.  
If all minima are positive, then $\varphi$ does not have potential good reduction. 
When $\varphi$ is defined over $\QQ$, Algorithm A can be implemented to run in probabilistic polynomial time.

When $\varphi$ is defined over a local field $H_v$, we give another algorithm (Algorithm B) which  
minimizes $\ordRes(\varphi^{\gamma})$ for $\gamma \in \GL_2(H_v)$. 
This algorithm is based on steepest descent, and runs in probabilistic polynomial time.  
It answers the same question as the Bruin-Molnar algorithm, but is more conceptual, 
and should be more efficient. 
However, the two algorithms have many aspects in common.

\smallskip
(2) If $\varphi$ is defined over a subfield $H \subset K$, there is an \`a priori bound 
of $(d+1)^2$ for the degree of an extension $L/H$ such that $\ordRes(\varphi^{\gamma})$ 
is minimal for some $\gamma \in \GL_2(L)$ (see Theorem \ref{DegreeBoundTheorem}).  
In particular, if $\varphi$ has potential good reduction, 
this is a bound for the degree of an extension
where it achieves good reduction. It follows from this 
that if $H$ is Henselian (in particular, if $H$ is complete),
the statement ``$\varphi$ has potential good reduction'' 
is first-order in the theory of $H$, in the sense of mathematical logic.

\smallskip 

(3) The Minimal Resultant Locus can be a segment of positive length 
(see Examples 2.5 and 2.7).   
Hence there can be fundamentally different coordinate changes 
(that is, coordinate changes by $\gamma$'s belonging to different cosets of $K^{\times} \cdot \GL_2(\cO)$)
for which $\varphi^\gamma$ has minimal resultant.  
However, this can only happen when $d$ is odd and $\varphi$
does not have potential good reduction.


\smallskip

(4) If $\varphi$ is defined over a subfield $H \subset K$, and $\varphi$ has potential good reduction, 
let $H_\varphi$ be the intersection of all fields $L$ with $H \subset L \subset K$ 
such that $\varphi^{\gamma}$ has good reduction
for some $\gamma \in \GL_2(L)$ (the `field of moduli for the good reduction problem').  
We give examples where $H_\varphi = H$ but $\varphi^{\gamma}$ 
does not have good reduction for any $\gamma \in \GL_2(H)$.
Thus there need not be a unique minimal extension $L/H$ where  
$\varphi$ achieves good reduction.  

\smallskip

(5) Suppose $H$ is a number field. An elliptic curve $E/H$ has a global minimal model over $H$
if and only if a certain class $[\fa_E]$ in the ideal class group of $\cO_H$,
the {\em Weierstrass class}, is principal.  
When $\varphi(z) \in H(z)$ and $\deg(\varphi) \ge 2$, 
Silverman has constructed an  ideal class $[\fa_{\varphi}]$   
such that if $\varphi$ has global minimal model over $H$, then $[\fa_{\varphi}]$
is trivial (see \cite{Sil}, Proposition 4.99). He asks if the converse is true 
(\cite{Sil},  p.237, Exercise 4.4.6(c)). 
We give examples of number fields $H$  and functions $\varphi(z) \in H(z)$  
for which $[\fa_{\varphi}]$ is trivial but $\varphi$ has no global minimal model.  

\medskip
Our second result concerns the stability of $\ordRes_\varphi(\cdot)$ and $\MinRes(\varphi)$
under perturbations of $\varphi$.  It also specifies the precision needed for numerical 
implementations of Algorithms A and B.

\begin{theorem} \label{StabilityThm}  Suppose $\varphi(z), \tvarphi(z) \in  K(z)$ have
degree $d \ge 2$, with normalized representations $(F,G)$, $(\tF,\tG)$ respectively.  
Put $R = \ordRes(\varphi)$, and let $M > 0$ be  arbitrary.  If 
\begin{equation} \label{Congruence}
\min\big(\ord(\tF-F),\ord(\tG-G)\big) \ > \ \max\big(R,\frac{1}{2d}(R + (d^2+d) M)\big) \ ,
\end{equation}
then $\ordRes_\varphi(\xi) = \ordRes_{\tvarphi}(\xi)$ 
for all $\xi$ with $\rho(\zeta_G,\xi) \le M$.  Let $f(d) = \frac{2d^2 + 3d-1}{2d^2 - 2d}$.
If
\begin{equation} \label{MinResCong}
\min\big(\ord(\tF-F),\ord(\tG-G)\big) \ > \ f(d) \cdot R \ , 
\end{equation}
then  $\MinRes(\varphi) = \MinRes(\tvarphi)$, and $\ordRes_\varphi(\xi) = \ordRes_{\tvarphi}(\xi)$ 
for all $\xi$ with $\rho(\zeta_G,\xi) \le \frac{2}{d-1} \ordRes(\varphi)$.
\end{theorem}

Note that $f(2) = 3.25$, $f(3) = 2.166 \cdots$, 
and $1 < f(d) < 2$ for $d \ge 4$.

\medskip
The structure of the paper is as follows.  In Section \ref{MainThmProofSection} 
we prove Theorems \ref{MainThm} and \ref{StabilityThm}.
In Section \ref{ExamplesSection} we give examples illustrating various phenomena which occur.  
In Section \ref{DiscApplicSection} we give applications of the theory.  
In Section \ref{AlgorithmsSection} we present Algorithms A and B.  Finally, in Section \ref{d=1Section} 
we prove an analogue of Theorem \ref{MainThm} when $d = 1$.

\section{Proof of the Main Theorems } \label{MainThmProofSection}  

In this section we establish Theorems \ref{MainThm} and \ref{StabilityThm}.  
Suppose $\varphi(z) \in K(z)$ has degree $d$.  Then 
\begin{equation*}
\varphi(z) \ = \ \frac{F(z,1)}{G(z,1)}
\end{equation*}
where $F(X,Y) = f_d X^d + f_{d-1} X^{d-1} Y + \cdots + f_0 Y^d$ and 
$G(X,Y) = g_d X^d + g_{d-1} X^{d-1} Y + \cdots + g_0 Y^d$ are homogeneous polynomials in $K[X,Y]$ of degree $d$ 
with no common factor. 
The pair $(F,G)$ is called a {\em representation} of $\varphi$;  it is unique up to scaling by a nonzero constant. 
Put $\ord(F) = \min_{0 \le i \le d} (\ord(f_i))$, $\ord(G) = \min_{0 \le i \le d} (\ord(g_i))$.    

The resultant of $F$ and $G$ is defined by the $2d \times 2d$ determinant in formula (\ref{FGResultant}).  
For any $c \in K^{\times}$, we have  $\Res(cF,cG) = c^{2d} \Res(F,G)$.  
By choosing $c$ so that $\ord(c)=\min(\ord(F),\ord(G))$ and replacing $(F,G)$ by $(c^{-1}F,c^{-1}G)$
we can assume that
\begin{equation*}
 \min\big(\ord(F),\ord(G)\big) \ = \ 0 \ ;
\end{equation*}
in this case $(F,G)$ is called a {\em normalized representation}
of $\varphi$, and $\ordRes(\varphi)$ is defined to be $\ord(\Res(F,G))$ as in (\ref{ResPhi}).  
Clearly  $\ordRes(\varphi)$ is independent of the choice of normalized representation, 
and $\ordRes(\varphi) \ge 0$.  

Whether or not $(F,G)$ is normalized, we have 
\begin{equation}
\ordRes(\varphi) \ = \ \ord(\Res(F,G)) - 2d \min(\ord(F),\ord(G)) \ .
\end{equation}   
Given
$\gamma = \left[ \! \begin{array}{cc} A & B \\ C & D \end{array} \! \right] \in \GL_2(K)$,
let  
$\Adj(\gamma) = \left[ \! \begin{array}{rr} D & -B \\ -C & A \end{array} \! \right]$ 
and define $(F^{\gamma},G^{\gamma})$ by 
\begin{equation} \label{FGg1} 
\left[ \! \begin{array}{c} F^{\gamma}(X,Y) \\ G^{\gamma}(X,Y) \end{array} \! \right]  = 
                    \Adj(\gamma) \! \circ \! \left[ \! \begin{array}{c} F \\ G \end{array} \! \right] \! \circ \! \gamma \! \circ \! 
                               \left[ \! \begin{array}{c} X \\ Y \end{array} \! \right]
                     = \left[ \! \begin{array}{c} D F(AX+BY) - BG(CX+DY) \\ -CF(AX+BY) + AG(CX+DY) \! \end{array} \right] \ .
\end{equation} 
Then $(F^{\gamma},G^{\gamma})$ is a homogeneous representation of $\varphi^{\gamma}$.  
It is known (see (\cite{Sil}, Exercise 2.7(c), p.76) that 
$\Res(F^{\gamma},G^{\gamma})  =  \Res(F,G) \cdot \det(\gamma)^{d^2 + d}$,
so 
\begin{equation} \label{KeyFormula}
\ordRes(\varphi^{\gamma}) \ = \ \ordRes(F,G) + (d^2 + d) \, \ord(\det(\gamma)) - 2d \min(\ord(F^{\gamma}), \ord(G^{\gamma})) \ . 
\end{equation} 

\smallskip
We will prove Theorems \ref{MainThm} and \ref{StabilityThm} after a series of preliminary results.  
In Theorem \ref{MainThm} it is assumed that $d \ge 2$;
however, for use in \S\ref{d=1Section}, we will develop the theory for $d \ge 1$, 
and make explicit the places where $d \ge 2$ is used. 

\smallskip
We begin by recalling some facts about the action of $\GL_2(K)$ on $\PP^1_\Berk$.

\begin{proposition} \label{GL_2Prop} 
The natural action of $\GL_2(K)$ on $\PP^1(K)$ extends to an action on $\PP^1_\Berk$
such that   

$(A)$ The stabilizer of $\zeta_G$ in $\GL_2(K)$ is $K^{\times} \cdot \GL_2(\cO)\,;$

$(B)$ For each $\gamma \in \GL_2(K)$, 
one has $\rho(\gamma(x),\gamma(y)) = \rho(x,y)$ for all $x, y \in \HH_\Berk\,;$  

$(C)$ For each $\gamma \in \GL_2(K)$ and each path $[x,y]$, one has 
$\gamma([x,y]) = [\gamma(x),\gamma(y)]\,;$  

$(D)$ For any triple $(a_0,A,a_1)$ where $a_0, a_1 \in \PP^(K)$, $a_0 \ne a_1$, 
and $A$ is a type II point in $[x,y]$, if $(b_0,B,b_1)$ is another triple of the same kind, 
there is a $\gamma \in \GL_2(K)$ such that $\gamma(a_0) = b_0$, $\gamma(A) = B$, and $\gamma(a_1) = b_1$.
In particular, $\GL_2(K)$ acts transitively on the type II points in $\PP^1_{\Berk}$.
\end{proposition} 

\begin{proof}  As discussed in (\cite{B-R}, \S2.3), 
the natural action of any rational function $f(z) \in K(z)$ on $\PP^1(K)$ extends uniquely to 
a continuous action on $\PP^1_\Berk$.  
For part (A), suppose $\gamma \in \GL_2(K)$ stabilizes $\zeta_G$, and let $\gamma(0) = a$, $\gamma(1) = b$,
$\gamma(\infty) = c$.   By (\cite{B-R}, Lemma 2.17) $\gamma(z)$ has nonconstant
reduction, so the reductions  $\abar$, $\bbar$, and $\cbar$ are distinct in $\PP^1(\tk)$.  
If none of $\abar, \bbar, \cbar$ is $\overline{\infty}$, then  
\begin{equation} \label{GL2O}
\gamma_0(z) \ = \ \frac{cz - a(b-c)/(b-a)}{z - (b-c)/(b-a)} 
\end{equation} 
belongs to $\GL_2(\cO)$ and satisfies $\gamma_0(0) = a$, $\gamma_0(1) = b$, $\gamma_0(\infty) = c$. 
If one of the reductions is $\overline{\infty}$, by making simple modifications to (\ref{GL2O})  
one still finds a $\gamma_0 \in \GL_2(\cO)$ with $\gamma_0(0) = a$, $\gamma_0(1) = b$,
$\gamma_0(\infty) = c$.  Since $\gamma_0^{-1} \circ \gamma \in \GL_2(K)$ fixes three points
in $\PP^1(K)$, it must be a multiple of the identity matrix.  Part (B) is (\cite{B-R}, Proposition 2.30).  
Part (C) follows from the fact that if $\gamma \in \GL_2(K)$, 
the action of $\gamma$ on $\PP^1_{\Berk}$ must be bijective and bicontinuous, 
since $\gamma^{-1} \circ \gamma = \gamma \circ \gamma^{-1} = id$.
Part (D) is (\cite{B-R}, Corollary 2.13 (B)).  
\end{proof}

\begin{lemma} \label{PathConvexityLemma}  
For any distinct points $x, y \in \PP^1(K)$, 
the function $\ordRes_{\varphi}(\cdot)$ on type {\rm II} points extends to a continuous function on the path $[x,y]$, 
which is piecewise affine with respect to the logarithmic path distance, and convex up. The extension is finite
on $[x,y] \cap \HH_{\Berk}$, and when $d \ge 2$, it is $\infty$ at $x$ and $y$.  

If $H$ is a field of definition for $\varphi$ {\rm (}so $H(x,y)$ is a field of definition
for $\varphi$, $x$, and $y${\rm )}, then each affine piece of $\ordRes_{\varphi}(\cdot)$ has the form $m t + c$
for some integer $m$ in the range $-d^2 -d \le m \le d^2 + d$ satisfying $m \equiv d^2 + d \pmod{2d}$, 
and some number $c$ in the value group $\ord(H(x,y)^{\times})$, 
where $t$ is a parameter measuring the logarithmic path distance along $[x,y]$.  
There are at most $d+1$ distinct affine pieces, and the breaks between affine pieces occur at type {\rm II} points.  
\end{lemma}

\begin{proof}
Fix $\gamma \in \GL_2(K)$ with $\gamma(0) = x$ and $\gamma(\infty) = y$.  
The action of $\GL_2(K)$ on $\PP^1_{\Berk}$ takes paths to paths, so $\gamma([0,\infty]) = [x,y]$. 
The type II points on $[0,\infty]$ are the points $\zeta_{|A|}$ corresponding to discs $D(0,|A|)$,  
as $A$ runs over elements of $K^{\times}$,    
and if we put $\mu_A = \left[ \! \begin{array}{cc} A & 0 \\ 0 & 1 \end{array} \! \right] \in \GL_2(K)$, 
then $\zeta_{|A|} = \mu_A(\zeta_G)$.  Now let $\gamma_A = \gamma \circ \mu_A$.  As $A$ varies, 
the type II points on $[x,y]$ are the points $\gamma(\zeta_{|A|}) = \gamma_A(\zeta_G)$, 
and for all $A, B \in K^{\times}$ we have 
\begin{equation*}
\rho(\gamma(\zeta_{|A|}),\gamma(\zeta_{|B|})) \ = \ |\,\ord(A) - \ord(B)| \ .
\end{equation*} 

Write 
\begin{eqnarray} 
F^{\gamma}(X,Y) & = & a_d X^d + a_{d-1} X^{d-1} Y + \cdots + a_0 Y^d \ , \label{FGg2} \\ 
G^{\gamma}(X,Y) & = & b_d X^d + b_{d-1} X^{d-1} Y + \cdots + b_0 Y^d \ . \notag
\end{eqnarray} 
Since $\varphi^{\gamma_A} = (\varphi^{\gamma})^{\mu_A}$ we have 
$\left[ \! \begin{array}{c} F^{\gamma_A}(X,Y) \\ G^{\gamma_A}(X,Y) \end{array} \! \right]         
              = \left[ \! \begin{array}{c} F^{\gamma}(AX,Y) \\ A \, G^{\gamma}(AX,Y)  \end{array} \! \right]$; thus  
\begin{eqnarray}
F^{\gamma_A}(X,Y)  & = & A^d a_d X^d + A^{d-1} a_{d-1} X^{d-1} Y + \cdots + a_0 Y^d \ , \label{FGg3} \\
G^{\gamma_A}(X,Y)  & = & A^{d+1} b_d X^d + A^d b_{d-1} X^{d-1} Y + \cdots + A b_0 Y^d \ . \notag
\end{eqnarray}
Put $Q_A = \gamma_A(\zeta_G)$ and write $t = \ord(A)$.  Since $\det(\gamma_A) =  A\det(\gamma)$, it follows from 
formula (\ref{KeyFormula}) that 
\begin{eqnarray} 
& & \ordRes_{\varphi}(Q_A) \ = \ \ordRes(\varphi^{\gamma_A}) \notag \\
           & &  \qquad = \ \ordRes(F^\gamma,G^\gamma) 
                       + \ (d^2 + d) \ord(A) \label{LinearFormula1} \\ 
           & &   \qquad \qquad     
           - 2d \min\big( \ord(a_0), \cdots, \ord(A^d a_d), \ord(A b_0) , \cdots , \ord(A^{d+1} b_d) \big)  \notag  \\
& &  \qquad  = \ \max\Big( \max_{0 \le \ell \le d} \big( (d^2 + d - 2 d \ell ) t + C_\ell \big), 
\max_{0 \le \ell \le d} \big((d^2 + d - 2 d(\ell + 1)) t + D_\ell)\big) \Big)  \ , \label{LinearFormula2}  
\end{eqnarray}
where  
$C_\ell = \ordRes(F^\gamma,G^\gamma) - 2d \, \ord(a_\ell)$,  
$D_\ell = \ordRes(F^\gamma,G^\gamma) - 2d \, \ord(b_\ell)$.

Now let $t$ vary over $\RR$.  Since the type II points $Q_A$ 
(which correspond to values of $t$ in the divisible group $\ord(K^{\times})$) 
are dense in $[x,y]$ for the path distance topology, 
we can use the right side of (\ref{LinearFormula2}) to extend $\ordRes_{\varphi}(\cdot)$ continuously to $[x,y]$, 
omitting any terms in (\ref{LinearFormula2}) for which $C_{\ell}$ or $D_{\ell}$ is $-\infty$ 
(such terms correspond to coefficients $a_\ell$ or $b_\ell$ which are $0$).   
Clearly the extension, being the maximum of finitely many affine functions of $t$, 
is piecewise affine and convex upwards.   Now suppose $d \ge 2$.
Since $F(X,Y)$ and $G(X,Y)$ have no common factors, the same is true for $F^{\gamma}(X,Y)$ and $G^{\gamma}(X,Y)$;  
it follows that at least one of $a_0, b_0$ is nonzero, and at least one of $a_d, b_d$ is nonzero.  
The slopes of the corresponding affine functions are are $d^2 + d$, $d^2-d$, $-(d^2-d)$ and $-(d^2 + d)$;  
since $d \ge 2$ these are all nonzero.  Thus at least one of the affine functions
in (\ref{LinearFormula2}) has positive slope and at least one has negative slope;    
this means the extended function $\ordRes_{\varphi}(\cdot)$
is finite on $[x,y] \cap \HH_{\Berk}$, and is $\infty$ at $x$ and $y$.   

Let $H$ be a field of definition for $\varphi$. 
Then  $F(X,Y)$, $G(X,Y)$  can be taken to be rational over $H$, and $\gamma$ can be taken to be rational over $H(x,y)$; 
if this is the case then $a_0, \ldots, a_d, b_0, \ldots, b_d$ and $\det(\gamma)$
will also be rational over $H(x,y)$.  
Comparing (\ref{LinearFormula1})
and (\ref{LinearFormula2}) we see that each affine piece of $\ordRes_{\varphi}(\cdot)$
has the form $m t + c$, where $m$ is an integer 
in the range $-d^2 -d \le m \le d^2 + d$ satisfying $m \equiv d^2 + d \pmod{2d}$, 
and $c$ belongs to the value group $\ord(H(x,y)^{\times})$. 
If two of the affine functions in (\ref{LinearFormula2}) have the same slope, 
only one will contribute to $\ordRes_{\varphi}(\cdot)$.  
There are $d+1$ possible slopes, so $\ordRes_{\varphi}(\cdot)$
has at most $d+1$ affine pieces on $[x,y]$.  

Finally, suppose $m_i t + c_i$ and $m_j t +c_j$ are consecutive affine pieces.  Their intersection occurs at 
\begin{equation} \label{IntRat}
t \ = \ t_{ij} \ = \ - \frac{c_j - c_i}{m_j - m_i} 
\end{equation} 
which belongs to $\ord(K^{\times})$;  thus the breaks between affine pieces occur at type II points.  
Indeed, $m = m_j - m_i$ is a nonzero integer satisfying $m \equiv 0 \pmod{2d}$, 
with $|m| \le 2d(d + 1)$;  and that by (\ref{LinearFormula1}) and (\ref{LinearFormula2}) 
 $c_j- c_i \in 2d \cdot \ord(H(x,y)^{\times})$.  
Thus $t_{ij}$ actually belongs to the divisible hull of $\ord(H(x,y)^{\times})$, 
with denominator taken from $\{1,2, \ldots, d+1\}$.      
\end{proof}

\begin{proposition} \label{ExtensionCor} 
There is a unique extension of $\ordRes_{\varphi}(\cdot)$ on type {\rm II} points to a function 
$\ordRes_{\varphi} : \PP^1_{\Berk} \rightarrow [0,\infty]$  
which agrees with the one given in {\rm Lemma \ref{PathConvexityLemma}}
on paths with endpoints in $\PP^1(K)$, and is continuous on $\HH_\Berk$ for the strong topology.  
When $d = 1$, the extension is continuous with respect to the strong topology 
at each $x \in \HH_\Berk$,  
and at each $x \in \PP^1(K)$ where $\ordRes_\varphi(x) = \infty$. 
When $d \ge 2$, it is continuous with respect to 
the strong topology at each $x \in \PP^1_{\Berk}$.    
The extension is finite on $\HH_\Berk$,
and when $d \ge 2$ it takes the value $\infty$ at each $x \in \PP^1(K)$.

On each path in $\PP^1_{\Berk}$,  
the extension is convex upwards and piecewise affine with respect to $\rho(x,y);$  
moreover, the slope of each affine piece is 
an integer $m \equiv d^2 + d \pmod{2d}$ with $-d^2 -d \le m \le d^2 + d,$ 
the breaks between affine pieces occur at type {\rm II} points$,$ 
and there are at most $d+1$ distinct affine pieces. 
In particular, on $\HH_\Berk$, the extension is Lipschitz continuous with respect to $\rho(x,y)$   
with Lipschitz constant $d^2+d$. 
\end{proposition}

\begin{proof} 
Given two paths $[x_1,y_1]$, $[x_2,x_2]$ with endpoints in $\PP^1(K)$,  
the extensions of $\ordRes_\varphi(\cdot)$ to 
$[x_1,y_1]$ and $[x_2,x_2]$ given by Lemma \ref{PathConvexityLemma} are consistent on $[x_1,y_1] \cap [x_2,x_2]$, 
since type II points are dense in the intersection if it is nonempty,
and the extension to each path is continuous.  
Define $\ordRes_{\varphi}(\cdot)$
to be the extension given by Lemma \ref{PathConvexityLemma}  
on each path $[x,y]$ with endpoints in $\PP^1(K)$. 
In this way, we obtain a well-defined function $\ordRes_\varphi(\cdot)$
on the points of type I, II, and III in $\PP^1_{\Berk}$. 
When $d \ge 2$, Lemma \ref{PathConvexityLemma}
shows that $\ordRes_\varphi(x) = \infty$ for each $x \in \PP^1(K)$.  

We next show that there is a unique continuous extension of $\ordRes_{\varphi}(\cdot)$ to type IV points.
Since any pair of type II points belongs to a path with endpoints in $\PP^1(K)$,
Lemma \ref{PathConvexityLemma} shows that for all type II points $x, y$ we have 
\begin{equation*} 
|\, \ordRes_{\varphi}(x) - \ordRes_{\varphi}(y)| \ \le \ (d^2+d) \cdot \rho(x,y) \ .
\end{equation*} 
Since each point of IV is at finite logarithmic path distance from $\zeta_G$, 
and type II points are dense in $\HH_{\Berk}$ with respect to $\rho(x,y)$, 
there is a unique extension of $\ordRes_{\varphi}(\cdot)$ to $\HH_{\Berk}$ which is Lipschitz
continuous with respect to $\rho(x,y)$, with Lipschitz constant $d^2+d$.   Since
$\ordRes_{\varphi}(x) \ge 0$ on type II points,  $\ordRes_{\varphi}(z) \ge 0$
for all $z \in \PP^1_\Berk$.  

Since each segment $[u,v]$ with type II endpoints is contained in a path $[x,y]$ with type I endpoints,
the restriction of $\ordRes_\varphi(\cdot)$ to $[u,v]$ is piecewise affine 
and convex upwards with respect to the logarithmic path distance, with most $d+1$ affine pieces, 
and slopes $m \equiv d^2 + d \pmod{2d}$ where $-d^2-d \le m \le d^2 + d$;
the breaks between affine pieces occur at type II points.    
These same properties must hold for $\ordRes_\varphi(\cdot)$ on 
an arbitrary path $[z,w]$ in $\PP^1_\Berk$, 
since the interior of the path can be exhausted by an increasing sequence of segments 
with type II endpoints, and the number of affine pieces on each such segment is uniformly bounded.    
 
To complete the proof, it suffices 
to show that $\ordRes_\varphi(\cdot)$ is continuous with respect to the strong topology 
at each type I point $x$ where $\ordRes_\varphi(x) = \infty$.  Fix $y \in \PP^1(K)$ with $y \ne x$,
and consider the path $[x,y]$.  For each $P \in [x,y] \cap \HH_{\Berk}$, let $U_x(P)$ be the component
of $\PP^1_{\Berk} \backslash \{P\}$ containing $x$.  As $P \rightarrow x$, the 
sets $U_x(P)$ form a basis for the neighborhoods of $x$ in the strong topology.  We claim that
for each $M \in \RR$, there is a $P_M$ such that $\ordRes_{\varphi}(z) > M$ 
for all $z \in U_x(P_M)$.  To see this, note that since $\ordRes_{\varphi}(P)$ increases
to $\infty$ as $P \rightarrow x$ along $[x,y]$, there is a $P_M$
such that $\ordRes_{\varphi}(P_M) > M$ and $\ordRes_{\varphi}(\cdot)$ is increasing on $[P_M,x]$.  
Let $z \in U_x(P_M)$ be arbitrary.  The path  $[P_M,z]$ shares an initial segment with $[P_M,x]$, 
and $\ordRes_{\varphi}(\cdot)$ is increasing along that initial segment.  Since $\ordRes_{\varphi}(\cdot)$ 
is convex up on $[P_M,z]$, we have $\ordRes_{\varphi}(z) > \ordRes_{\varphi}(P_M) > M$.  
\end{proof}

For each $Q \in \PP^1_{\Berk}$, we call paths $[Q,x]$ and $[Q,y]$ emanating from $Q$ {\em equivalent} 
if they share an initial segment.  The {\em tangent space} $T_Q$ is the set of equivalence classes
of paths emanating from $Q$;  these classes are called {\em directions}.  
The directions at $Q$ are in $1-1$ correspondence
with the components of $\PP^1_{\Berk} \backslash \{Q\}$.  
If $Q$ is of type I or IV, $T_Q$ has one element;  if $Q$ is of type III, $T_Q$ has two elements;  
and if $Q$ is of type II, $T_Q$ is infinite.  Given $\beta \ne Q$, 
we will write $\vv_\beta \in T_Q$ for the direction containing $[Q,\beta]$,
or $\vv_{Q,\beta}$ if is necessary to specify $Q$.   

Recall that $\tk = \cO/\fM$ is the residue field of $K$.
When $Q = \zeta_G$, the components of $\PP^1_{\Berk} \backslash \{\zeta_Q\}$ correspond to 
elements of $\PP^1(\tk)$;  thus the directions in $T_{\zeta_G}$ are $\vv_{\infty}$
and the $\vv_{\beta}$ for $\beta \in \cO$,  where $\vv_{\beta_1} = \vv_{\beta_2}$
iff $\beta_1 \equiv \beta_2 \pmod{\fM}$.  For an arbitrary type II point $Q$, we can write 
$Q = \gamma(\zeta_G)$ for some $\gamma \in \GL_2(K)$;  
since $\gamma$ takes paths to paths, it induces a $1-1$ correspondence 
$\gamma_* : T_{\zeta_G} \rightarrow T_Q$ with $\gamma_*(\vv_{\beta}) = \vv_{\gamma(\beta)} \in T_Q$. 
Hence the directions in $T_Q$ are $\vv_{\gamma(\infty)}$ 
and the $\vv_{\gamma(\beta)}$ for $\beta \in \cO$, 
where again $\vv_{\gamma(\beta_1)} = \vv_{\gamma(\beta_2)}$ iff $\beta_1 \equiv \beta_2 \pmod{\fM}$.

\smallskip
We will say $\ordRes_{\varphi}(\cdot)$ 
is {\em locally decreasing} (resp. {\em locally constant}, resp. {\em increasing}) 
in a direction $\vv$ at $Q$  
if it is initially decreasing (resp. constant, resp. increasing) along $[Q,\beta]$ 
for some (hence every) path with $\vv = \vv_{\beta}$.      
A crucial observation is that since $\ordRes_{\varphi}(\cdot)$ is convex upward, 
at each point $Q$ there can be at most one direction in which  $\ordRes_{\varphi}(\cdot)$ is locally decreasing:
thus, $\ordRes_\varphi(\cdot)$ satisfies the principle of steepest descent. 
Likewise, if it is locally constant in some direction at $Q$, it must be locally constant 
or increasing in every other direction.  
If it is locally increasing in some direction at $Q$, by convexity 
it must be increasing along every path $[Q,\beta]$ in that direction, 
so we do not distinguish between {\em locally increasing} and {\em increasing}.  

When $Q$ is of type II, we will now give necessary and sufficient conditions for $\ordRes_{\varphi}(\cdot)$ 
to be locally decreasing, locally constant, or increasing in a given direction.  
Suppose $Q = \gamma(\zeta_G)$ where $\gamma \in \GL_2(K)$;  
let $(F^{\gamma},G^{\gamma})$ be the representation of $\varphi^{\gamma}$ from (\ref{FGg1}).  By replacing
$\gamma$ with $c\gamma$ for an appropriate $c \in K^{\times}$ (which does not change action of $\gamma$) we can 
assume $(F^{\gamma},G^{\gamma})$ is normalized.  As in (\ref{FGg2}), write 
\begin{eqnarray} 
F^{\gamma}(X,Y) & = & a_d X^d + a_{d-1} X^{d-1} Y + \cdots + a_0 Y^d \ , \label{FGg2b} \\ 
G^{\gamma}(X,Y) & = & b_d X^d + b_{d-1} X^{d-1} Y + \cdots + b_0 Y^d \ . \notag
\end{eqnarray} 

For each $\beta \in \cO$, 
the map $\nu^\beta := \left[ \begin{array}{cc} 1 & \beta \\ 0 & 1 \end{array} \right] \in \GL_2(\cO)$
stabilizes $\zeta_G$ and takes the path $[0,\infty]$ to $[\beta,\infty]$.  
Put $\gamma^\beta = \gamma \circ \nu^\beta$;  
then $\gamma^\beta(\zeta_G) = Q$ 
and since $\varphi^{\gamma^\beta} = (\varphi^{\gamma})^{\nu^\beta}$  it follows that 
the pair $(F^{\gamma^\beta},G^{\gamma^\beta})$ given by 
\begin{equation*} 
\left[ \! \begin{array}{c} F^{\gamma^\beta}(X,Y) \\ G^{\gamma^{\beta}}(X,Y) \end{array} \! \right]  = 
         \Adj(\nu^{\beta}) \! \circ \! \left[ \! \begin{array}{c} F^{\gamma} 
                         \\ G^{\gamma} \end{array} \! \right] \! \circ \! \nu^{\beta} \! \circ \! 
                               \left[ \! \begin{array}{c} X \\ Y \end{array} \! \right]
          = \left[ \! \begin{array}{c} F^{\gamma}(X+\beta Y,Y) - \beta G^{\gamma}(X + \beta Y, Y) 
                 \\ G^{\gamma}(X + \beta Y, Y) \! \end{array} \right] 
\end{equation*} 
is another representation of $\varphi$ at $Q$.  It is normalized since $\nu^{\beta} \in \GL_2(\cO)$. 
Write 
\begin{eqnarray} 
F^{\gamma^{\beta}}(X,Y) & = & a_d(\beta) X^d + a_{d-1}(\beta) X^{d-1} Y + \cdots + a_0(\beta) Y^d \ , \label{FGg5} \\ 
G^{\gamma^{\beta}}(X,Y) & = & b_d(\beta) X^d + b_{d-1}(\beta) X^{d-1} Y + \cdots + b_0(\beta) Y^d \ . \notag
\end{eqnarray} 

\begin{lemma} \label{CriterionLemma}
Let $Q$ be a type {\rm II} point;  suppose $Q = \gamma(\zeta_G)$ where $\gamma \in \GL_2(K)$ is such that
$(F^\gamma,G^\gamma)$ is normalized.  Then for each direction $\vv \in T_Q$

\smallskip
\noindent{\quad $(A)$}  $\ordRes_{\varphi}(\cdot)$ is locally decreasing in the direction $\vv$ if and only if 
\begin{equation*}
\text{$\vv = \vv_{Q,\gamma(\infty)}$ and \quad} \left\{ \begin{array}{l} 
             \text{$\ord(a_\ell) > 0$ when $(d+1)/2 \le \ell \le d$ and } \\
             \text{$\ord(b_\ell) > 0$ when $(d-1)/2 \le \ell \le d$\ , } 
         \end{array} \right.
\end{equation*}
\qquad \quad or for some $\beta \in \cO$,  
\begin{equation*}
\text{\quad \ $\vv = \vv_{Q,\gamma(\beta)}$ \ and \quad} 
\left\{ \begin{array}{l} 
             \text{$\ord(a_\ell(\beta)) > 0$ when $0 \le \ell \le (d+1)/2$ and} \\
             \text{$\ord(b_\ell(\beta)) > 0$ when $0 \le \ell \le (d-1)/2$\ . } 
         \end{array} \right.
\end{equation*}

\noindent{\quad $(B)$} $\ordRes_{\varphi}(\cdot)$ 
is locally constant in the direction $\vv$ if and only if $d$ is odd and 
\begin{equation*}
\text{$\vv = \vv_{Q,\gamma(\infty)}$ and \quad} \left\{ \begin{array}{l} 
             \text{$\ord(a_{(d+1)/2}) = 0$ or $\ord(b_{(d-1)/2} ) = 0$, and} \\
             \text{$\ord(a_\ell) > 0$ when $(d+1)/2 < \ell \le d$, and} \\
             \text{$\ord(b_\ell) > 0$ when $(d-1)/2 < \ell \le d$ \ ,} 
         \end{array} \right.
\end{equation*}
\qquad \quad or $d$ is odd and for some $\beta \in \cO$, 
\begin{equation*}
\text{\qquad \quad $\vv = \vv_{Q,\gamma(\beta)}$ \ and \quad} \left\{ \begin{array}{l} 
             \text{$\ord(a_{(d+1)/2}(\beta)) = 0$ or $\ord(b_{(d-1)/2}(\beta)) = 0$, and } \\
             \text{$\ord(a_\ell(\beta)) > 0$ when $0 \le \ell < (d+1)/2$, and} \\
             \text{$\ord(b_\ell(\beta)) > 0$ when $0 \le \ell < (d-1)/2$ \ .} 
         \end{array} \right.
\end{equation*}

\noindent{\quad $(C)$}  $\ordRes_{\varphi}(\cdot)$ is increasing in the direction $\vv$, otherwise.
\end{lemma}

\begin{proof} Note that $\gamma([0,\zeta_G]) = [\gamma(0),Q]$ and $\gamma([\zeta_G, \infty]) = [Q,\gamma(\infty)]$.
We will prove the criteria for the directions $\vv_{\gamma(0)}$ and $\vv_{\gamma(\infty)}$ 
using formula (\ref{LinearFormula2}) and the normalized representation $(F^{\gamma},G^{\gamma})$.   
Since $\gamma^{\beta}([0,\zeta_G]) = [\gamma(\beta),Q]$, 
the criteria for the  directions $\vv_{\gamma(\beta)}$ with arbitrary $\beta \in \cO$   
follow by applying the same arguments to $(F^{\gamma^\beta},G^{\gamma^\beta})$. 
  
Using the same notation as in formulas (\ref{LinearFormula1}) and (\ref{LinearFormula2}),  
for each $A \in K^{\times}$ put $Q_A = \gamma(\zeta_{|A|}) = \gamma_A(\zeta_G)$.  
Making the constants $C_\ell$, $D_\ell$ in formula (\ref{LinearFormula2}) explicit, we have  
\begin{eqnarray}
& & \ordRes_{\varphi}(Q_A) - \ordRes_{\varphi}(Q) \ = \  
 \max\Big( \max_{0 \le \ell \le d} \big( (d^2 + d - 2 d \ell ) t - 2 d \, \ord(a_\ell) \big), \notag \\
& & \qquad \qquad \qquad \qquad \qquad \qquad \qquad \qquad  
\max_{0 \le \ell \le d} \big((d^2 + d - 2 d(\ell + 1)) t -  2d \, \ord(b_\ell)\big) \Big)  \ , \label{LinearFormula3}  
\end{eqnarray}
where $t = \ord(A)$.  
By assumption  some $\ord(a_\ell)$ or $\ord(b_\ell)$ is $0$, and $\ord(a_\ell), \ord(b_\ell) \ge 0$ for each $\ell$.
When $t = 0$ we have $Q_A = Q$ and both sides of (\ref{LinearFormula3}) are $0$.  

Values of $t > 0$ correspond to points in the direction $\vv_{\gamma(0)}$ at $Q$.
For small positive $t$, the right side of (\ref{LinearFormula3}) will be negative if
and only if each of the affine functions in (\ref{LinearFormula3}) with a nonnegative slope 
has a negative constant term.
Hence $\ordRes_{\varphi}(\cdot)$ is locally decreasing
in the direction $\vv_{\gamma(0)} \in T_Q$ if and only if 
$\ord(a_\ell) > 0$ for each $\ell$ such that $d^2 + d - 2 d \ell \ge 0$, 
and $\ord(b_\ell) > 0$ for each $\ell$ such that 
$d^2 + d - 2d(\ell + 1) \ge 0$.   
Similarly $\ordRes_{\varphi}(\cdot)$ is locally constant in the direction $\vv_{\gamma(0)}$ if and only if 
one of the affine functions with slope $0$ has a constant term $0$, and each of the affine functions
with positive slope has negative constant term.  This happens if and only if  
$d$ is odd, either $\ord(a_{(d+1)/2}) = 0$ or $\ord(b_{(d-1)/2}) = 0$,  
$\ord(a_\ell) > 0$ for each $\ell$ such that $d^2 + d - 2 d \ell  > 0$, 
and $\ord(b_\ell) > 0$ for each $\ell$ such that 
$d^2 + d - 2d(\ell + 1) >  0$.  

Values of $t < 0$ correspond to points in the direction $\vv_{\gamma(\infty)}$ at $Q$.
For small negative $t$, the right side of (\ref{LinearFormula3}) will be negative if
and only if each of the affine functions in (\ref{LinearFormula3}) with a nonpositive slope 
has a negative constant term.
Hence $\ordRes_{\varphi}(\cdot)$ is locally decreasing
in the direction $\vv_{\gamma(\infty)}$ if and only if 
$\ord(a_\ell) > 0$ for each $\ell$ such that $d^2 + d - 2 d \ell \le 0$, 
and $\ord(b_\ell) > 0$ for each $\ell$ such that 
$d^2 + d - 2d(\ell + 1) \le 0$.   
Similarly $\ordRes_{\varphi}(\cdot)$ is locally constant in the direction $\vv_{\gamma(\infty)}$ if and only if 
$d$ is odd, either $\ord(a_{(d+1)/2}) = 0$ or $\ord(b_{(d-1)/2}) = 0$,  
$\ord(a_\ell) > 0$ for each $\ell$ such that $d^2 + d - 2 d \ell  < 0$, 
and $\ord(b_\ell) > 0$ for each $\ell$ such that 
$d^2 + d - 2d(\ell + 1) <  0$.  
\end{proof} 

\begin{lemma} \label{TwoDirsCor}
If $d \ge 2$ is even, $\ordRes_{\varphi}(\cdot)$ is never locally constant.  
If $d \ge 3$ is odd,  
then at each $Q \in \PP^1_{\Berk}$, 
there are at most two directions in $T_Q$ where $\ordRes_{\varphi}(\cdot)$ is locally constant. 
\end{lemma} 

\begin{proof}  
If $d \ge 2$ is even, then on any path the slope of each affine piece of 
$\ordRes_\varphi(\cdot)$ is an integer $m \equiv d^2 + d \pmod {2d}$, hence is nonzero.  

Suppose $d \ge 3$ is odd.  
If $Q \in \PP^1_{\Berk}$ is of type I, III, or IV then there are are at most two directions in $T_Q$,
so trivially there are at most two directions in $T_Q$ 
in which $\ordRes_{\varphi}(\cdot)$ is locally constant.   
Let $Q$ be a type II point with at least two distinct directions 
where $\ordRes_{\varphi}(\cdot)$ is locally constant,  say $\vv_\alpha$ and $\vv_\beta$.
Take any $\gamma \in \GL_2(K)$ with $Q = \gamma(\zeta_G)$.  After replacing $\gamma$ with $\gamma \tau$
for a suitable $\tau \in \GL_2(\cO)$,  we can assume that $\vv_{\alpha} = \vv_{\gamma(0)}$ 
and $\vv_{\beta} = \vv_{\gamma(\infty)}$. Also, after replacing $\gamma$ with $c \gamma$ 
for a suitable $c \in K^{\times}$, we can assume that $(F^{\gamma},G^{\gamma})$ 
is a normalized representation of $\varphi$.  
Write $F^{\gamma}(X,Y) = a_d X^d + a_{d-1} X^{d-1} Y + \cdots + a_0 Y^d$,  
$G^{\gamma}(X,Y)  = b_d X^d + b_{d-1} X^{d-1} Y + \cdots + b_0 Y^d$.
By Lemma \ref{CriterionLemma}(B), if we put  
$D = (d+1)/2$ and $E = (d-1)/2$, 
then  $\ord(a_\ell) > 0$ for all $\ell \ne D$, $\ord(b_\ell) > 0$ for all $\ell \ne E$, 
and either $\ord(a_D) = 0$ or $\ord(b_E) = 0$.  Since $d \ge 3$, we have $D, E \ge 1$.

First suppose $\ord(b_E) = 0$;  then $G^{\gamma}(X,Y) \equiv b_E X^E Y^{d-E} \pmod{\fM}$,   
so for each $\beta \in \cO$  
\begin{equation*} 
G^{\gamma^\beta}(X,Y) \ := \ G^{\gamma}(X+\beta Y,Y) \ \equiv \  b_E(X+\beta)^E Y^{d-E} \!\!\! \pmod{\fM} \ .
\end{equation*}
Comparing this with (\ref{FGg5}) shows $b_0(\beta) \equiv b_E \beta^E \pmod{\fM}$.  
If $\beta \not\equiv 0 \ \pmod{\fM}$, this means $\ord(b_0(\beta)) = 0$, so the criterion
in Lemma \ref{CriterionLemma}(B) is not met for the direction $\vv_{\gamma(\beta)}$.  Thus 
$\vv_{\gamma(0)}$ and $\vv_{\gamma(\infty)}$ are the only directions 
in which $\ordRes_{\varphi}(\cdot)$ is locally constant.  

Next suppose $\ord(b_E) > 0$, so necessarily $\ord(a_D) = 0$.  Then $G^{\gamma}(X,Y) \equiv 0 \pmod{\fM}$
and $F^{\gamma}(X,Y) \equiv a_D X^D Y^{d-D} \pmod{\fM}$,  so for each $\beta \in \cO$ 
\begin{equation*} 
F^{\gamma^\beta}(X,Y) \, := \, F^{\gamma}(X+\beta Y,Y) - \beta G^{\gamma}(X+ \beta Y,Y)
 \, \equiv \,  b_D(X+\beta)^D Y^{d-D} \!\!\! \pmod{\fM} \ .
\end{equation*}
Comparing this with (\ref{FGg5}) shows $a_0(\beta) \equiv a_D \beta^D \pmod{\fM}$.  
When $\beta \not\equiv 0 \ \pmod{\fM}$, this means $\ord(a_0(\beta)) = 0$, so the criterion
in Lemma \ref{CriterionLemma}(B) is not met for the direction $\vv_{\gamma(\beta)}$, and again 
$\vv_{\gamma(0)}$ and $\vv_{\gamma(\infty)}$ are the only directions 
in which $\ordRes_{\varphi}(\cdot)$ can be locally constant.  
\end{proof}

\noindent{\bf Remark.} Using a similar argument, one can show that at any type II point 
there can be at most one direction in which $\ordRes_{\varphi}(\cdot)$ is locally decreasing,
without appealing to convexity.

\medskip
Our next goal is to show that $\ordRes_{\varphi}(\cdot)$ is strictly increasing as one moves away from the tree
$\Gamma_{\Fix,\varphi^{-1}(\infty)}$ in $\PP^1_{\Berk}$ spanned by the fixed points and the poles of $\varphi$.
This means that $\ordRes_{\varphi}(\cdot)$ achieves a minimum on $\PP^1_{\Berk}$, 
and shows that the locus $\MinRes(\varphi)$ where it takes on its minimum is contained in that tree. 

Two main facts underlie this.  
The first is that the group of affine transformations 
$\Aff_2(K) = \{a z + b: a \in K^{\times}, b \in K\}$, corresponding to matrices 
$\left[ \begin{array}{cc} a & b \\ 0 & 1 \end{array} \right] \in \GL_2(K)$, acts transitively 
on type II points.  Indeed, if $Q$ corresponds to a disc $D(b,r)$ with $r \in |K^{\times}|$, 
and $|a| = r$, then $\gamma(z) = az+b$ takes $\zeta_G$ to $Q$.  The second is that the fixed points  of $\varphi$ are 
equivariant under $\GL_2(K)$, and the poles are equivariant under $\Aff_2(K)$: for each $\gamma \in \GL_2(K)$, 
$\Delta$ is a fixed point of $\varphi$ iff $\gamma^{-1}(\Delta)$ is a fixed point of $\varphi^{\gamma}$; 
and for each $\gamma \in \Aff_2(K)$, $\delta$ is a pole of $\varphi$ iff $\gamma^{-1}(\delta)$ 
is a pole of $\varphi^{\gamma}$.   

\begin{lemma} \label{FixPoleSet} 
If $d \ge 2$, the set of poles and fixed points of $\varphi$ in $\PP^1(K)$ contains at least
two distinct elements.
\end{lemma} 

\begin{proof}  The fixed points of $\varphi$ correspond to solutions of $\varphi(z) = z$ in $\PP^1(K)$.  
Using the representation $(F(X,Y),G(X,Y))$ for $\varphi(z)$, we obtain the representation $( YF(X,Y) - XG(X,Y), G(X,Y))$ 
for $\varphi(z) - z$.  

Suppose all the poles and fixed points of $\varphi$ occur at a single point $\alpha \in \PP^1(K)$. 
If $\alpha = \infty$, there are $C, D \in K^{\times}$ such that $G(X,Y) = CY^d$ and $YF(X,Y) - XG(X,Y) = D Y^{d+1}$.
Solving, we see that  $Y F(X,Y) = D Y^{d+1} + C X Y^d$.  Since $d \ge 2$,
this contradicts that $F(X,Y)$ and $G(X,Y)$ have no common factors.  If $\alpha \in K$, there are $C, D \in K^{\times}$ 
such that $G(X,Y) = C (X-\alpha Y)^d$ and $Y F(X,Y) - XG(X,Y) = D(X - \alpha Y)^{d+1}$.  
In this case $Y F(X,Y) = D(X - \alpha Y)^{d+1} + C X (X - \alpha Y)^d$, which again contradicts 
that $F(X,Y)$ and $G(X,Y)$ have no common factors.
\end{proof} 

\begin{proposition} \label{IncreasingProp}
If $d \ge 2$, the function $\ordRes_{\varphi}(\cdot)$ 
is strictly increasing as one moves away from the tree $\Gamma_{\Fix,\varphi^{-1}(\infty)}$ 
in $\PP^1_{\Berk}$
spanned by the fixed points and poles of $\varphi$ in $\PP^1(K)$.
\end{proposition} 

\begin{proof} Let $\Gamma = \Gamma_{\Fix,\infty}(\varphi)$ 
be the tree spanned by the fixed points and poles of $\varphi$.  
Branches off $\Gamma$ in $\PP^1_{\Berk}$ can only occur at type II points.  
By the convexity of $\ordRes_{\varphi}(\cdot)$, it suffices to show that at each type II point $Q \in \Gamma$, 
$\ordRes_{\varphi}(\cdot)$ is increasing in each direction $\vv \in T_Q$ which points away from $\Gamma$.

Fix a type II point $Q \in \Gamma$, and let $\vv \in T_Q$ be a direction away from $\Gamma$. 
Let $\gamma \in \Aff_2(K)$ be such that $\gamma(\zeta_G) = Q$.  
If $\vv = \vv_{Q,\infty}$, then $\gamma_*(\vv_{\zeta_G,\infty}) = \vv$. 
If $\vv \ne \vv_{Q,\infty}$, there is some $\beta \in \cO$ such that $\gamma_*(\vv_{\zeta_G,\beta}) = \vv$,
and after replacing $\gamma$ with $\gamma^{\beta} = \gamma \circ \nu^{\beta}$ we can assume that  
$\gamma_*(\vv_{\zeta_G,0}) = \vv$.  Finally, by replacing $\gamma$ with $c \gamma$ for some $c \in K^{\times}$, 
we can assume that the representation $(F^{\gamma},G^{\gamma})$ of $\varphi^{\gamma}$ is normalized.  

First suppose $\vv = \gamma_*(\vv_{\zeta_G,\infty}) = \vv_{\gamma(\infty)}$.
By the equivariance of poles and fixed points under $\Aff_2(K)$, 
$\varphi^{\gamma}$ has no poles or fixed points in the direction 
$\vv_{\infty}$ at $\zeta_G$.  As in (\ref{FGg2}), write 
$F^{\gamma}(X,Y) = a_d X^d + a_{d-1} X^{d-1} Y + \cdots + a_0 Y^d$, 
$G^{\gamma}(X,Y) = b_d X^d + b_{d-1} X^{d-1} Y + \cdots + b_0 Y^d$. 
By hypothesis the poles $\delta_i$ of $\varphi^{\gamma}$
all belong to $\cO$, so we can factor $G^{\gamma}(X,Y) = b_d \cdot \prod_{i=1}^{d} (X - \delta_i Y)$
where $|\delta_i| \le 1$ for each $i$.  
Expanding this and comparing coefficients shows that $\max(|b_d|, \ldots, |b_0|) = |b_d|$.  
Likewise, the fixed points $\Delta_i$ of $\varphi^{\gamma}$ all belong to $\cO$.
Since the fixed points are the zeros of 
\begin{equation*} 
YF^{\gamma}(X,Y) - XG^{\gamma}(X,Y) \ = \ a_d X^{d+1} + (a_{d-1}-b_d) X^d Y + \cdots + (a_0 - b_1) XY^d - b_0 Y^{d+1} \ , 
\end{equation*}
we can write $X F^{\gamma}(X,Y) - X G^{\gamma}(X,Y) = a_d \prod_{i=1}^{d+1} (X - \Delta_i Y)$. 
Expanding this and comparing coefficients shows that $\max(|a_d|, |a_{d-1}-b_d|, \cdots, |a_0-b_1|, |b_0|) = |a_d|$. 
However, it is an easy consequence of the ultrametric inequality that
\begin{eqnarray} 
& & \max(|a_d|, |a_{d-1}-b_d|, \cdots, |a_0-b_1|, |b_0|, \  |b_d|, |b_{d_1}|, \cdots, |b_0|) \notag \\
& & \qquad \qquad \qquad \qquad 
           \ = \ \max(|a_d|, |a_{d-1}|, \cdots, |a_0|, \  |b_d|, |b_{d-1}|, \cdots, |b_0|) \ . \label{BasicEq} 
\end{eqnarray}
Thus $\max(|a_d|, |a_{d-1}|, \cdots, |a_0|, \  |b_d|, |b_{d-1}|, \cdots, |b_0|) = \max(|a_d|,|b_d|)$. 
Since $(F^{\gamma},G^{\gamma})$ is normalized, it follows that $\ord(a_d) = 0$ or $\ord(b_d) = 0$.  
By Lemma \ref{CriterionLemma}, 
$\ordRes_{\varphi}(\cdot)$ cannot be decreasing or constant in the direction $\vv = \vv_{\gamma(\infty)}$, 
so it must be increasing.   

Next suppose $\vv = \gamma_*(\vv_{\zeta_G,0}) = \vv_{\gamma(0)}$.
In this case  $\varphi^{\gamma}$ has no poles or fixed points in the direction $\vv_0$ at $\zeta_G$. 
As before, write $F^{\gamma}(X,Y) = a_d X^d + a_{d-1} X^{d-1} Y + \cdots + a_0 Y^d$, 
$G^{\gamma}(X,Y) = b_d X^d + b_{d-1} X^{d-1} Y + \cdots + b_0 Y^d$. 
By hypothesis the poles of $\varphi^{\gamma}$ belong to $(K \backslash \fM) \cup \{\infty\}$, 
so we can factor $G^{\gamma}(X,Y) = C Y^m \cdot \prod_{i=1}^{d-m} (X - \delta_i Y)$  
for some $C \in K^{\times}$, where $m$ is the number of poles of $\varphi^{\gamma}$ at $\infty$  
and $|\delta_i| \ge 1$ for $i = 1, \ldots, d-m$.  
Expanding and comparing coefficients shows that $|b_0| = \max(|b_d|, \ldots, |b_0|)$.  
Likewise, the fixed points of $\varphi^{\gamma}$ 
all belong to $(K \backslash \fM) \cup \{\infty\}$, 
so for some $D \in K^{\times}$ we can write 
$X F^{\gamma}(X,Y) - X G^{\gamma}(X,Y) = D \cdot Y^n  \prod_{i=1}^{d-n} (X - \Delta_i Y)$
where $n$ is the number of fixed points of $\varphi^{\gamma}$ at $\infty$, 
and $|\Delta_i| \ge 1$ for $i = 1, \ldots, d-n$. Expanding and
comparing coefficients shows that $|b_0| = \max(|a_d|, |a_{d-1}-b_d|, \cdots, |a_0-b_1|, |b_0|)$. 
Using (\ref{BasicEq}) we see that  
$|b_0| = \max(|a_d|, |a_{d-1}|, \cdots, |a_0|, \  |b_d|, |b_{d-1}|, \cdots, |b_0|)$. 
Since $(F^{\gamma},G^{\gamma})$ is normalized, it must be that $\ord(b_0) = 0$.  
By Lemma \ref{CriterionLemma}, $\ordRes_{\varphi}(\cdot)$ cannot be 
locally decreasing or constant in the direction $\vv = \vv_{\gamma(0)}$, 
so it must be increasing.
\end{proof} 

\begin{proposition} \label{BoundProp} Suppose $d \ge 2$.
Given a point $x \in \PP^1(K)$,  
let $\xi$ be the unique point in $[\zeta_G,x]$ 
such that $\rho(\zeta_G,x) = \frac{2}{d-1} \ordRes(\varphi)$.  
Then $\ordRes_\varphi(\cdot)$ is increasing along $[\xi,x]$
as one moves away from $\xi$. 
\end{proposition} 

\begin{proof} Choose a point $y \in \PP^1(K)$ lying in a different direction at $\zeta_G$ than $x$.  
Then there is a $\gamma \in \GL_2(\cO)$ such that $\gamma(0) = x$, $\gamma(\infty) = y$, 
and $\gamma([0,\infty]) = [x,y]$. 
Since $\gamma \in \GL_2(\cO)$, it fixes $\zeta_G$.  
Thus $\ordRes_{\varphi^{\gamma}}(\zeta_G) = \ordRes(\varphi^{\gamma}) = \ordRes(\varphi)$. 
Let $(F^{\gamma},G^{\gamma})$ be a representation 
of $\varphi^{\gamma}$, as in (\ref{FGg2});  
after scaling $(F^{\gamma}, G^{\gamma})$ we can assume it is normalized.
At least one of the coefficients $a_0, b_0$ in $F^{\gamma}, G^{\gamma}$ must be nonzero.
Expanding the determinant (\ref{FGResultant}) 
for $\Res(F^{\gamma},G^{\gamma})$  
using its last column, one sees that  $\min(\ord(a_0), \ord(b_0)) \le \ordRes(\varphi)$. 
Similarly, $\min(\ord(a_d), \ord(b_d)) \le \ordRes(\varphi)$.  
 
Given $A \in K^{\times}$, put $Q_A = \zeta_{D(0,|A|)}$ and let $(F^{\gamma_A},G^{\gamma_A})$ 
be as in (\ref{FGg3}).  By (\ref{LinearFormula1}), (\ref{LinearFormula2}) and the discussion above,
\begin{eqnarray}  
\ordRes_{\varphi^{\gamma}}(Q_A) \!\!\!\! &-&\!\!\!\!  \ordRes(\varphi) 
     \ =  \ (d^2 + d) \ord(A) \notag \\       
      & & \qquad 
       - 2d \min\big( \ord(a_0), \cdots, \ord(A^d a_d), \ord(A b_0) , \cdots , \ord(A^{d+1} b_d) \big) 
                  \notag  \\
     & \ge & \max\big(-2d \,\ord(a_0) + (d^2 + d) \ord(A), -2d \,\ord(b_0) + (d^2-d) \ord(A), \notag \\
     &  & \qquad -2d \, \ord(a_d) + (d-d^2) \ord(A) , -2d \, \ord(b_d) + (-d - d^2) \ord(A)\big) \notag \\
     & \ge & -2d \, \ordRes(\varphi) + \max\big( (d^2-d) \ord(A), (d-d^2) \ord(A)\big).
     \label{LinearFormula3b}
\end{eqnarray}
Since $\ordRes_{\varphi^{\gamma}}(\zeta_G) - \ordRes(\varphi) = 0$, 
the minimum value of $\ordRes_{\varphi^{\gamma}}(\cdot) -\ordRes(\varphi)$ on $[0,\infty]$ 
is nonpositive.  Since $d \ge 2$, we have $d^2 - d > 0$;  
hence the right side of (\ref{LinearFormula3b}) is nonpositive precisely when
\begin{equation} 
-\frac{2}{d-1} \ordRes(\varphi) \ \le \ \ord(A) \ \le \ \frac{2}{d-1} \ordRes(\varphi) \ . \label{FR1}
\end{equation}
By convexity, $\ordRes_{\varphi^{\gamma}}(Q_A)$ must be increasing
with $\ord(A)$ for $\ord(A) > \frac{2}{d-1} \ordRes(\varphi)$. 
Since $\ordRes_{\varphi}(\gamma(Q_A)) = \ordRes_{\varphi^{\gamma}}(Q_A)$, the Proposition follows.    
\end{proof} 

\begin{proof}[Proof of Theorem \ref{MainThm}]
Assume $d \ge 2$.
By Proposition \ref{ExtensionCor}, the function $\ordRes_{\varphi}(\cdot)$ on type II points extends to a 
function $\ordRes_{\varphi} : \PP^1_{\Berk} \rightarrow [0,\infty]$ 
which is continuous with respect to the strong topology, 
finite on $\HH_{\Berk}$ and $\infty$ on $\PP^1(K)$,
and piecewise affine and convex upwards with respect to $\rho(x,y)$ on each path. 
By Lemma \ref{FixPoleSet}, the tree $\Gamma = \Gamma_{\Fix,\infty}(\varphi)$ 
spanned by the poles and fixed points of $\varphi$ is nontrivial, and by Proposition \ref{IncreasingProp}, 
$\ordRes_{\varphi}(\cdot)$ is strictly increasing as one moves away from $\Gamma$.  
It follows that $\ordRes_{\varphi}(\cdot)$ takes on a minimum value on $\PP^1_{\Berk}$,
and that the set $\MinRes(\varphi)$ where the minimum is achieved is a 
compact connected subset of $\Gamma \cap \HH_{\Berk}$.  

On any path the slopes of $\ordRes_{\varphi}(\cdot)$ are integers 
$m \equiv d^2 + d \pmod{2d}$.  If $d$ is even, then $d^2 + d \equiv d \pmod{2d}$, 
so none of the slopes are $0$.
Since the breaks between affine pieces occur at type II points, 
$\MinRes(\varphi)$ consists of a single type II point.  
If $d$ is odd, then $d^2 + d \equiv 0 \pmod{2d}$.
By Lemma \ref{TwoDirsCor}, at each $Q \in \MinRes(\varphi)$ there are at most two 
directions where $\ordRes_{\varphi}(\cdot)$ is constant.  
Thus $\MinRes(\varphi)$ is either a single type II point, or a segment with type II endpoints.  

We next show show that 
$\MinRes(\varphi) \subseteq \{z \in \HH_\Berk : \rho(\zeta_G,z) \le \frac{2}{d-1} \ordRes(\varphi) \}$.  
Fix $z$ with $\rho(\zeta_G,z) > \frac{2}{d-1} \ordRes(\varphi)$.  Let $\xi$ be the unique point on $[\zeta_G,z]$ 
with $\rho(\zeta_G,\xi) = \frac{2}{d-1} \ordRes(\varphi)$;  then $\xi$ is of type II.  Let $x \in \PP^1(K)$ be a  
type I point whose direction from $\xi$ is the same as that of $z$.  By Proposition \ref{BoundProp},
$\ordRes_{\varphi}(\cdot)$ is increasing along $[\xi,x]$.  Since $[\xi,z]$ and $[\xi,x]$ share 
an initial segment, by convexity $\ordRes_{\varphi}(\cdot)$ is increasing along  $[\xi,z]$.
Thus $z \notin \MinRes(\varphi)$. 

The final assertion in Theorem \ref{MainThm} reformulates of a result of 
Favre and Rivera-Letelier (\cite{FRLErgodic}, Theorem E).  Suppose the minimal value of
$\ordRes_\varphi(\xi)$ is $0$.  By what has been shown above,  
there is a type II point $\xi \in \MinRes(\varphi)$  where $\ordRes_\varphi(\xi) = 0$.
Let $\gamma \in \GL_2(K)$ be such that $\gamma(\zeta_G) = \xi$.  Then $\ordRes(\varphi^\gamma) = 0$,
so $\varphi^{\gamma}$ has good reduction.  
Since $d \ge 2$, by (\cite{FRLErgodic}, Theorem E, or \cite{B-R}, Proposition 10.5), 
$\xi$ is the unique point where $\varphi$ achieves good reduction.  
Thus $\MinRes(\varphi) = \{\xi\}$.
\end{proof} 

\begin{proof}[Proof of Theorem \ref{StabilityThm}] 
Since $\ordRes_\varphi(\cdot)$ and $\ordRes_\tvarphi(\cdot)$ are continuous for the strong topology, 
to prove the first assertion  
it suffices to show that if (\ref{Congruence}) holds then 
$\ordRes_\varphi(\xi) = \ordRes_\tvarphi(\xi)$ for all type II points $\xi$
with $\rho(\zeta_G,\xi) \le M$.  If $\xi = \zeta_{D(a,r)}$ then the path from $\zeta_G$ to $\xi$ goes 
from $\zeta_G$ to $\zeta_{D(0,T)}$ where $T = \max(1,|a|)$, 
and then from $\zeta_{D(0,T)} = \zeta_{D(a,T)}$ to $\xi$.  Hence if $A, B \in K^{\times}$ are such that
$|A| = 1/T$ and $|B| = r/T$, then $\rho(\zeta_G,\xi) = \ord(A \cdot B)$ and $\xi = \gamma(\zeta_G)$, where 
\begin{equation*}
\gamma \ = \ \left[ \begin{array}{cc} 1 & 0 \\ 0 & A \end{array} \right] \cdot 
                 \left[ \begin{array}{cc} B & aA \\ 0 & 1 \end{array} \right]
          \ = \ \left[ \begin{array}{cc} B & aA \\ 0 & A \end{array} \right] \in \ \GL_2(K) \cap M_2(\cO) \ .
\end{equation*} 
By (\ref{KeyFormula}) we have  
\begin{equation} \label{NextFormula}
\ordRes_\varphi(\xi) \ = \ 
\ordRes(F,G) + (d^2 + d)(\ord(A \cdot B)) - 2d \min(\ord(F^{\gamma}), \ord(G^{\gamma}))  
\end{equation}   
where $F^\gamma$ and $G^\gamma$ are given by (\ref{FGg1}),
and an analogous formula holds for $\ordRes_\tvarphi(\xi)$. 
Since $\ordRes_\varphi(\xi) \ge 0$, $\ordRes(F,G) = R$, and $\ord(A \cdot B) \le M$, 
we conclude from (\ref{NextFormula}) that
\begin{equation*}
\min(\ord(F^{\gamma}), \ord(G^{\gamma})) \ \le \ \frac{1}{2d}\big(R + (d^2+d) M\big) \ .
\end{equation*} 
Now (\ref{Congruence}) gives  
$\min(\ord(F^{\gamma}), \ord(G^{\gamma})) = \min(\ord(\tF^{\gamma}), \ord(\tG^{\gamma}))$
and $\ordRes(F,G) = \ordRes(\tF,\tG) = R$.
Hence $\ordRes_\varphi(\xi) = \ordRes_\tvarphi(\xi)$.

The second assertion follows by taking $M = \frac{2}{d-1} R + \varepsilon$ in (\ref{Congruence}),
with $\varepsilon > 0 $ small, and using Theorem \ref{MainThm}.
\end{proof}

\section{Examples} \label{ExamplesSection}  

Throughout this section, we write $\CC_p$ for the completion of the algebraic closure of $\QQ_p$.
The valuation $\ord(\cdot)$ on $\CC_p$ will be  normalized so that $\ord(p) = 1$, 
and $| \cdot |_p = p^{-\ord(\cdot)}$ is the usual absolute value on $\CC_p$.
We write $\Res(\varphi)$ for $\Res(F,G)$, where $(F,G)$ is the obvious homogenization of
the pair of polynomials defining $\varphi$. 

\smallskip
We first give two examples where $\varphi(z)$ has potential good reduction.

\smallskip
\noindent{\bf Example 2.1}. \label{Example1}
The function $\varphi(z) = \displaystyle{\frac{z^d - p}{z^{d-1}}}$\,, 
with $p$ arbitrary and $K = \CC_p$.

\smallskip
Here $\Res(\varphi) = (-1)^{d(d-1)/2}  p^{d-1}$, so $\ordRes(\varphi) = d-1$.
The poles of $\varphi$ are $0$ and $\infty$, 
and there is a $(d+1)$-fold fixed point at $\infty$.
The tree $\Gamma$ spanned the fixed points and poles is just the path $[0,\infty]$.  
   
Consider $\ordRes_\varphi(\cdot)$ on $[0,\infty]$.  
Let $Q_A \in \PP^1_{\Berk}$  correspond to $D(0,|A|)$; by (\ref{LinearFormula2})  
\begin{eqnarray*} 
\ordRes_\varphi(Q_A) & = & (d-1) + (d^2 + d) \ord(A)
      - 2d \, \min\big(\ord(A^d), \ord(p), \ord(A \! \cdot \! A^{d-1}) \big) \\
          & = & \max\big((d-1)+(d-d^2) \ord(A), (-d-1) + (d^2 + d) \ord(A)) \ .  
\end{eqnarray*} 
This achieves its minimum when $\ord(A) = 1/d$, 
 and $\ordRes_{\varphi}(\zeta_{D(0,p^{1/d})}) = 0$.  

Thus $\varphi(z)$ has potential good reduction at the point $\zeta_{D(0,p^{1/d})}$, 
and conjugation by 
$\gamma  = \left[ \begin{array}{cc} p^{1/d} & 0 \\ 0 & 1 \end{array} \right]$  
achieves the necessary change of coordinates: indeed  
\begin{equation*}
\varphi^{\gamma}(z) \ = \ \frac{z^d-1}{z^{d-1}} \ .
\end{equation*} 
Here $\rho(\zeta_G,\zeta_{D(0,p^{1/d})}) = 1/d < \frac{2}{d-1} \ordRes(\varphi) = 2$.
Note also that $\ordRes_{\varphi}(\zeta_G) = d-1$, 
and $\ordRes_{\varphi}(\zeta_{D(0,p)}) = d^2-1$. 

\smallskip
\noindent{\bf Example 2.2}. \label{Example2} The function $\varphi(z) = \displaystyle{\frac{z^2 - 1}{2z}}$\,, 
with $K = \CC_2$.

\smallskip
Here $\Res(\varphi) = -4$, so $\ordRes(\varphi) = 2$.
The poles of $\varphi$ are $0$ and $\infty$, 
and the fixed points are $\infty$ and $\pm i$, where $i = \sqrt{-1}$.
If we write  $\zeta_{D(a,r)}$ for the point in $\PP^1_{\Berk}$ corresponding to 
the disc $D(a,r) \subset K$, 
then the tree $\Gamma$ spanned by $\{0,\infty, i, -i\}$ has branch points at  $\zeta_G = \zeta_{D(0,1)}$ 
and $\zeta_{D(i,1/2)}$.  
    
First consider $\ordRes_\varphi(\cdot)$ on the path $[0,\infty]$.  
Let $Q_A = \zeta_{D(0,|A|)} \in \PP^1_{\Berk}$; then  
\begin{equation*} 
\ordRes_\varphi(Q_A) \ = \ \max\big(2 - 2 \ord(A), 2 + 6 \ord(A)\big) \ .  
\end{equation*} 
This takes on its minimum when $\ord(A) = 0$, where $\ordRes_{\varphi}(Q_A) = 2$ 
and $Q_A = \zeta_G$. 

Next consider $\ordRes_\varphi(\cdot)$ on the path $[i,\infty]$.  
Let $\gamma = \left[ \begin{array}{cc} 1 & i \\ 0 & 1 \end{array} \right]$, 
so $\gamma(0) = i$ and $\gamma(\infty) = \infty$.  Then 
\begin{equation*}
\varphi^{\gamma}(z) \ = \ \frac{(z+i)^2  - 1}{2(z+i)} - i \ = \ \frac{z^2 - 4iz}{2z + 2i} \ .
\end{equation*} 
Let $Q_A$ be the point corresponding to the disc $D(i,|A|)$;  then 
\begin{equation*}
\ordRes_{\varphi}(Q_A) \ = \ \max\big(2 - 2 \ord(A), -2 + 2 \ord(A)\big) \ .     
\end{equation*} 
This achieves its minimum when $\ord(A) = 1$, and $\ordRes_{\varphi}(Q_A) = 0$. 
The corresponding point $Q_A$ is $\zeta_{D(i,1/2)}$;
note that $\rho(\zeta_G,\zeta_{D(i,1/2)}) = 1 < \frac{2}{d-1} \ordRes(\varphi) = 4$. 

          Thus $\varphi(z)$ has potential good reduction at $\zeta_{D(i,1/2)}$, 
and the map $\eta = \gamma \circ \nu^2  = \left[ \begin{array}{cc} 2 & i \\ 0 & 1 \end{array} \right]$  
achieves the necessary change of coordinates.  One sees that 
\begin{equation*}
\varphi^{\eta}(z) \ = \ \frac{z^2-2iz}{2z+i} 
\end{equation*} 
indeed has good reduction. 
The nearest point to $\zeta_{D(i,1/2)}$ in the tree spanned by $\PP^1(\QQ_2)$ is $\zeta_{D(1,1/\sqrt{2})}$;
one sees that $\ordRes_\varphi(\zeta_{D(1,1/\sqrt{2})}) = 1$.  The nearest points in that tree with radii 
belonging to the value group $|\QQ_2^{\times}|$ 
(we call such points $\QQ_2$-rational type II points) 
are $\zeta_G$ and $\zeta_{D(1,1/2)}$; 
one has $\ordRes_\varphi(\zeta_G) = 2$ and  $\ordRes_\varphi(\zeta_{D(1,1/2)}) = 4$.
 
\smallskip
We next give several examples where $\varphi$ does not have potential good reduction.

\smallskip
\noindent{\bf Example 2.3}. \label{Example3} The function 
$\varphi(z) = \displaystyle{\frac{z^p-z}{p}}$\,, 
with $K = \CC_p$ for an arbitrary prime $p$. 

It is known (see \cite{B-R}, Example 10.120) 
that the Berkovich Julia set of $\varphi(z)$ is contained in $\PP^1(\CC_p)$
(indeed, it is precisely $\ZZ_p$),
and its invariant measure $\mu_\varphi$ is the additive Haar measure on $\ZZ_p$. 
Thus, $\varphi(z)$ cannot have potential good reduction; 
if it did, its Berkovich Julia set would be the unique point $Q \in \HH_{\Berk}$ 
where it attained good reduction.  
Below we will give a direct proof that $\varphi$ does not have potential good reduction.

Here $d = p$, and $\Res(\varphi) = p^p$,  so $\ordRes(\varphi) = p$.
The fixed points of $\varphi(z)$ are $\infty$ and the solutions $u_0, u_1, \cdots, u_{p-1}$ to $z^p - (1+p)z = 0$.   
Since $z^p - (1+p)z \equiv z^p - z \equiv z(z-1) \cdots (z-(p-1)) \pmod{p}$, 
Hensel's Lemma shows that each $u_i$ belongs to $\ZZ_p$, and we can label the $u_i$ so 
that $u_0 = 0$ and $u_i \equiv i \pmod{p}$ for $i = 1, \ldots, p-1$. 
The poles of $\varphi(z)$ are all at $\infty$.  The tree $\Gamma$ spanned by 
$\{\infty, 0, u_1, \ldots, u_{p-1} \}$ has $\zeta_G$ as its only branch point.

First consider $\ordRes_{\varphi}(\cdot)$ on the path $[0,\infty]$.  
As before, write $Q_A$ for  $\zeta_{D(0,|A|)}$; by (\ref{LinearFormula2}) 
\begin{eqnarray*} 
\ordRes_\varphi(Q_A) & = & p + (p^2 + p) \ord(A)
      - 2p \, \min\big(\ord(A^p), \ord(A), \ord(pA) \big) \\
          & = & \max\big(p+(p-p^2) \ord(A), p + (p^2 - p) \ord(A)) \ .  
\end{eqnarray*} 
The minimum is achieved when $\ord(A) = 0$, corresponding to $\ordRes_{\varphi}(\zeta_G) = p$.  

Next fix $i$ with $1 \le i \le p-1$, and consider $\ordRes_{\varphi}(\cdot)$ on the path $[u_i,\infty]$. 
Taking $\gamma  = \left[ \begin{array}{cc} 1 & u_i \\ 0 & 1 \end{array} \right]$, we see that 
\begin{equation*}
\varphi^\gamma(z) \  = \ \frac{(z+u_i)^p - (z+u_i) - pu_i}{p}
\ = \ \frac{a_p z^p + a_{p-1}z^{p-1} + \cdots + a_1 z}{p}
\end{equation*} 
where $a_p = 1$, $a_j = {\scriptsize \left( \! \begin{array}{c} p \\ j \end{array} \! \right) } u_i^{p-j}$ 
for $j = 2, \ldots, p-1$,
and $a_1 = p u_i^{p-1} - 1$.  In particular $\ord(a_p) = \ord(a_1) = 0$, and $\ord(a_j) = 1$ for $j =2, \ldots, p-1$.
By (\ref{LinearFormula2}) 
\begin{eqnarray*} 
\ordRes_\varphi(\gamma(Q_A)) & = & \ordRes_{\varphi^\gamma}(Q_A) \\ 
      & = & p + (p^2 + p) \ord(A) - 2p \, \min\big(\ord(a_p A^p), \cdots, \ord(a_1 A), \ord(pA) \big) \\
      & = & \max\big(p+(p-p^2) \ord(A), p + (p^2 - p) \ord(A)) \ .  
\end{eqnarray*} 
Again the minimum is achieved when $\ord(A) = 0$, corresponding to $\ordRes_{\varphi}(\zeta_G) = p$.
Thus $\MinRes(\varphi) = \{\zeta_G\}$, and $\varphi(z)$ does not have potential good reduction.
Here $\rho(\zeta_G,\zeta_G) = 0 < \frac{2}{d-1} \ordRes(\varphi) = 2p/(p-1)$.
Note that $\varphi(\zeta_G) = \zeta_{D(0,p)}$, so $\MinRes(\varphi)$ is not fixed by $\varphi$. 
  
\smallskip
\noindent{\bf Example 2.4}. \label{Example4} The function 
$\varphi(z) = \displaystyle{\frac{p^4 z^3  + pz + 1}{p^6 z^3}}$\,, 
where $K = \CC_p$\,, and $p$ is odd. 

Here $\Res(\varphi) = p^{18}$.  The function $\varphi(z)$ has a triple pole at $0$, and its fixed points 
are the roots of  $-p^6 z^4 + p^4 z^3 + p z + 1$.  By the theory of Newton Polygons, 
if the fixed points $u_1, \cdots, u_4$ are ordered by increasing size, 
then $|u_1| = p$, $|u_2| = |u_3| = p^{3/2}$, and $|u_4| = p^2$.  The tree $\Gamma$ spanned by 
$\{0, u_1, u_2, u_3, u_4\}$ has branch points at  $\zeta_{D(0,p)}$ and $\zeta_{D(0,p^{3/2})}$.  

Consider $\ordRes_\varphi(\cdot)$ on the path $[0,\infty]$; 
note that only the subsegment $[0,\zeta_{D(0,p^2)}]$ is contained in $\Gamma$.  
Let $Q_A \in \PP^1_{\Berk}$  be the point corresponding to $D(0,|A|)$. Then  
\begin{equation*} 
\ordRes_\varphi(Q_A) 
\ = \ \max\big(\!-\!18\!-\!12\,\ord(A), -6\!-\!6\,\ord(A), 12 + 6\,\ord(A), 18 + 12\,\ord(A)\big) \ .  
\end{equation*} 
This function achieves its minimum value of $3$ when $\ord(A) = -3/2$; 
it has breaks when $\ord(A) = -1$, $\ord(A) = -3/2$, and $\ord(A) = -2$. 

The initial segments of $[\zeta_{D(0,p^{3/2})},u_1]$ and  $[\zeta_{D(0,p^{3/2})},u_4]$ 
belong to $[0,\infty]$, so $\ordRes_{\varphi}(\cdot)$ is increasing along them.  To show that 
$\ordRes_{\varphi}(\cdot)$ achieves its minimum on $\PP^1_{\Berk}$ at $\zeta_{D(0,p^{3/2})}$,
it suffices to check that it is increasing along $[\zeta_{D(0,p^{3/2})},u_2]$ and  $[\zeta_{D(0,p^{3/2})},u_3]$. 

Take  $\gamma  = \left[ \begin{array}{cc} p^{3/2} & 0 \\ 0 & 1 \end{array} \right]$; conjugating $\varphi$
by $\gamma$ brings $\zeta_{D(0,p^{3/2})}$ to $\zeta_G$.  One finds that  
\begin{equation*}
\varphi^\gamma(z) \ = \ \frac{z^3 + z + p^{1/2}}{p^{1/2}z} \ .
\end{equation*} 
The fixed points of $\varphi^{\gamma}$ lie in the directions of $0$, $\pm i$ and $\infty$ at $\zeta_G$, 
where $i = \sqrt{-1}$; 
these correspond to the directions of $u_1$, $u_2$, $u_3$ and $u_4$ at $\zeta_{D(0,p^{3/2})}$, 
respectively.  Since $p$ is odd, the  directions of $\pm i$ at $\zeta_G$ are distinct.  Conjugating
$\varphi^{\gamma}$ by $\nu = \left[ \begin{array}{cc} 1 & i \\ 0 & 1 \end{array} \right]$ yields 
\begin{equation*}
\varphi^{\gamma \nu}(z) \ = \ \frac{(1- i p^{1/2})z^3 + (3+i)z^2 + (-2 + 3 i p^{1/2}) z}{p^{1/2}(z+i)^3} \ .
\end{equation*}
Since $\ord(-2 + 3i p^{1/2}) = 0$ when $p$ is odd, it follows from Lemma \ref{CriterionLemma}
(or directly from formula (\ref{LinearFormula2})), that $\ordRes_{\varphi}(\cdot)$ is increasing 
in the direction of $u_2$ at $\zeta_{D(0,p^{3/2})}$.  A similar argument applies for $u_3$. 

Thus $\varphi(z)$ does not have potential good reduction: $\MinRes(\varphi) = \{\zeta_{D(0,p^{3/2})}\}$, 
with  $\ordRes_\varphi(\zeta_{D(0,p^{3/2})}) = 3$.
Here $\varphi(\zeta_{D(0,p^{3/2})}) = \zeta_{D(0,p)}$, so $\MinRes(\varphi)$ is not fixed by $\varphi$. 
Note that $\rho(\zeta_G,\zeta_{D(0,p^{3/2})}) = 3/2 < \frac{2}{d-1} \ordRes(\varphi) = 12$.
Note also that 
$\ordRes_\varphi(\zeta_{D(0,p)}) = \ordRes_\varphi(\zeta_{D(0,p^2)}) = 6$. 

\smallskip
\noindent{\bf Example 2.5}. \label{Example5} The function 
$\varphi(z) = \displaystyle{\frac{p^n z^3  + z^2 - p^n z}{-p^n z^2 + z + p^n}}$  with $n > 0$ and  
$K = \CC_p$ for any  $p$. 

\smallskip
Here $\Res(\varphi) = -4p^{4n}$, so $\ordRes(\varphi) = 4n + 2 \ord(2)$.
The poles of $\varphi(z)$ are $\alpha_\pm = (1 \pm \sqrt{1+4p^{2n}})/(-2p^n)$, where 
\begin{equation*} 
\alpha_- = -p^n + p^{3n} + \cdots \ , \qquad  \alpha_+ = 1/p^n + p^n - p^{3n} + \cdots \ ,
\end{equation*}  
and the fixed points are $0$ and $\infty$.
The tree $\Gamma$ spanned by $\{0,\infty, \alpha_-, \alpha_+\}$ has branch points at  $\zeta_{D(0,1/p^n)}$ 
and $\zeta_{D(0,p^n)}$.  
    
First consider $\ordRes_\varphi(\cdot)$ on the path $[0,\infty]$.
Let $Q_A \in \PP^1_{\Berk}$  be the point corresponding to $D(0,|A|)$; then  
\begin{equation*} 
\ordRes_\varphi(Q_A) \ = \ 2 \ord(2) + \max\big(-2n - 6 \,\ord(A), 4n, -2n + 6 \,\ord(A)\big) \ .  
\end{equation*} 
This takes its minimum value of $4n + 2 \ord(2)$ for all $A$ with $\ord(A) \in [-n,n]$.  

Note that the paths $[\zeta_{D(0,1/p^n)},\alpha_+]$ and $[\zeta_{D(0,1/p^n)},-p^n]$ share an initial segment.  
Take $\gamma  = \left[ \begin{array}{cc} p^n & 0 \\ 0 & 1 \end{array} \right]$, 
$\nu = \left[ \begin{array}{cc} 1 & -1 \\ 0 & 1 \end{array} \right]$ 
and put $\eta = \gamma \circ \nu = \left[ \begin{array}{cc} p^n & -p^n \\ 0 & 1 \end{array} \right]$.  
Then  $\eta$ takes $[0,\infty]$ 
to the path $[-p^n,\infty]$, with $\eta(\zeta_G) = \zeta_{D(0,1/p^n)}$.  
One computes 
\begin{equation*} 
\varphi^{\eta}(z) \ = \ \big(\varphi^{\gamma}\big)^{\nu}(z) 
        \ = \ \frac{p^{2n} (z-1)^3 - p^{2n} (z-1)^2 + z^2  -z +1}{-p^{2n}(z-1)^2 + z} \ .
\end{equation*} 
If we write the numerator of $\varphi^{\eta}$ as $a_3 z^3 + a_2 z^2 + a_1 z + a_0$, 
then $\ord(a_2) = 0$, and it follows from (\ref{LinearFormula2}) that  
$\ordRes_{\varphi^{\eta}}(\cdot)$ is increasing in the 
direction $\vv_0$ at $\zeta_G$.  This means $\ordRes_{\varphi}(\cdot)$ is increasing 
in the direction $\vv_{\alpha_{+}}$ at $\zeta_{D(0,1/p^n)}$.  By a similar argument,
one sees that $\ordRes_{\varphi}(\cdot)$ is increasing 
in the direction $\vv_{\alpha_{-}}$ at $\zeta_{D(0,p^n)}$. 

Thus $\MinRes(\varphi)$ 
is the segment $[\zeta_{D(0,1/p^n)}, \zeta_{D(0,p^n)}]$,
and the minimal value of $\ordRes_\varphi(\cdot)$ is $4n + 2 \ord(2)$; in particular 
$\varphi(z)$ does not have potential good reduction.  
Each point of $[\zeta_{D(0,1/p^n)}, \zeta_{D(0,p^n)}]$ is fixed by $\varphi$:  
$\zeta_{D(0,p^\alpha)}$ is an indifferent fixed point for $-n < \alpha < n$,
and $\zeta_{D(0,1/p^n)}$ and $\zeta_{D(0,p^n)}$ are repelling fixed points of degree $2$. 
In this case $\MinRes(\varphi)$ is contained in $\{ z \in \HH_{\Berk} : \rho(\zeta_G,z) \le n\}$, 
while $\frac{2}{d-1} \ordRes(\varphi) = \frac{8}{3} n + \frac{4}{3} \ord(2)$.

\smallskip
For rationality considerations later, 
it will be useful to examine  some conjugates of $\varphi(z)$.  
For each $\gamma \in \GL_2(K)$, it is a formal consequence of the definitions that for all $Q$ 
\begin{equation} \label{Equicontinuity}
\ordRes_{\varphi^{\gamma}}(Q) \ = \ \ordRes_{\varphi}(\gamma(Q)) \ .
\end{equation}  
To show this, by continuity it is enough to check it for type II points.
Suppose $Q = \tau(\zeta_G)$ for some $\tau \in \GL_2(K)$. Then 
\begin{equation*}
\ordRes_{\varphi^{\gamma}}(Q) \ = \ \ordRes_{\varphi^{\gamma}}(\tau(\zeta_G)) 
\ = \ \ordRes_{\varphi}(\gamma(\tau(\zeta_G))) \ = \ \ordRes_{\varphi}(\gamma(Q)) \ .
\end{equation*}  

\smallskip
Take $u \in \CC_p$ with $|u| = 1$, 
and let $\gamma_1 = \left[ \begin{array}{cc} 1 & u \\ -1 & u \end{array} \right]$.
One easily sees that 
\begin{equation} \label{Variant1}
\varphi_1(z) \ := \ \varphi^{\gamma_1}(z) \ = \ \frac{- z^3  + (4p^n+1)u^2 z}{(4p^n-1) z^2 + u^2} \ ,
\end{equation} 
and that $\gamma_1(\zeta_{D(u,1/p^n)}) = \zeta_{D(0,p^n)}$ and  
$\gamma_1(\zeta_{D(-u,1/p^n)}) = \zeta_{D(0,1/p^n)}$.  It follows that 
$\MinRes(\varphi_1)$ is the segment $[\zeta_{D(-u,1/p^n)},\zeta_{D(u,1/p^n)}]$.
When $p$ is odd, the midpoint of this segment is $\zeta_G = \zeta_{D(0,1)}$.  When $p = 2$,
its midpoint is $\zeta_{D(u,1/2)}$.  

Next conjugate $\varphi_1(z)$ by $\gamma_2 = \left[ \begin{array}{cc} 1 & 0 \\ 0 & p^{1/2} \end{array} \right]$.
Then 
\begin{equation} \label{Variant2}
\varphi_2(z) \ := \ (\varphi_1)^{\gamma_2}(z) \ = \ 
\frac{- z^3 + (4p^n+1)u^2 p z}{(4p^n-1) z^2 + p u^2} \ ,
\end{equation}  
and $\MinRes(\varphi_2) = [\zeta_{D(-up^{1/2},1/p^{(n+1/2)})},\zeta_{D(up^{1/2},1/p^{(n+1/2)})}]$. 
When $p$ is odd, the midpoint of this segment is $\zeta_{D(0,p^{-1/2})}$. 
When $p = 2$, its midpoint is  $\zeta_{D(u 2^{1/2},2^{-3/2})}$.

\smallskip
\noindent{\bf Example 2.6}. \label{Example6} The function 
$\varphi(z) = \displaystyle{\frac{z^2}{(1 + pz)^4}}$\,, 
where  $K = \CC_p$ and $p \ge 5$.  
This function was studied by Favre and Rivera-Letelier (\cite{FRLErgodic}; or see \cite{B-R}, Example 10.124), 
who showed that its Berkovich Julia set is the segment $[\zeta_G,\zeta_{D(0,p^2)}]$ and that its
invariant measure $\mu_{\varphi}$ is the uniform measure of mass $1$ on that segment 
(relative to the path distance).  Here $\Res(\varphi) = p^8$.
The poles of $\varphi$ are all at $z = -1/p$, and the fixed points of $\varphi$ are $z = 0$
and the roots of $1 + (4p-1)z + 6p^2z^2 + 4 p^3 z^3 + p^4z^4 = 0$. 
By the theory of Newton polygons, these roots can be labeled so that
$|u_1| = 1$ and $|u_2| = |u_3| = |u_4| = p^{4/3}$.  The tree $\Gamma$ spanned by $\{0,-1/p, u_1, u_2, u_3, u_4\}$
has branch points at $\zeta_G$, $\zeta_{D(0,p)}$, and $\zeta_{D(0,p^{4/3})}$.

On the path $[0,\infty]$, we have
\begin{equation*}
\ordRes_{\varphi}(\zeta_{D(0,|A|)}) \ = \ \max\big(-24 - 20 \,\ord(A), 8 + 4\,\ord(A) \big) \ ,
\end{equation*}
which takes its minimum value of $8/3$ at $\ord(A) = 4/3$.  
Conjugating by $\gamma = \left[ \begin{array}{cc} p^{-4/3} & 0 \\ 0 & 1 \end{array} \right]$
gives 
\begin{equation} \label{Fixed43} 
\varphi^{\gamma}(z) \ = \ \frac{z^2}{z^4 + 4p^{1/3} z^3 + 6 p^{2/3} z^2 +4pz + p^{4/3}} \ .
\end{equation}
The fixed points  $u_2, u_3, u_4$ lie in the 
directions $\vv_{p^{4/3}}, \vv_{\zeta_3 p^{4/3}}, \vv_{\zeta_3^2 p^{4/3}}$ at $\zeta_{D(0,p^{4/3})}$, 
where $\zeta_3$ is a primitive cube root of unity,  
and it is easily checked that $\ordRes_{\varphi}(\cdot)$ is increasing in each of those directions.
Thus $\MinRes(\varphi) = \{\zeta_{D(0,p^{4/3})}\}$.   
Note that $\zeta_{D(0,p^{4/3})}$ is fixed by $\varphi$; indeed, by (\ref{Fixed43}), 
$\zeta_{D(0,p^{4/3})}$ is a repelling fixed point of $\varphi$ of degree $2$.   
Also note that $\rho(\zeta_G,\zeta_{D(0,p^{4,3})}) = 4/3 < \frac{2}{d-1} \ordRes(\varphi) = 16/3$.

\smallskip
\noindent{\bf Example 2.7}. \label{Example7} The function 
$\varphi(z) = \displaystyle{\frac{pz^3+z^2}{p}}$\,, 
with $K = \CC_p$ for an arbitrary prime $p$. 

Here $\Res(\varphi) = p^6$.
The fixed points of $\varphi(z)$ are $0$, $\infty$ and the solutions $u_1, u_2$ to $pz^2 + z - p = 0$:  
\begin{equation*}
u_1 = p + p^3 + \cdots, \qquad  u_2 = -p^{-1} -p -p^3 + \cdots 
\end{equation*}  
so that $|u_1| = 1/p$, $|u_2| = p$.  The poles of $\varphi(z)$ are all at $\infty$. 
 The tree $\Gamma$ spanned by 
$\{0, \infty, u_1, u_2 \}$ has branch points at $\zeta_{D(0,1/p)}$ and $\zeta_{D(0,p)}$.
   
First consider $\ordRes_{\varphi}(\cdot)$ on  $[0,\infty]$.  
We have  
\begin{eqnarray*} 
\ordRes_\varphi(\zeta_{D(0,|A|)}) & = & 6 + 12 \ord(A) - 6 \, \min\big(\ord(pA^3), \ord(A^2), \ord(pA) \big) \\
          & = & \max\big(-6 \ord(A), 6, 6 + \ord(A)) \ .  
\end{eqnarray*} 
This takes the constant value $6$ when $-1 \le \ord(A) \le 0$.  
By convexity, the minimum value of  $\ordRes_{\varphi}(\cdot)$ on $\PP^1_{\Berk}$ is $6$,
and $\MinRes(\varphi)$ contains the segment $[\zeta_G,\zeta_{D(0,p)}]$.  

To see that $\MinRes(\varphi)$ contains no other points, note that the path $[\zeta_{D(0,p)},u_2]$
shares an initial segment with $[\zeta_{D(0,p)},p^{-1}]$.  Conjugating $\varphi$ by 
$\gamma = \left[ \begin{array}{cc} 1/p & 1/p \\ 0 & 1 \end{array} \right]$,
which takes $0$ to $p^{-1}$ and $\zeta_G$ to $\zeta_{D(0,1/p)}$, yields
$\varphi^{\gamma}(z) = (z^3 + 4z^2 + 5z + (2-p^2))/p^2$.
One computes
\begin{eqnarray*}
& & \ordRes_{\varphi^{\gamma}}(\zeta_{D(0,|A|)}) \\ 
& & \qquad  \ = \ \max\big(6 -6 \,\ord(A), 6 - 6 \, \ord(5) + 6 \, \ord(A), 6 - 6 \,\ord(2) + 12 \, \ord(A)) \ .  
\end{eqnarray*} 
Since either $\ord(5) = 0$ or $\ord(2) = 0$, the right side is increasing for small positive values of $\ord(A)$.
Thus $\ordRes_{\varphi}(\cdot)$ is increasing along $[\zeta_{D(0,p)},u_2]$, 
and $\MinRes(\varphi) = [\zeta_G,\zeta_{D(0,p)}]$. 
Here $\MinRes(\varphi)$ is contained in $\{ z \in \HH_{\Berk} : \rho(\zeta_G,z) \le 1\}$, 
while $\frac{2}{d-1} \ordRes(\varphi) = 6$.  
For $0 \le \alpha \le 1$, we have 
$\varphi(\zeta_{D(0,p^{\alpha})}) = \zeta_{D(0,p^{2 \alpha + 1})}$, 
so no point of $\MinRes(\varphi)$ is fixed by $\varphi$.  

\section{Discussion, Applications, and Questions} \label{DiscApplicSection} 

Examples 2.1, 2.2, and 2.6 show that $\MinRes(\varphi)$ 
need not be contained in the tree spanned by the 
fixed points alone, or the poles alone.  Examples 2.3, 2.5, and 2.7 show that when $d$ is odd, 
$\MinRes(\varphi)$ can be either a point or a segment.  

When $\varphi$ has potential good reduction, 
$\MinRes(\varphi)$ consists of a single point, 
which is necessarily fixed by $\varphi$.
When $\varphi$ does not have potential good reduction, $\MinRes(\varphi)$ may or may not contain fixed points.
In Example 2.6 it consists of a single point, which is fixed. In Example 2.5, it consists of a segment,
which is pointwise fixed.  In Examples 2.3 and 2.4, it consists of a single point, which is not fixed; 
in Example 2.7, it consists of a segment, of which no point is fixed.

In the examples, $\MinRes(\varphi)$ lies well inside  
$\{z \in \HH_{\Berk} : \rho(\zeta_G,z) \le \frac{2}{d-1} \ordRes(\varphi)\}$. 
Probably the radius $\frac{2}{d-1} \ordRes(\varphi)$ given by Theorem \ref{MainThm} is not sharp.  
 
\medskip
{\bf Rationality Considerations.}  Let $H$ be a subfield of $K$.  Throughout this subsection, we 
will assume $\varphi(z) \in H(z)$.

We will say that a point $Q \in \PP^1_{\Berk}$ is {\em rational over $H$} if it is type I point in $\PP^1(H)$ 
or is a type II point corresponding to a disc $D(b,r)$ with $b \in H$ and radius $r \in |H^{\times}|$. 
A type II point is rational over $H$ if and only if it belongs to the tree spanned by $\PP^1(H)$
and corresponds to a disc with radius $r \in |H^{\times}|$.  
The following proposition shows the $H$-rational type II points are those which can be reached
from $\zeta_G$ by an element of $\GL_2(H)$;  it also shows that the notion of $H$-rationality
for type II points is invariant under $H$-rational changes of coordinates.

\begin{proposition} \label{RationalityEquiv} A type II point $Q$ is rational over $H$ 
if and only if $Q = \gamma(\zeta_G)$
for some $\gamma \in \GL_2(H)$.
\end{proposition}

\begin{proof} 
If $Q$ is rational over $H$, it corresponds to a disc $D(b,|a|)$ where $b \in H$
and $a \in H^{\times}$.  Put $\gamma = \left[ \begin{array}{cc} a & b \\ 0 & 1 \end{array} \right] \in \GL_2(H)$;
then $Q = \gamma(\zeta_G)$.  Conversely, suppose $Q = \gamma(\zeta_G)$ where 
$\gamma = \left[ \begin{array}{cc} a & b \\ c & d \end{array} \right] \in \GL_2(H)$.  
Write $\cO_H$ for the ring of integers of $H$. Multiplying $\gamma$ on the 
right by $\left[ \begin{array}{cc} 0 & 1 \\ 1 & 0 \end{array} \right] \in \GL_2(\cO_H)$ interchanges
the columns of $\gamma$, so without loss we can assume  $|c| \le |d|$.  Then, multiplying $\gamma$
on the right by $\left[ \begin{array}{cc} 1 & 0 \\ -c/d & 1 \end{array} \right] \in \GL_2(\cO_H)$ 
brings it to the form $\left[ \begin{array}{cc} a_1 & b_1 \\ 0 & d_1 \end{array} \right] \in \GL_2(H)$.
Since $\GL_2(\cO_H)$ stabilizes $\zeta_G$, $Q$ corresponds to the disc $D(b_1/d_1,|a_1/d_1|)$, 
and is rational over $H$.       
\end{proof} 

Let $\Aut^c(K/H)$ be the group of continuous automorphisms 
of $K$ fixing $H$.  The  natural action of $\Aut^c(K/H)$ on $\PP^1(K)$ 
extends to an action on $\PP^1_{\Berk}$ which preserves the type of each point. 
On points of type II or III, the action can be described as follows: if $\sigma \in \Aut^c(K/H)$ and $Q$ 
corresponds to the disc $D(b,r)$, then $\sigma(Q)$ corresponds to $D(\sigma(b),r)$.
The image disc is well-defined, since for any $b^{\prime} \in K$ with $D(b^{\prime},r) = D(b,r)$ we have 
$|\,\sigma(b^{\prime}) - \sigma(b)| = |\,b^{\prime} - b| \le r$. 
For a point $Q$ of type IV,
if $Q$ corresponds to a sequence of nested discs $\{D(a_i,r_i\}_{i \ge 0}$ under Berkovich's classification theorem,
then $\sigma(Q)$ corresponds to the sequence of nested discs $\{D(\sigma(a_i),r_i)\}_{i \ge 0}$.

If a type II point is rational over $H$, clearly it is fixed by each $\sigma \in \Aut^c(K/H)$.
However, the converse is not true:  a type II point fixed by $\Aut^c(K/H)$ need not be $H$-rational.
Indeed, each point in the tree spanned by $\PP^1(H)$,
with radius in $|H^{\times}|$ or not, is fixed by $\Aut^c(K/H)$.
There can also be type II points in $\PP^1_\Berk$ outside the tree spanned by $\PP^1(H)$
which are fixed by $\Aut^c(K/H)$.
For example, if $H = \QQ_2$ and $K = \CC_2$, then $\zeta_{D(i,1/2)}$ is fixed by $\Aut^c(\CC_2/\QQ_2)$
since $|\,\sigma(i) - i| \le 1/2$ for each $\sigma \in \Aut^c(\CC_2/\QQ_2)$.  
However $D(i,1/2) \cap \QQ_2$ is empty: $|x-i| \ge 1/\sqrt{2}$ for each $x \in \QQ_2$.  
Thus $\zeta_{D(i,1/2)}$ is not in the tree spanned by $\PP^1(\QQ_2)$.  

The action of $\sigma \in \Aut^c(K/H)$ on $\PP^1_\Berk$ is continuous for the strong topology: 
indeed, the description of the action shows that for 
all $x, y \in \HH_\Berk$, one has $\rho(\sigma(x),\sigma(y)) = \rho(x,y)$.  
It follows that $\sigma$ takes paths to paths: if $[x,y]$ is a path with endpoints in $\HH_\Berk$, 
then for each $Q \in \HH^1_\Berk$ we have $Q \in [x,y]$ iff $\rho(x,y) = \rho(x,Q) + \rho(Q,y)$;
thus $Q \in [x,y]$ iff $\sigma(Q) \in [\sigma(x),\sigma(y)]$.  If $[x,y]$ has one or both endpoints in $\PP^1(K)$,
it can be exhausted by an increasing sequence of paths with endpoints in $\HH_\Berk$, so we still have
$\sigma([x,y]) = [\sigma(x),\sigma(y)]$.

\begin{proposition} \label{GaloisGL2} For all $\varphi(z) \in K(z)$, all $\sigma \in \Aut^c(K/H)$, 
and all $Q \in \PP^1_\Berk$, we have $\sigma(\varphi(Q)) = (\sigma(\varphi))(\sigma(Q))$.
\end{proposition}

\begin{proof} Given $\varphi(z) \in K(z)$ and $\sigma \in \Aut^c(K/H)$,
if $Q$ is of type I the assertion is clear.  If $Q$ is of type II and corresponds to a disc $D(b,r)$, 
the assertion follows from the case of type I points and the description of the action of $\varphi$ 
on generic type I points in $D(b,r)$ given in (\cite{B-R}, Proposition 2.18).  Finally, if $Q$ is of type III
or IV, the assertion follows from the case of type II points and continuity.     
\end{proof} 

In particular, $\sigma(\gamma(Q)) = \gamma(\sigma(Q))$ 
for all $\gamma \in \GL_2(H)$ and $\sigma \in \Aut^c(K/H)$.  
This shows the action of $\Aut^c(K/H)$ on $\PP^1_\Berk$ is independent of $H$-rational changes of coordinates.

\smallskip
We will say that a subset $X \subset \PP^1_{\Berk}$
is {\em stable under $\Aut^c(K/H)$} if $\sigma(x) \in X$ for each $x \in X$ and $\sigma \in \Aut^c(K/H)$, 
that $X$ is {\em pointwise fixed by $\Aut^c(K/H)$} 
if $\sigma(x) = x$ for each $x \in X$ and $\sigma \in \Aut^c(K/H)$. 

\begin{proposition} \label{GaloisFixedPt}
If $\varphi$ is rational over a subfield $H \subset K$, then 
$\MinRes(\varphi)$ is stable under $\Aut^c(K/H)$, and it 
contains at least one point fixed by $\Aut^c(K/H)$.   
However, $\MinRes(\varphi)$  need not contain points of the tree spanned by $\PP^1(H)$,
and it need not be pointwise fixed by $\Aut^c(K/H)$.  
On the other hand, if $\deg(\varphi)$ is odd, $\MinRes(\varphi)$ 
can contain arbitrarily many $H$-rational type II points.     
\end{proposition}

\begin{proof}
If $\varphi$ is rational over $H$, then $\ordRes_{\varphi}(\sigma(Q)) = \ordRes_{\varphi}(Q)$ 
for all $\sigma \in \Aut^c(K/H)$ and all $Q \in \PP^1_{\Berk}$.  
Thus, $\MinRes(\varphi)$ is stable under $\Aut^c(K/H)$.  To see that $\MinRes(\varphi)$ always
contains at least one point fixed by $\Aut^c(K/H)$, note that if $\MinRes(\varphi)$ consists of a single 
point, $\Aut^c(K/H)$ fixes that point.  On the other hand, if $\MinRes(\varphi)$ is a segment, 
then since $\Aut^c(K/H)$ preserves path distances, 
each $\sigma \in \Aut^c(K/H)$ must either leave $\MinRes(\varphi)$ pointwise
fixed, or flip it end-to-end;  in either case $\sigma$ fixes the midpoint of $\MinRes(\varphi)$.  

Example 2.2, with $\varphi(z) = (z^2-z)/(2z)$ and $H = \QQ_2$,  
shows that $\MinRes(\varphi)$ can be pointwise fixed by $\Aut^c(K/H)$ 
without meeting the tree spanned by $\PP^1(H)$: $\MinRes(\varphi) = \{\zeta_{D(i,1/2)}\}$, 
and $\zeta_{D(i,1/2)}$ does not belong to the tree spanned by $\PP^1(H)$, as shown above.  
It would be interesting to know how far off the tree $\MinRes(\varphi)$ can lie.  
 
Example 2.5,  with 
$\varphi(z) = \displaystyle{\frac{p^n z^3  + z^2 - p^n z}{-p^n z^2 + z + p^n}}$ and $H = \QQ_p$, 
shows that when $d = \deg(\varphi)$ is odd, $\MinRes(\varphi)$ can contain arbitrarily many type II 
points rational over $H$:  in that example $\MinRes(\varphi)$ is a segment of path-length $2n$ contained 
in the path $[0,\infty]$ with $H$-rational endpoints.  

It is also possible for $\MinRes(\varphi)$ to be a segment 
``orthogonal to'' the tree spanned by $\PP^1(H)$:  take $H = \QQ_p$ with $p$ odd.  
If $a \in \ZZ_p^{\times}$ is a non-square unit, 
and $u = \sqrt{a}$, then the function 
$\displaystyle{\varphi_1(z) = \frac{- z^3  + (4p^n+1)u^2 z}{(4p^n-1) z^2 + u^2}}$ 
from Example 2.5 is $\QQ_p$-rational.
Its minimal resultant locus is $[\zeta_{D(-u,1/p^n)},\zeta_{D(u,1/p^n)}]$, 
which meets the tree spanned by $\PP^1(\QQ_p)$ only at the $\QQ_p$-rational type II point $\zeta_G$.   
Likewise, the function 
$\displaystyle{\varphi_2(z) = \frac{- z^3 + (4p^n+1)u^2 p z}{(4p^n-1) z^2 + p u^2}}$
from Example 2.5 is $\QQ_p$-rational. Its minimal resultant locus 
meets the tree spanned by $\PP^1(\QQ_p)$ at $\zeta_{D(0,p^{-1/2})}$, 
but that point is not $\QQ_p$-rational 
because its radius does not belong to $|\QQ_2^{\times}|$.  In both examples, 
each $\sigma \in \Aut^c(K/H)$ with $\sigma(\sqrt{a}) = -\sqrt{a}$ 
flips $\MinRes(\varphi)$ end-to-end;  the midpoint of $\MinRes(\varphi)$ 
is the only point fixed by $\Aut^c(K/H)$.
\end{proof}

Now assume that $H$ is discretely valued:  
in this case, the $H$-rational type II points are discrete in $\HH_\Berk$ for the strong topology,
and the subtree of  $\PP^1_{\Berk}$ spanned by $\PP^1(H)$ 
is branched at precisely the $H$-rational type II points.

If $Q$ is a type II point rational over $H$, the action 
of $\Aut^c(K/H)$ on $\PP^1_\Berk$ induces an action of $\Aut^c(K/H)$ on the tangent space $T_Q$, which takes
the class of a path $[Q,x]$ to the class of $[Q,\sigma(x)]$.  This is well-defined, since if $x$ and $x^{\prime}$ 
belong to the same tangent direction at $Q$, then the paths $[Q,x]$ and $[Q,x^{\prime}]$ share an initial segment;
thus $[Q,\sigma(x)]$ and $[Q,\sigma(x^{\prime})]$ share an initial segment as well.  

\smallskip
The following proposition shows that if $\varphi$ is rational over $H$, 
and if $Q \notin \MinRes(\varphi)$ is a type II point rational over $H$, 
then $\MinRes(\varphi)$ lies in a tangent direction at $Q$ fixed by $\Aut^c(K/H)$.
When $H = H_v$ is a local field, we will use this 
in giving a steepest descent algorithm for finding 
an $H_v$-rational point where $\ordRes_\varphi(\cdot)$ is minimal for $H_v$-rational points.

\begin{proposition} \label{SteepestDescentProp} 
Suppose $H_v$ is a local field and $\varphi$ is rational over $H_v$.  
Let $Q$ be an $H_v$-rational type II point not contained in $\MinRes(\varphi)$. 
Then $\MinRes(\varphi)$ lies in a tangent direction at $Q$
coming from the tree spanned by $\PP^1(H_v)$.  
\end{proposition} 

\begin{proof} 
If $H_v$ has residue field $\FF_q$, 
then $T_Q$ is parametrized by $\PP^1(\overline{\FF_q})$ and the 
tangent directions at $Q$ fixed by $\Aut^c(K/H_v)$ 
correspond to the points of $\PP^1(\FF_q)$.  These are precisely the tangent directions 
at $Q$ coming from the tree spanned by $\PP^1(H_v)$.  (We remark that even if $H$ is not a local field,
the conclusion of the proposition will hold if the residue field of $K$ is separable 
over the residue field of $H$.)   
\end{proof}

If $H_v$ is a local field and $\MinRes(\varphi)$ contains no $H_v$-rational type II 
points,
there are exactly two $H_v$-rational type II points adjacent to it in the tree spanned by $\PP^1(H_v)$.  
The function $\ordRes_\varphi(\cdot)$
may take the same or different values at those points;
its value is strictly larger at all other $H_v$-rational type II points.  
Example 2.2 gives a case where the minimum is taken on at one of the two adjacent $H$-rational type II points,
and Example 2.4 gives a case where it is taken on at both points.  

\medskip

{\bf Bounds for the degree of an extension where $\varphi^{\gamma}$ has Minimal Resultant. \quad} 
It is useful to note that in Theorem \ref{MainThm}, the tree $\Gamma_{\Fix,\varphi^{-1}(\infty)}$
can be replaced by the tree $\Gamma_{\Fix,\varphi^{-1}(a)}$ spanned by the fixed points of $\varphi$ 
and the preimages of $a$, for any $a \in \PP^1(K)$:

\begin{proposition} \label{GeneralTreeProp} 
For each \,$a \in \PP^1(K)$, $\MinRes(\varphi)$ is contained in the tree $\Gamma_{\Fix,\varphi^{-1}(a)}$
spanned by the fixed points of $\varphi$ and the set of preimages $\{z \in \PP^1(K) : \varphi(z) = a\}$. 
\end{proposition} 

\begin{proof} Take $a \in \PP^1(K)$, and choose $\gamma \in \GL_2(K)$ with $\gamma(\infty) = a$.  
It follows from (\ref{Equicontinuity}) that  
\begin{equation*} 
\MinRes(\varphi)  \ = \ \gamma(\MinRes(\varphi^{\gamma})) \ .
\end{equation*} 
By Theorem \ref{MainThm}, $\MinRes(\varphi^{\gamma})$ 
is contained in the tree $\Gamma_{\Fix,(\varphi^{\gamma})^{-1}(\infty)}$
spanned by the fixed points and poles of $\varphi^{\gamma}$,
so $\MinRes(\varphi)$
is contained in the tree $\gamma(\Gamma_{\Fix,(\varphi^{\gamma})^{-1}(\infty)})$.  
By equivariance, 
$Q$ is a fixed point of $\varphi^{\gamma}$ if and only if $\gamma(Q)$ is a fixed point of $\varphi$,
and $P$ is a pole of $\varphi^{\gamma}$ if and only if $\varphi(\gamma(P)) = a$.
Thus $\gamma(\Gamma_{\Fix,(\varphi^{\gamma})^{-1}(\infty)}) = \Gamma_{\Fix,\varphi^{-1}(a)}$.
\end{proof} 

Let $M = \min_{Q \in \PP^1_{\Berk}}\big(\ordRes_{\varphi}(Q)\big) 
= \min_{\gamma \in \GL_2(K)}\big(\ordRes(\varphi^{\gamma})\big)$.  
 
\begin{theorem} \label{DegreeBoundTheorem}
Let $H$ be a subfield of $K$, and suppose $\varphi(z) \in H(z)$ has degree $d \ge 2$.  
Then there is an extension $L/H$ in $K$ with $[L:H] \le (d+1)^2$ such that 
$\ordRes(\varphi^{\gamma}) = M$ for some $\gamma \in \GL_2(L)$. 
\end{theorem} 

\begin{proof} It is enough to show there is an extension $L/H$ with $[L:H] \le (d+1)^2$ such that 
$\MinRes(\varphi)$ contains a type II point $Q$ rational over $L$.  

Put $a = \varphi(\infty) \in \PP^1(H)$.  Let $F_1, \ldots, F_{d+1}$ be the fixed points of $\varphi$,
and let $A_1, \ldots,A_d$ be the preimages of $a$ under $\varphi$, listed with multiplicity.  
Without loss we can assume that $A_1 = \infty$.  By Proposition \ref{GeneralTreeProp}, 
$\MinRes(\varphi)$ is contained in the tree $\Gamma_{\Fix,\varphi^{-1}(a)}$, 
which is the union of the paths $[F_i,\infty]$ and $[A_j,\infty]$ for $i = 1, \ldots, d+1$, $j = 2, \ldots, d$.
Let $Q$ be an endpoint of $\MinRes(\varphi)$, and let $P \in \{F_1, \ldots, F_{d+1}, A_2, \ldots, A_d\}$
be such that $Q \in [P,\infty]$.  Put $L_0 = H(P)$.  
We have $[H(F_i):H] \le d+1$ for each $i$, and $[H(A_j):H] \le d-1$ for each $j$, 
so $[L_0:H] \le d+1$.  Fix $\gamma \in \GL_2(L_0)$ with $\gamma(0) = P$ and $\gamma(\infty) = \infty$, 
and let $Q_0 = \gamma^{-1}(Q) \in [0,\infty]$.  By the discussion after formula (\ref{IntRat}),
there are an $\alpha \in L_0^{\times}$
and an integer $e$ with $1 \le e \le d+1$ such that $Q_0 = \zeta_{D(0,|\alpha|^{1/e})}$. 
Put $L = L_0(\alpha^{1/e})$.  Then $Q$ is rational over $L$, and $[L:H] \le (d+1)^2$.  
\end{proof}

\begin{corollary} \label{LogicCor} For each $d \ge 2$, 
there is a first order formula $\cF_d(f_0,\ldots, f_d, g_0, \ldots, g_d)$ in the language of valued fields
such that if $H$ is a Henselian nonarchimedean valued field, and if $\varphi(z) = (f_d a^d +  \cdots + f_0)/(g_d z^d + \cdots + g_0) \in H(z)$, 
then $\varphi$ has potential good reduction if and only if 
$H \models \cF_d(f_0, \cdots,f_d, g_0, \cdots, g_d)$.
\end{corollary}

\begin{proof}  If $H$ is Henselian (in particular, if $H$ is complete), then for each finite extension 
$H(\beta)/H$ there is a unique extension of the valuation $\ord(\cdot)$ on $H$ to a valuation on $H(\beta)$,
given by $\ord_{H(\beta)}(z) = (1/m) \ord(N_{H(\beta)/H}(z))$ for $z \in H(\beta)$, where  $[H(\beta):H] = m$.
If $z = a_0 + a_1 \beta + \cdots + a_{m-1} \beta^{m_1}$ with $a_0, \ldots, a_{m-1} \in H$, 
then $N_{H(\beta)/H}(z)$ is a universal polynomial in the $a_i$ 
and the coefficients of the minimal polynomial of $\beta$ over $H$.  

Write $(F,G)$ for the natural representation of $\varphi$. 
Let $\cF_{d,0}(f_0, \cdots,f_d, g_0, \cdots, g_d)$ be the formula  ``$\Res(F,G) \ne 0$'',
and for $m = 1, \ldots, (d+1)^2$ let $F_{d,m}(f_0, \cdots,f_d, g_0, \cdots, g_d)$ be the formula 
\begin{verse}
``\ There exist $a_1, \ldots, a_m \in H$ such that  $h_m(x) = x^m + a_1 x^{m-1} + \cdots + a_m$\\ 
           \qquad \qquad is irreducible over $H$, and there exist a root $\beta$ of $h_m(x)$  \\
 \ and $a, b, c, d \in H(\beta)$ with $ad-bc \ne 0$, such that for 
$\gamma = \left[ \begin{array}{cc} a & b \\ c & d \end{array} \right]$ \\
\ we have \ $\ord_{H(\beta)}(\Res(F^\gamma,G^\gamma)) 
- 2d \min(\ord_{H(\beta)}(F^\gamma),\ord_{H(\beta)}(G^\gamma)) = 0$.''
\end{verse}
We can take $\cF_d$ to be $\cF_{d,0} \wedge \big(\cF_{d,1} \vee \cdots \vee \cF_{d,(d+1)^2}\big)$.
\end{proof} 

\medskip
{\bf Failure to achieve the Minimal Resultant over the Field of Moduli.} 

Suppose $\varphi(z) \in H(z)$, where $H \subset K$.  Let $\cF_H(\varphi)$
be the set of fields $L$ with $H \subseteq L \subseteq K$ for which there is some $\gamma \in \GL_2(L)$
such that $\ordRes(\varphi^{\gamma})$ is minimal.
When $\MinRes(\varphi) = \{Q\}$ consists of a single point, $\cF_H(\varphi)$ is the 
set of fields $H \subseteq L \subseteq K$ such that there is some $\gamma \in \GL_2(L)$
with $\gamma(\zeta_G) = Q$.   
The {\em field of moduli} for the minimal resultant problem is 
\begin{equation*}
H_\varphi \ := \ \bigcap_{L \in \cF_H(\varphi)} L \ .
\end{equation*}
 
It is natural to ask if there is a $\gamma \in \GL_2(H_{\varphi})$ for which $\ordRes(\varphi^{\gamma})$ is minimal.
If $\MinRes(\varphi)$ contains an $H$-rational point, the answer is trivially yes.
If $\MinRes(\varphi)$ contains no $H$-rational points, the answer is generally no.  
In Example 2.1, take $d = p > 2$, with $\varphi(z) = \frac{z^p-p}{z^{p-1}}$ and $H = \QQ_p$.
We have $\MinRes(\varphi) = \{Q\}$ where $Q = \zeta_{D(0,p^{1/p})}$.  Here $Q$ is rational over $L$
if and only if the value group of $L$ contains $p^{1/p}$.  
In particular, $Q$ is rational over $L_1 = \QQ_p(\sqrt[p]{p})$ 
and over $L_2 = \QQ_p(\zeta_p\sqrt[p]{p})$ where $\zeta_p$ is any primitive $p^{th}$ root of unity.
Since $p > 2$, necessarily $L_1 \cap L_2 = \QQ_p$ (otherwise $L_1 = L_2$, since both extensions have degree $p$;
but then $\zeta_p \in L_1$, so $p-1 = [\QQ_p(\zeta_p):\QQ_p]$ divides $[L_1:\QQ_p] = p$).
Thus $H_{\varphi} = \QQ_p$.
However, $p^{1/p}$ is not in the value group of $\QQ_p^{\times}$, 
so by Proposition \ref{RationalityEquiv} there can be no $\gamma \in \GL_2(\QQ_p)$ with $\gamma(\zeta_G) = Q$.
Likewise, in Example 2.2, for $\varphi(z) = \frac{z^2-1}{2z}$ and $H = \QQ_2$, 
we have $\MinRes(\varphi) = \{Q\}$ where $Q = \zeta_{D(i,1/2)}$.  Here $D(i,1/2) = D(\sqrt{3},1/2)$  
since $|i - \sqrt{3}| = 1/2$, so $Q = \gamma_1(\zeta_G) = \gamma_2(\zeta_G)$ where
$\gamma_1 = \left[ \begin{array}{cc} 2 & i \\ 0 & 1 \end{array} \right]$
and $\gamma_2 = \left[ \begin{array}{cc} 2 & \sqrt{3} \\ 0 & 1 \end{array} \right]$.
Since $\QQ_2(i) \cap \QQ_2(\sqrt{3}) = \QQ_2$,  we have $H_{\varphi} = \QQ_2$.  
However,  $D(i,1/2) \cap \QQ_2$ is empty.  
Hence there can be no $\gamma \in \GL_2(\QQ_2)$ with $\gamma(\zeta_G) = Q$. 

\medskip
{\bf Answers to questions of Silverman concerning global Minimal Models.}

Throughout this subsection, $H$ will be a number field,
and $\varphi(z) \in H(z)$ will have degree $d \ge 2$.
Let $\cO_H$ be the ring of integers of $H$.
Given a nonarchimedean place $v$ of $H$, 
let $H_v$ be the completion of $H$ at $v$, $\cO_v$ the valuation ring of $H_v$, and $\pi_v$ 
a generator for the maximal ideal of $\cO_v$. Let $\CC_v$ be the completion of the algebraic closure of $H_v$.  
We will write $\ord_v(\cdot)$ for the valuation on $\CC_v$ normalized so that $\ord_v(\pi_v) = 1$, 
and $\ordRes_v(\varphi)$ and $\ordRes_{\varphi,v}(\cdot)$
for the functions previously denoted $\ordRes(\varphi)$ and $\ordRes_\varphi(\cdot)$.  
In this way the theory developed above is applicable for each nonarchimedean place $v$ of $H$.

A representation $(F,G)$ of $\varphi$ with $F(X,Y), G(X,Y) \in H[X,Y]$ is called a 
{\em representation of $\varphi$ over $H$};  such a pair
is unique up to scaling by an element of $H^{\times}$.  One can always arrange that 
$F,G \in \cO_H[X,Y]$;  in that case, the representation is called {\em integral}.  

In (\cite{Sil}, \S4.11), Silverman asks if (and when) it is possible to choose  
an ``optimal'' integral representation for $\varphi$, analogous to a minimal Weierstrass 
model for an elliptic curve.  For each prime $\fp = \fp_v$ of $\cO_H$, he defines an integer
\begin{equation*}
\varepsilon_\fp(\varphi) \ = \ \min_{\gamma \in \GL_2(H)} \ordRes_v(\varphi^{\gamma}) \ \ge \ 0 \ . 
\end{equation*} 
He then defines ``global minimal resultant'' of $\varphi$ to be the ideal  
\begin{equation*} 
\fR_\varphi \ = \ \prod_\fp \fp^{\varepsilon_\fp(\varphi)} \ .
\end{equation*} 
Here the product is finite since for a given representation $(F,G)$ of $\varphi$ over $H$,
for all but finitely many $\fp$ we have $\ord_\fp(\Res(F,G)) = 0$. 
  
Given a representation $(F,G)$ for $\varphi$
over $H$, in (\cite{Sil}, Proposition 4.99) Silverman shows there is a fractional ideal $\fa_{F,G}$ of $H$   
such that 
\begin{equation*}
\fR_\varphi \ = \ \left\{ \begin{array}{ll} 
                                 \fa_{F,G}^{2d} \cdot(\Res(F,G)) & \text{if $d$ is odd,} \\
                                 \fa_{F,G}^{d} \cdot (\Res(F,G)) & \text{if $d$ is even.} 
                          \end{array} \right.
\end{equation*}
Let $I(K)$ be the group of fractional ideals of $H$, and $P(K)$ the group of principal fractional ideals.
Silverman shows that if $d$ is odd, 
the ideal class $[\fa_\varphi] := [\fa_{F,G}] \in I(K)/P(K)$ 
is independent of the choice of $(F,G)$, while if $d$ is even, the refined ideal class  
$[\fa_\varphi] := [\fa_{F,G}] \in I(K)/\{(\alpha)^2 : (\alpha) \in P(K)\}$
is independent of the choice of $(F,G)$.
He calls $[\fa_\varphi]$ the {\em Weierstrass class of $\varphi$ over $H$}.

We will say that $\varphi$ has a {\em global minimal model over $H$} 
if for some $\gamma \in \GL_2(H)$, the function $\varphi^{\gamma}$ 
has an integral representation $(F^\gamma,G^\gamma)$ over $H$ such that  
\begin{equation*}
\ord_\fp\big(\Res(F^\gamma,G^\gamma)\big) 
\ = \ \varepsilon_\fp(\varphi) \quad \text{for each prime $\fp$ of $\cO_H$.} 
\end{equation*}
In (\cite{Sil}, Proposition 4.100), Silverman shows that if $\varphi$ has a global minimal model over $H$, 
then the Weierstrass class $\overline{\fa}_\varphi$ is trivial.
In (\cite{Sil}, Exercise 4.46) he asks

\begin{enumerate}

\item[(a)] When $H = \QQ$, does every $\varphi(z) \in \QQ(z)$ of degree $d \ge 2$ have a global 
minimal model over $\QQ$?

\item[(b)] When $H$ is an arbitrary number field and $\varphi(z) \in H(z)$ has degree $d \ge 2$, 
if $S$ is a finite set of primes of $\cO_H$
such that the localization $\cO_{H,S}$ is a Principal Ideal Domain, does $\varphi$ have a 
{\em global $S$-minimal model}?  In other words, is there a $\gamma \in \GL_2(H)$ 
such that $\varphi^\gamma$ has a representation $(F^\gamma,G^\gamma)$ 
with $F^\gamma(X,Y),G^\gamma(X,Y) \in \cO_{H,S}[X,Y]$, satisfying 
\begin{equation*}
\ord_\fp\big(\Res(F^\gamma,G^\gamma)\big) 
\ = \ \varepsilon_\fp(\varphi) \quad \text{for each prime $\fp \notin S$?} 
\end{equation*}

\item[(c)] When $H$ is an arbitrary number field and $\varphi(z) \in H(z)$ has degree $d \ge 2$, 
if the Weierstrass class $[\fa_\varphi]$ is trivial, does $\varphi$ have a global
minimal model over $H$? 

\end{enumerate}

As has already been noted by Bruin and Molnar (\cite{BM}), the answer to the first two questions is ``Yes''.
This follows from the Strong Approximation Theorem and the fact that the subgroup 
$\Aff_2(K) \subset \GL_2(K)$ acts transitively on the type II points in $\PP^1_\Berk$.
Indeed, in (b), let $\tS \supseteq S$ be a finite set of primes such that $\varphi$ has good reduction 
outside $\tS$.  For each prime $\fp = \fp_v \in \tS$, choose a $\gamma_\fp \in \GL_2(H)$ such that 
$\ordRes_v(\varphi^{\gamma_\fp}) = \varepsilon_\fp$ and put $\xi_\fp = \gamma_\fp(\zeta_G)$.  By 
Proposition \ref{RationalityEquiv}, $\xi_\fp \in \PP^1_{\Berk,v}$ is rational over $H$;  thus there 
exist $a_\fp, b_\fp \in H$ with $a_\fp \ne 0$, such that $\xi_\fp = \zeta_{D(b_{\fp},|a_\fp|_v)}$.  
Since $\cO_{H,S}$ is a PID there is an $a \in H$ such that $\ord_\fp((a)) = \ord_\fp((a_\fp))$ for each 
$\fp \in \tS$ and $\ord_\fp((a)) = 0$ for each $\fp \notin \tS$.  By the Strong Approximation Theorem
there is a $b \in H$ such that $\ord_\fp(b - b_\fp) > \ord(a_\fp)$ for each $\fp \in \tS$  
and $\ord_\fp(b) = 0$ for each $\fp \notin \tS$.  
Put $\gamma = \left[ \begin{array}{cc} a & b \\ 0 & 1 \end{array} \right]$;  then $\gamma(\zeta_G) = \xi_\fp$
for each $\fp \in \tS$ and $\gamma(\zeta_G) = \zeta_G$ for each $\fp \notin \tS$,
so $\ordRes_v(\varphi^\gamma) = \varepsilon_{\fp_v}$ for each prime $\fp_v$.  
Let $(F^{\gamma},G^\gamma)$ be a representation of $\varphi^\gamma$ over $H$;  since $\cO_{H,S}$
is a PID, we can assume $(F^{\gamma},G^\gamma)$ has been scaled so that 
$\min(\ord_{\fp}\big(F^\gamma),\ord_{\fp}(G^\gamma)\big) = 0$ for each $\fp \notin S$.   
Then $F^\gamma, G^\gamma$ are defined over $\cO_{H,S}$, 
and $\ord_\fp(\Res(F^{\gamma},G^\gamma)) = \varepsilon_\fp$ for each $\fp \notin S$, 
so $(F^{\gamma},G^\gamma)$ is a global $S$-minimal model.

\smallskip
The answer to question (c) is ``No'' in general.  The underlying reason for this is a disconnect
between the values of  $\ordRes_v(\cdot)$ and the points at which they are taken.
To obtain counterexamples, consider polynomials 
of the form $\varphi(z) = z^d + c$ with $d \ge 2$, $c \in H$.  For a given prime $\fp = \fp_v$ of $\cO_H$,
if $\ord_v(c) \ge 0$ then $\varphi(z)$ has good reduction at $\fp$.  Suppose $\ord_v(c) < 0$. 
Then $\ordRes_v(\varphi) = -2d\, \ord_v(c)$.  Computing $\ordRes_{\varphi,v}(\cdot)$ 
on the path $[0,\infty] \subset \PP^1_{\Berk,v}$, we find that for each $A \in \CC_v^{\times}$  
\begin{equation*}
\ordRes_{\varphi,v}\big(\zeta_{D(0,|A|_v)}\big) \ = \ 
\max\big((d-d^2) \ord_v(A), -2d \, \ord_v(c) + (d+d^2) \ord_v(A)\big) \ .
\end{equation*} 
This is minimal when $\ord_v(A) = (1/d) \ord_v(c)$.  If $(1/d) \ord_v(c)$ is not an integer, 
by convexity the least value of $\ordRes_{\varphi,v}(\cdot)$ on $H$-rational points in $\PP^1_{\Berk,v}$
occurs when $\ord_v(A)$ is one of the two integers adjacent to $(1/d) \ord_v(c)$.

\smallskip
For a counterexample when $d$ is odd, 
take $\varphi(z) = z^5 + 1/(1+ 4 \sqrt{-5})$, so $d = 5$ and $c = 1/(1+ 4 \sqrt{-5})$, with $H = \QQ(\sqrt{-5})$. 
The field $H$ has class number $2$.  
The ideal $\fp = \fp_v = (3,1+\sqrt{-5})$ in $\cO_H$ is one of the primes containing $(3)$;
it is not principal, but $\fp^4 = (1 + 4 \sqrt{-5})$, so $\ord_v(c) = -4$.  

Clearly $\varphi(z)$ has good reduction at all primes other than $\fp$. 
The least value of $\ordRes_{\varphi,v}(\cdot)$ on $H$-rational points 
occurs only when $\ord_v(A) = -1$, and one has 
\begin{equation*}
20 \ = \ \ordRes_{\varphi,v}\big(\zeta_{D(0,|A|_v)}\big) \ < \ \ordRes_v(\varphi) \ = \ 40 \ .
\end{equation*}
The integral representation $(F,G)$ with $F(X,Y) = X^5/c + Y^5$, $G(X,Y) = Y^5/c$ satisfies 
$(\Res(F,G)) = \fp^{40}$, while $\fR_\varphi = \fp^{20}$.  Since
$\fR_\varphi = \fa_{F,G}^{10} \cdot (\Res(F,G))$, 
it follows that $\fa_{F,G} = \fp^{-2} = (1/(2-\sqrt{-5}))$.  
Thus the class $[\fa_\varphi]$ is trivial.  However, there is no $\gamma \in \GL_2(H)$
for which $\ordRes_v(\varphi^\gamma) = \fR_\varphi$.  If there were, in $\PP^1_{\Berk,v}$ we would have 
$\gamma(\zeta_G) = \zeta_{D(0,3)}$, while for each finite place $w \ne v$, 
in $\PP^1_{\Berk,w}$ we would have $\gamma(\zeta_G) = \zeta_G$.  
By the proof of Proposition \ref{RationalityEquiv}, this would mean that  
$\ord_v(\det(\gamma)) = -1$ and $\ord_w(\det(\gamma)) = 0$ for all $w \ne v$,  
so $(\det(\gamma)) = \fp^{-1}$.  This is a contradiction since $\fp^{-1}$ is not principal.

\smallskip    
For a counterexample when $d$ is even, take $\varphi(z) = z^4 + 1/(19 + 4 \sqrt{-23})$, so $d = 4$ and 
$c = 1/(19 + 4 \sqrt{-23})$, with  $H = \QQ(\sqrt{-23})$. 
The field $H$ has class number $3$.
The ideal $\fp = \fp_v = (3,(1 + \sqrt{-23})/2)$ in $\cO_H$ is one of the primes containing $(3)$; 
it is  not principal, but $\fp^3 = (2 - \sqrt{-23})$ and $\fp^6 = (19 + 4 \sqrt{-23})$, so $\ord_v(c) = -6$.  

Clearly $\varphi(z)$ has good reduction at all primes other than $\fp$. 
The least value of $\ordRes_{\varphi,v}(\cdot)$ on $H$-rational points 
occurs only when $\ord_v(A) = -2$, and one has 
\begin{equation*}
24 \ = \ \ordRes_{\varphi,v}\big(\zeta_{D(0,|A|_v)}\big) \ < \ \ordRes_v(\varphi) \ = \ 48 \ .
\end{equation*}
The normalized representation $(F,G)$ with $F(X,Y) = X^4/c + Y^4$, $G(X,Y) = Y^4/c$ satisfies 
$(\Res(F,G)) = \fp^{48}$, while $\fR_\varphi = \fp^{24}$.  Since
$\fR_\varphi = \fa_{F,G}^{4} \cdot (\Res(F,G))$, 
it follows that $\fa_{F,G} = \fp^{-6} = (1/(2 - \sqrt{-23}))^2$.  
Thus the class $[\fa_\varphi]$ is trivial.  However, there is no $\gamma \in \GL_2(H)$
for which $\ordRes_v(\varphi^\gamma) = \fR_\varphi$.  If there were,  
we would have  $\ord_v(\det(\gamma)) = -2$ and $\ord_w(\det(\gamma)) = 0$ for all $w \ne v$,
so $(\det(\gamma)) = \fp^{-2}$.  This is impossible since $\fp^{-2}$ is not principal.    

\medskip
{\bf What is the dynamical significance of the Minimal Resultant Locus?}

\nopagebreak
When $\varphi$ has potential good reduction, the Minimal Resultant Locus consists of the unique
repelling fixed point of $\varphi$ in $\HH_{\Berk}$.  It is natural to ask about the dynamical significance of the 
Minimal Resultant Locus when $\varphi$ does not have potential good reduction.  

We do not know the answer to this.  The examples in \S2 show it does not always consist of fixed points. 
Rob Benedetto has remarked that another set which arises naturally in arithmetic dynamics, 
and is either a point or a segment, is the {\em Barycenter} of $\varphi$, defined to be the set of
points $Q \in \PP^1_\Berk$  which minimize the Arakelov-Green's function $g_\varphi(Q,Q)$ 
(see \cite{B-R}, \S10.2), and can be computed as the set of points $Q \in \HH_{\Berk}$ 
such that each component of $\PP^1_\Berk \backslash \{Q\}$
has mass at most $1/2$ for the invariant measure $\mu_{\varphi}$.  
(The author thanks Benedetto for pointing this out.)
In Example 2.3, when $p = 2$ the Barycenter is the segment $[\zeta_{D(0,1/2)},\zeta_{D(1,1/2)}]$ 
while the Minimal Resultant Locus is $\{\zeta_G\}$.  In Example 2.6, the Barycenter is $\{\zeta_{D(0,p)}\}$
while the Minimal Resultant Locus is $\{\zeta_{D(0,p^{4/3})}\}$.  Thus there is no clear relationship
between the Minimal Resultant Locus and the Barycenter.

Some other questions about the Minimal Resultant Locus, which may shed light on the 
general question of its dynamical meaning, are as follows:  

\begin{enumerate}
\item  How are the Minimal Resultant Loci of the iterates 
$\varphi, \varphi^{(2)}, \varphi^{(3)}, \cdots$ related?  
There are examples where the Minimal Resultant Loci of all the iterates are the same.  Does this happen 
in general?  If not, do they stabilize for large $n$, or converge to something with geometric significance?
Where do they lie relative to the Berkovich Julia set of $\varphi$?  

\smallskip
\item  Can one give a geometric description of the Minimal Resultant Locus?
This appears necessary in order to address stability questions of the kind above.

By Proposition \ref{GeneralTreeProp}, $\MinRes(\varphi)$ is contained in 
the intersection of the trees $\Gamma_{\Fix,\varphi^{-1}(a)}$ for all $a \in \PP^1(K)$. 
Recall that a {\em repelling fixed point} of $\varphi$ in $\HH_{\Berk}$
is a point $x \in \HH_{\Berk}$ such that $\varphi(x)= x$ 
and the degree of the reduction of $\varphi$ at $x$ is at least $2$
(see \cite{B-R}, p.340). 
In \cite{RR-GMR} the author shows

\begin{theorem} \label{ThmRSR2} The intersection of the trees 
$\Gamma_{\Fix,\varphi^{-1}(a)}$ for all $a \in \PP^1(K)$ 
is the tree $\Gamma_{\Fix,\Repel}$ spanned by the fixed points of $\varphi$ in $\PP^1(K)$ 
and the repelling fixed points of $\varphi$ in $\HH_{\Berk}$.  
\end{theorem} 
Where does the Minimal Resultant Locus lie in this tree?  Does it consist of points subject to some balance 
condition, like the one describing the Barycenter?  Is it possible to prune the tree still further?  
It seems plausible that the Minimal Resultant Locus might belong to the subtree spanned by 
the attracting and repelling fixed points of $\varphi$. 
%
%

%

\smallskip
\item What is the arithmetic significance of the value  of $\ordRes_{\varphi}(\cdot)$ 
on $\MinRes(\varphi)$?  It is clearly a conjugacy invariant which measures the 
complexity of $\varphi$.
\end{enumerate} 

\section{Algorithms} \label{AlgorithmsSection}

In this section we give two algorithms:  one which computes the Minimal Resultant Locus of $\varphi$,
and another which finds the $H$-rational points where $\ordRes_{\varphi}(\cdot)$ is minimal, 
in the case when $H$ is a local field and $\varphi$ is rational over $H$.

\medskip
Given $\varphi(z) \in K(z)$ with $d = \deg(\varphi) \ge 2$, put $a = \varphi(\infty) \in K \cup \{\infty\}$.  
The following algorithm finds the minimal value of $\ordRes_\varphi(\cdot)$ and determines $\MinRes(\varphi)$ 
by working in the tree $\Gamma_{\Fix,\varphi^{-1}(a)}$.  This tree is well suited for computations, 
because it is spanned by $\infty$ and the finite fixed points and finite solutions to $\varphi(z) = a$.  
This means the necessary changes of coordinates can be done with conjugacies by affine translations.  

\medskip
{\bf Algorithm A:  Minimize $\ordRes_\varphi(\cdot)$, find $\MinRes(\varphi)$, 

\qquad \qquad and find a $\gamma \in \GL_2(K)$ for which $\ordRes(\varphi^\gamma)$ is minimal.} 

\smallskip
Given a complete nonarchimedean valued field $K$ with absolute value $|x| = q^{-\ord(x)}$, 

\quad and a function $\varphi(z) \in K(z)$ with $d = \deg(\varphi) \ge 2$:
  
\begin{enumerate}
\item[(1)] [Find the endpoints of $\Gamma_{\Fix,\varphi^{-1}(a)}$.]

    \begin{enumerate}
          \item Write $\varphi(z) = f(z)/g(z)$ with $f(z), g(z) \in K[z]$ and put $a = \varphi(\infty)$.
          
          \item Find the roots of $f(z) - z g(z) = 0$ (the finite fixed points). 
                
          \item  If $a = \infty$, find the roots of $g(z) = 0$ (the finite poles). 
                 If $a \ne \infty$, 

                 find the roots of $f(z) - a \cdot g(z) =  0$ 
                 (the finite solutions to $\varphi(z) = a$).
         
          \item List the distinct roots from (b) and (c) as $\{\alpha_1, \ldots, \alpha_k\}$. 
    
    \end{enumerate} 

\item[(2)] [Minimize $\ordRes_\varphi(\cdot)$ on each path $[\alpha_i,\infty]$.]  

            For each $i = 1, \ldots, k$, do the following: 

    \begin{enumerate}
          \item Put $\gamma_i(z) = z + \alpha_i$. 
          \item Find a normalized representation $(F_i,G_i)$ 
                for $\varphi^{\gamma_i}(z) = \varphi(z + \alpha_i) - \alpha_i$.
          \item Compute $R_i = \ord\big(\Res(F_i,G_i)\big)$.
          \item Writing $F_i(X,Y)= a_d X^d + \cdots + a_0 Y^d$, $G_i(X,Y) = b_d X^d + \cdots + b_0 Y^d$,
          
               put $C_\ell = R_i - 2d \, \ord(a_\ell)$, $D_\ell = R_i - 2d \, \ord(b_\ell)$ 
               for $\ell = 0, \ldots, d$.
          
          \item Minimize the piecewise affine function 

                $\displaystyle{ \chi_i(t) = \max\big( \max_{0 \le \ell \le d}( C_\ell + (d^2+d - 2d \ell)t), 
                         \max_{0 \le \ell \le d} (D_\ell + (d^2+d - 2d (\ell+1)t) \big) }$.
                         
          \item Record the minimum value of $\chi_i(t)$ as $M_i$, and record the set of points where 
                it is achieved as a singleton $\{\zeta_{D(\alpha_i,r_i)}\}$ 
                or a segment $[\zeta_{D(\alpha_i,r_{i,1})}, \zeta_{D(\alpha_i,r_{i,2})}]$,
                where  $r = q^{-t}$ for a given $t$.

    \end{enumerate} 

\item[(3)] [Find the Minimum.]  Let $M = \min_{1 \le i \le k} M_i$,
                       output  ``$\min\big(\ordRes_\varphi(\cdot)\big) = M$''.
        
\item[(4)] [Find the Minimal Resultant Locus.]  Consider the indices $i$ with $M = M_i$\ : 

    \begin{enumerate}
    
           \item If for each such $i$, $\chi_i(t)$ achieved $M$ at a single point,     
           
                \qquad output ``$\MinRes(\varphi) = \{\zeta_{D(\alpha_i,r_i)}\}$'' for any such $i$, and go to (5).
                
           \item If for some such $i$, $\chi_i(t)$ achieved $M$ on a segment, 
           
                   \begin{enumerate}
                   
                     \item Find the relevant nodes of the tree $\Gamma_{\Fix,\varphi^{-1}(a)}$:  
 
                             \quad for all $(i,j)$ with $1 \le i < j \le k$ such that $M = M_i = M_j$, 

                             \qquad find $r_{ij} = |\,\alpha_i - \alpha_j|$, 
						then record $\zeta_{D(\alpha_i,r_{ij})}=\zeta_{D(\alpha_j,r_{ij})}$ as a node.
                           
                       \item Using the nodes, collate the segments 
                       $[\zeta_{D(\alpha_i,r_{i,1})}, \zeta_{D(\alpha_i,r_{i,2})}]$ 
                       
                      \quad into a single                 
                                  segment $[\zeta_{D(a,r_a)}, \zeta_{D(b,r_b)}]$, 
                       
                      \quad and output ``$\MinRes(\varphi) = [\zeta_{D(a,r_a)}, \zeta_{D(b,r_b)}]$''.

                   \end{enumerate} 
       
    \end{enumerate} 

\item[(5)]  [Find  $\gamma \in \GL_2(K)$ with $\ordRes(\varphi^{\gamma}) = M$.]  
 
    \begin{enumerate}
       \item Choose an endpoint of $\MinRes(\varphi)$  

             \qquad and write it as $\zeta_{D(B,|A|)}$  with $A \in K^{\times}$, $B \in K$.  

       \item Output ``$\gamma = \left[ \begin{array}{cc} A & B \\ 0 & 1 \end{array} \right]$'', then halt.
     \end{enumerate} 
     
\end{enumerate}

The correctness of Algorithm A follows from Theorem \ref{MainThm} and Proposition \ref{GeneralTreeProp}.

\medskip
When $\varphi(z) \in \QQ(z)$ and $|x| = |x|_p$ for a rational prime $p$, 
Algorithm A could be implemented 
either using arithmetic over global fields or over local fields.  
Working over global fields, one has the advantage 
of exact results, but care is needed to avoid coefficient explosion in intermediate steps. 
Over local fields, the implementation is more transparent and coefficient explosion does not occur, 
but careful error estimates are needed to assure that the results are correct.  
Below we sketch a possible implementation   
using arithmetic in global fields.  An implementation using local fields could be given  
using Theorem \ref{StabilityThm} and the factoring algorithm of Cantor and Gordon (\cite{CG}), 
which runs in probabilistic polynomial time and provides explicit error estimates for the precision needed.   
See also the factoring algorithms of Pauli (\cite{Pauli1}, \cite{Pauli2}) and the references therein.  

\smallskip
  
 Take $K = \CC_p$,  
and normalize the valuation $\ord(\cdot)$ on $\CC_p$ so it extends the valuation $\ord_p(\cdot)$  
on $\QQ$.  Let $\alpha_1, \ldots, \alpha_k$ be the roots from Step (1), 
and put $L = \QQ(\alpha_1, \ldots, \alpha_k)$.  Up to the action of $\Aut^c(\CC_p/\QQ_p)$, 
embeddings of $L$ in $\CC_p$ correspond to primes of $\cO_L$ above $p$.  Since $L/\QQ$ is galois,
it suffices to find one of those primes $\fp$, 
and work with the corresponding valuation $\ord_\fp(\cdot)$ on $L$.  
However, implementing Algorithm A does not require computing the full ring of integers $\cO_L$:   
it is enough to find a $p$-maximal order $\cO_{L,p} \subset \cO_L$ and a maximal ideal 
of that order lying over $(p)$.  It is beneficial to localize at $p$, 
and work over $\ZZ_{(p)}$ rather than $\ZZ$:  the localization $\cO_{L,(p)}$ of $\cO_{L,p}$,  
which is the integral closure of $\ZZ_{(p)}$ in $L$, is a PID.  
Finally, the computations  for Algorithm A need not be done in  $L$:  
they can be carried out in the subfields $L_i = \QQ(\alpha_i)$ and $L_{ij} = \QQ(\alpha_i,\alpha_j)$,
working with the restriction of $\ord_\fp(\cdot)$ to those fields.      

Since $\alpha_1, \ldots, \alpha_k$, $L$, and $\fp$ are not known in advance,    
one can proceed as follows.  Put $P(x) = f(x) - x g(x)$, and put $Q(x) = g(x)$ or $Q(x) = f(x) - a \, g(x)$ 
according as $a = \varphi(\infty)$ is infinite or finite. 
Let $f_1(x), \ldots, f_r(x)$ be the distinct monic irreducible factors of $P(x)$ and $Q(x)$,   
so $\alpha_1, \ldots, \alpha_k$ are the roots of $f_1(x), \ldots, f_r(x)$.  
For each $j = 1, \ldots, r$, put $\tL_j = \QQ[x]/(f_j(x))$ and let $\talpha_j$ be the image of $x$ in $\tL_j$.  
Find the maximal ideals $\widetilde{\fp}_{j\ell}$ of $\cO_{\tL_j,(p)}$ 
and the corresponding valuations $\ord_{\widetilde{\fp}_{i\ell}}(\cdot)$.  Carry out Step (2) 
of Algorithm A for each pair $(\talpha_j, \ord_{\widetilde{\fp}_{i\ell}}(\cdot))$.  
Up to conjugacy, this is equivalent to carrying out Step (2) for the roots $\alpha_i$ of $f_j(x)$ 
and the valuation $\ord_{\fp}(\cdot)$.  

The minimization of $\chi_i(t)$ in Step (2e) 
can be done crudely in $O(d^3)$ steps by computing the intersection points of each pair of affine functions, 
and comparing the values of the functions at those points.  It could be done more efficiently by first 
finding highest of the functions with given slope $m \equiv d^2 + d \pmod{2d}$, 
then solving for the intersection point of the 
functions with slopes $d^2 + d$ and $-d^2 - d$ and comparing function values at that point, and continuing
on with a binary search. 

If in Step (4a) the Minimal Resultant Locus turns out to be a single point 
(in particular if $d$ is even) the algorithm terminates. 
However, if  the Minimal Resultant Locus is a segment, it must either have the form 
$[\zeta_{D(\alpha_i,r_{i,1})}, \zeta_{D(\alpha_i,r_{i,2})}]$ for some $i$, 
or $[\zeta_{D(\alpha_i,r_{i,1})}, \zeta_{D(\alpha_i,r_{i,2})}] 
\cup [\zeta_{D(\alpha_j,r_{j,1})}, \zeta_{D(\alpha_j,r_{j,2})}]$ for some $i$ and $j$, where $r_{i,2} = r_{j,2}$
and the segments are disjoint except for their upper endpoint.  To carry out Step (4b) one should 
first find the segments $[\zeta_{D(\alpha_i,r_{i,1})}, \zeta_{D(\alpha_i,r_{i,2})}]$ for which $r_{i,2}$ 
is maximal, and among those, choose one for which $r_{i,1}$ is minimal.
The corresponding segment $[\zeta_{D(\alpha_i,r_{i,1})}, \zeta_{D(\alpha_i,r_{i,2})}]$ will either be the entire 
Minimal Resultant Locus, or one leg of it.  Suppose this segment came 
from the pair $(\talpha_1, \ord_{\widetilde{\fp}_{1,1}}(\cdot))$ and the field $\tL_1$.  
One should then factor $f_1(x), \ldots, f_r(x)$ over $\tL_1[x]$, and for each irreducible
factor $f_{i,h}(x)$ (except the linear factor $x - \talpha_1$ of $f_1(x)$) 
one should form the field $\tL_{1,i,h} = \tL_1[x]/(f_{i,h}(x))$, find the maximal ideals of 
$\cO_{\tL_{1,i,h},(p)}$ lying over $\widetilde{\fp}_{1,1}$ and carry out Step (2) again 
for these fields and valuations.  One can then determine the relevant nodes of $\Gamma_{\Fix,\varphi^{-1}(a)}$, 
and complete Step (4b) by using them to decide whether the Minimal Resultant Locus has one leg or two.

The author has not carried out 
a detailed running time analysis of this procedure 
(which would be lengthy, and tangential to the purposes of the paper), but 
using the standard number-theoretic algorithms below it is evident that it 
could be implemented to run in probabilistic polynomial time.

Lenstra, Lenstra and Lovasz (\cite{LLL}) showed that a polynomial $h(z) = a_0 + a_1 z + \cdots + a_n z^n \in \QQ[z]$
can be deterministically factored over $\QQ$ in $O(n^{12} + n^9 \log|h|)$ bit operations, where 
$|h| = (\sum_i |a_i|^2)^{1/2}$.  A.K. Lenstra (\cite{AKL}) proved an analogous result for     
polynomials over a number field.  A result of Mignotte (\cite{Mig},  see for example Cohen \cite{Coh}, \S3.5.1)
assures that the lengths of the  coefficients of the factors are polynomially bounded in terms of the input.  
If $F = \QQ(\beta)$ is a number field, where $\beta$ is an algebraic integer, 
standard methods for finding $\cO_F$ such as Zassenhaus's Round Two algorithm 
(see \cite{Coh}, \S6.1) involve factoring the discriminant of the minimal polynomial of $\beta$, 
and then using linear algebra to successively enlarge the order $\ZZ[\beta]$ 
to be $q$-maximal at each prime $q$ dividing the discriminant.  
There is no known polynomial time algorithm for factoring integers,  
but Zassenhaus's algorithm can compute a $p$-maximal order $\cO_{F,p} \subset \cO_F$  
without factoring the discriminant.  Since the discriminant is known,
using (\cite{Coh}, Algorithm 2.4.6) the linear algebra computations can be done without coefficient explosion.
A $\ZZ$-basis for $\cO_{F,p}$ gives an integral basis for $\cO_{F,(p)}$ over $\ZZ_{(p)}$;  
thus the algorithm of Buchmann-Lenstra (see \cite{Coh}, \S6.2) 
can be used to find the maximal ideals $\fp$ of $\cO_{F,(p)}$ above $(p)$.
This involves carrying out a series of matrix computations over $\FF_p$.  
Given $0 \ne x \in \cO_F$, the standard way to compute $\ord_{\fp}(x)$ 
is to find an element $\beta \in \fp^{-1} \backslash \cO_F$, and then determine the largest 
integer $N$ such that $\beta^N x \in \cO_F$ (see Cohen \cite{Coh}, \S4.8.3).
However, this can equally well be done over $\cO_{F,(p)}$.       
The algorithms of Zassenhaus and Buchmann-Lenstra 
run in probabilistic polynomial time;  the probabilistic aspect comes from the need to 
factor polynomials over finite fields.  Using  Berlekamp's algorithm (\cite{Ber}) 
polynomials of degree $\ell$ in $\FF_q[x]$ can be factored in  
probabilistic polynomial time $O(\ell^3 \log(q)^3)$;  improvements have been given by 
Cantor-Zassenhaus (\cite{CZ}), Kaltofen-Shoup (\cite{KS}), and others.  

With suitable modifications, the procedure outlined above could be generalized to 
rational functions $\varphi(z)$  over arbitrary global fields, and should still run in probabilistic polynomial time.
This uses that  
polynomials over a global field can be factored in polynomial time, 
as shown by Pohst and Oman\~{a} (\cite{Pohst}, \cite{PohstOm}) 
and Belabas, van Hoeij, Kl\"uners and Steel (\cite{BvHKS}).

\smallskip
Now let $H_v$ be a local field. 
Suppose  $\varphi(z) \in H_v(z)$ has degree $d \ge 2$, and take $K = \CC_v$.  
Below, we give a ``steepest descent'' algorithm for finding an $H_v$-rational type II point 
where $\ordRes_\varphi(\cdot)$ takes its least value.  
Working within $H_v$, the algorithm finds the $H_v$-minimum for $\ordRes_\varphi(\cdot)$
and a $\gamma \in \GL_2(H_v)$ which achieves it.  
The algorithm also decides whether the $H_v$-minimum is the absolute minimum. 

The algorithm takes the path of steepest descent towards $\MinRes(\varphi)$, starting at $\zeta_G$. 
The path necessarily begins with a segment going ``upward'' from $\zeta_G$ towards $\infty$
(this segment may have length $0$) to some point $\zeta_{D(0,R)} = \zeta_{D(a,R)}$,
then goes ``downward'' from $\zeta_{D(a,R)}$ to a point $\zeta_{D(a,r)} \in \MinRes(\varphi)$.  
By Proposition \ref{SteepestDescentProp}, at any $H_v$-rational type II point outside $\MinRes(\varphi)$,
$\MinRes(\varphi)$ lies in a direction coming from the tree spanned by $\PP^1(H_v)$. 
Thus the path of steepest descent 
runs along the tree spanned by $\PP^1(H_v)$ until it either reaches a point
in $\MinRes(\varphi)$, or branches off that tree between two $H_v$-rational type II points,
one of which will minimize $\ordRes_\varphi(\cdot)$ on $H_v$-rational type II points.  
The algorithm steps between $H_v$-rational type II points and   
stops when an $H_v$-rational type II point minimizing $\ordRes_\varphi(\cdot)$ is reached. 

We will assume the residue field $k_v = \cO_v/\fM_v$ is isomorphic to $\FF_q$, and that $\ord(\cdot)$
is normalized so that $\ord(\pi)=1$ for a uniformizer $\pi$ for $\fM_v$.  
Given $a \in \cO_v$, we write $\abar = a \pmod{\fM_v} \in \FF_q$ for the residue class of $a$.

\medskip
{\bf Algorithm B:  Minimize $\ordRes_\varphi(\cdot)$ on $H_v$-rational type II points.}


\smallskip
Given a nonarchimedean local field $H_v$, and $\varphi(z) \in H_v(z)$ with $d = \deg(\varphi) \ge 2$:
  
\begin{enumerate}
\item[(1)] [Initialize.] 

    \begin{enumerate}
          \item Find a normalized representation $(F,G)$ for $\varphi$.
          
          \item Compute $R = \ord(\Res(F,G))$. 
         
          \item Set $\gamma = \left[ \begin{array}{cc} 1 & 0 \\ 0 & 1 \end{array} \right]$.  
    
          \item Fix an element $\pi \in H_v$ with $\ord(\pi) = 1$.  
    \end{enumerate} 

\item[(2)] [First, go up from $\zeta_G$ towards $\infty$.] 

    \begin{enumerate}
          \item[(a)] [See if $\ordRes_\varphi(\cdot)$ is locally decreasing in the direction $\vv_\infty$.] 
          
             Write $F(X,Y) = a_d X^d + a_{d-1}X^{d-1}Y + \cdots + a_0 Y^d$,  
          
                   \qquad \quad \ \ $G(X,Y) = b_d X^d + b_{d-1}X^{d-1}Y + \cdots + b_0 Y^d$,
                   
          \qquad then using the criterion from Lemma  \ref{CriterionLemma}, 
       
          \qquad \quad test whether \quad 
             $\left\{ \begin{array}{ll} \abar_\ell = 0 & \text{for $\frac{d+1}{2} \le \ell \le d$, } \\
                                   \bbar_\ell = 0 & \text{for $\frac{d-1}{2} \le \ell \le d$.}
                      \end{array} \right.$  
           
                       \quad If not, $\ordRes_\varphi(\cdot)$ 
                           is not decreasing in the direction $\vv_\infty$; go to (3). 
               
                      \quad If so, $\vv_\infty$ is the unique direction in which
							$\ordRes_\varphi(\cdot)$ is decreasing; 

                       \qquad \qquad continue on to (3b).

          \item [(b)] [Compute how far to go up.]
          
                         \quad For $\ell = 0, \ldots, d$, put $C_\ell = R - 2d \, \ord(a_\ell)$, 
                                                           $D_\ell = R - 2d \, \ord(b_\ell)$, 

                         \qquad then minimize the piecewise affine function 

                       $\displaystyle{ \chi(t) = \max\big( \max_{0 \le \ell \le d}( C_\ell + (d^2+d - 2d \ell)t), 
                         \max_{0 \le \ell \le d} (D_\ell + (d^2+d - 2d (\ell+1)t) \big) }$. 
                         
                         \quad Let $R_{\new}$ be the minimum value of $\chi(t)$, 

                          \qquad and let $[M,N]$ be the subset of $\RR$ on which it is attained 
                                      
                          \qquad \quad (necessarily $N < 0$; possibly $M = N$).

          \item [(c)] [Test the nature of the minimum.]

                $(\rm{i})$ If $M = N \in \ZZ$, the new minimum is at an $H_v$-rational type II point: 
                
                    \qquad \qquad \quad  take a step up to that point.
                
                     \qquad \quad Put $\eta = \left[ \begin{array}{cc} 1 & 0 \\ 0 & \pi^{-N} \end{array} \right]$  
                           and find a normalized representation for $(F^\eta,G^\eta)$: 
             
             \quad \ Let $F_*(X,Y) = \pi^{-N} a_d X^d + \pi^{-2N} a_{d-1}X^{d-1}Y + \cdots + \pi^{-(d+1)N} a_0 Y^d$,  
          
                   \qquad \quad \ \ $G_*(X,Y) = b_d X^d + \pi^{-N} b_{d-1}X^{d-1}Y + \cdots + \pi^{-d N} b_0 Y^d$,
                   
             \qquad \ \ then normalize $(F_*,G_*) \rightarrow (F_{\new},G_{\new})$,

             \qquad \qquad  update $F \leftarrow F_{\new}$, $G \leftarrow G_{\new}$, $R \leftarrow R_{\new}$,
                     $\gamma \leftarrow \eta$, 
                           and go to (3).

                $(\rm{ii})$ If $M = N \notin \ZZ$, or if $[M,N]$ is an interval which contains no integers, 

                  \qquad \qquad the $H_v$-minimum for $\ordRes_\varphi(\cdot)$ is not the absolute minimum: 
                
               \qquad \ \ let $m$, $n$ be the two integers bracketing $[M,N]$, 
               
               \qquad \quad put $R = \min(\chi(m), \chi(n))$  and let $k \in \{m,n\}$ be a point where 

               \qquad \qquad the minimum is attained;  
                       put $\gamma = \left[ \begin{array}{cc }1 & 0 \\ 0 & \pi^{-k} \end{array} \right]$,  
                                      and  go to (4a).
               
                $(\rm{iii})$ If $[M,N]$ is an interval of positive length containing an integer $k$, 
                
                    \qquad \qquad the  new minimum for $\ordRes_\varphi(\cdot)$ is the absolute minimum: 
                
               \qquad \qquad \quad  put $R = R_{\new}$,     
                          put $\gamma = \left[ \begin{array}{cc} 1 & 0 \\ 0 & \pi^{-k} \end{array} \right]$, 
                                        and  go to (4b).

    \end{enumerate} 

\item[(3)] [Iterate steps down, until the $H_v$-minimum is reached.]

     \begin{enumerate} 
        \item[(a)] [Limit the possible directions towards $\MinRes(\varphi)$.]
        
                    Put $g(z) = G(z,1)$ and $h(z) = F(z,1)-zG(z,1)$, 
                         
                      \quad \qquad then find the common roots of $\gbar(z)$ and $\hbar(z)$  
                                     belonging to $\FF_q$.                 
                    
                    \quad If there are no such roots, the current $R$ is minimal:  go to (4b).
                    
                    \quad If there are common roots, list them as $\{\betabar_1, \ldots, \betabar_k\}$, 
                          and continue to (3b).
        
        \item[(b)] [Find the direction of steepest descent.]
        
                   For each $i = 1, \ldots, k$, do the following:

          \begin{enumerate} 
                
               \item[(i)] Let $\beta_i \in \cO_v$ be a lift of $\betabar_i$; 
                    change coordinates by $\left[ \begin{array}{cc} 1 & \beta_i \\ 0 & 1 \end{array} \right]$,
 
                   \qquad putting $F_*(X,Y) = F(X + \beta_i Y, Y) - \beta_i G(X + \beta_i Y, Y)$, 
  
                   \qquad \qquad \qquad  $G_*(X,Y) =  G(X + \beta_i Y, Y)$.           
                         
            \item[(ii)] Write $F_*(X,Y) = a_d X^d + a_{d-1}X^{d-1}Y + \cdots + a_0 Y^d$,  
          
                   \qquad \quad \ \ $G_*(X,Y) = b_d X^d + b_{d-1}X^{d-1}Y + \cdots + b_0 Y^d$,
                   
                    \qquad then using the criterion from Lemma  \ref{CriterionLemma}, 
       
          \qquad \quad  test whether \quad 
             $\left\{ \begin{array}{ll} \abar_\ell = 0 & \text{for $0 \le \ell \le \frac{d+1}{2}$, } \\
                                   \bbar_\ell = 0 & \text{for $0 \le \ell \le \frac{d-1}{2}$.}
                      \end{array} \right.$

                      \qquad If so, $\vv_{\beta_i}$ is the unique direction in which
							$\ordRes_\varphi(\cdot)$ is decreasing; 

                       \qquad \qquad exit the loop on $i$, and go to (3c).

                      \qquad If not, continue the loop and take the next value of $i$.
                                        
          \end{enumerate}  
          
          \qquad If there are no more values of $i$,  the current $R$ is minimal;  go to (4b).

   \item [(c)] [Compute how far to go down.]
          
                         \quad For $\ell = 0, \ldots, d$, put $C_\ell = R - 2d \, \ord(a_\ell)$, 
                                                           $D_\ell = R - 2d \, \ord(b_\ell)$, 

                         \qquad then minimize the piecewise affine function 

                       $\displaystyle{ \chi(t) = \max\big( \max_{0 \le \ell \le d}( C_\ell + (d^2+d - 2d \ell)t), 
                         \max_{0 \le \ell \le d} (D_\ell + (d^2+d - 2d (\ell+1)t) \big) }$. 
                         
                         \quad Let $R_{\new}$ be the minimum value of $f(t)$, 

                          \qquad and let $[M,N]$ be the subset of $\RR$ on which it is attained 
                                      
                          \qquad \quad (necessarily $M > 0$; possibly $M = N$).

          \item [(d)] [Test the nature of the minimum.]

                $(\rm{i})$ If $M = N \in \ZZ$,  the new minimum is at an $H_v$-rational type II point: 
                
                           \qquad \qquad \quad  take a step down to that point.
                
                   \qquad \quad   Put $\eta = \left[ \begin{array}{cc} \pi^N & \beta_i \\ 0 & 1 \end{array} \right]$
                                                 and find a normalized representation for $(F^\eta,G^\eta)$: 
             
             \quad \ \ Let $F_{**}(X,Y) = F_*(\pi^N X, Y)$,  $G_{**}(X,Y) =  \pi^N G_*(\pi^N X,Y)$,
                   
             \qquad \ \ then normalize $(F_{**},G_{**}) \rightarrow (F_{\new},G_{\new})$, 
                       
             \qquad \qquad  update $F \leftarrow F_{\new}$, $G \leftarrow G_{\new}$, $R \leftarrow R_{\new}$,
                     $\gamma \leftarrow \gamma \cdot \eta$, 
                           and go to (3).

                $(\rm{ii})$ If $M = N \notin \ZZ$, or if $[M,N]$ is an interval which contains no integers,  
                
                \qquad \qquad the $H_v$-minimum for $\ordRes_\varphi(\cdot)$ is not the absolute minimum: 
                
               \qquad \ \ let $m$, $n$ be the two integers bracketing $[M,N]$, 
               
               \qquad \quad put $R = \min(\chi(m), \chi(n))$  and let $k \in \{m,n\}$ be a point where 

               \qquad \qquad the minimum is attained;  
                  put $\gamma = \gamma \cdot \left[ \begin{array}{cc} \pi^k & \beta_i \\ 0 & 1 \end{array} \right]$,  
                                     and  go to (4a).
               
                $(\rm{iii})$ If $[M,N]$ is an interval of positive length containing an integer $k$,

                \qquad \qquad the new minimum for $\ordRes_\varphi(\cdot)$ is the absolute minimum:   
                
               \qquad \qquad \quad  put $R = R_{\new}$,    
             put $\gamma = \gamma \cdot \left[ \begin{array}{cc} \pi^k & \beta_i \\ 0 & 1 \end{array} \right]$,
                                    and go to (4b).          
          
       \end{enumerate}  

\item[(4)] [Output whether the $H_v$-minimum for $\ordRes_\varphi(\cdot)$ is the absolute minimum.]

           (a)  Output ``The $H_v$-minimum for $\ordRes_\varphi(\cdot)$ is not the absolute minimum'', 
               
               \qquad \qquad  and  go to (5).

           (b) Output ``The $H_v$-minimum for $\ordRes_\varphi(\cdot)$ is the absolute minimum'', 
               
               \qquad \qquad  and  continue on to (5).          
          
\item[(5)] [Output $R$ and $\gamma$, and halt.]  

        Output ``The minimal value of $\ordRes_\varphi(\cdot)$ on $H_v$-rational type II points  $=  R$''; 
        
        Output ``$\gamma = \scriptstyle{ \left[ \begin{array}{cc} a & b \\ 0 & d \end{array} \right] }$ 
                   is a matrix for which $\ordRes(\varphi^\gamma) = R$''.  Halt.

\end{enumerate}

The correctness of the algorithm and the fact that it terminates follow 
from Theorem \ref{MainThm}, Lemma \ref{CriterionLemma} 
and Proposition \ref{SteepestDescentProp}, but perhaps some remarks are in order. 

\smallskip 
After each step to a new $H_v$-rational type II point, 
the algorithm changes coordinates to bring that point back to $\zeta_G$.  
The corresponding coordinate change matrices are affine, 
so they preserve the direction $\vv_\infty$.  This means that in Lemma \ref{CriterionLemma} 
we can use the tangent directions $\vv_\infty$, $\vv_\beta$ at $\zeta_G$, rather than the 
tangent directions $\vv_{Q,\eta(\infty)}$, $\vv_{Q,\eta(\beta)}$ at $Q = \eta(\zeta_G)$.

If the path of steepest descent 
branches off the tree spanned by $\PP^1(H_v)$, when the algorithm moves between the two 
$H_v$-rational type II points adjacent to $\MinRes(\varphi)$, $\ordRes_{\varphi}(\cdot)$
will initially decrease, then increase.  The stopping criteria in Steps (2c), (3a), (3b) and (3d)
assure that a point where the minimum is taken is chosen.

In Step (3a), it cannot be that both $\gbar(z) \equiv 0$ and $\hbar(z) \equiv 0$, 
as in that case $\fbar(z) = \hbar(z) + z \gbar(z) \equiv 0$.  Since the coefficients
of $f$ and $g$ are the same as those of $F$ and $G$ respectively, this would mean  
$(F,G)$ was not normalized, contrary to its construction.  
To motivate Step (3a), note that in Step (3b), $g(\beta_i)$ is the coefficient of $Y^d$ in $G_*(X,Y)$
and $h(\beta_i)$ is the coefficient of $Y^d$ in $F_*(X,Y)$.  
If $\overline{g(\beta_i)} \ne 0$, then in Step (3b) the coefficient $\bbar_0$ 
would be nonzero, and the test in Step (3b) would fail.  Likewise, if $\overline{h(\beta_i)} \ne 0$,
then $\abar_0 \ne 0$, and again the test would fail.

In Step (3c), the matrix 
$\eta = \left[ \begin{array}{cc} \pi^N & \beta_i \\ 0 & 1 \end{array} \right]$ 
makes the step from $\zeta_G$ to the $H_v$-rational type II point $\zeta_{D(\beta_i,|\pi|^N)}$.  
This coordinate change is realized as the composite of two partial steps, using    
$\eta = \left[ \begin{array}{cc} 1 & \beta_i \\ 0 & 1 \end{array} \right]
\cdot \left[ \begin{array}{cc} \pi^N & 0 \\ 0 & 1 \end{array} \right]$.  

\smallskip
Algorithm B is content with finding one point where the $H_v$-minimum is attained.  
By incorporating additional tests and an extra search  
based on the criteria in Lemma \ref{CriterionLemma}(B), it could 
easily be modified to find all $H_v$-rational type II points where the $H_v$-minimum was attained.  
We leave this modification to the reader.

\smallskip
In implementing Algorithm B it is not necessary to work in a local field. If $\varphi(z)$ is defined over 
a number field $H$, and $\ord_v(\cdot)$ is a nonarchimedean valuation of $H$ (specified, for example, 
by giving a $p$-maximal order $\cO_{H,p} \subset \cO_H$ and a prime ideal $\fp_v$ of $\cO_{H,p}$ above $p$),
one could carry out the algorithm using computations in $H$ using ideas similar to those discussed 
in Algorithm A.   

We will now discuss its running time when $H = \QQ$ and $v = p$ is a rational prime.  
For simplicity, assume that $\varphi(x) = f_0(x)/g_0(x)$ is the quotient of relatively prime
polynomials $f_0(x), g_0(x) \in \ZZ[x]$, where the coefficients of $f_0$ and $g_0$ have absolute
value at most $B$.  Let $(F_0,G_0)$ be the initial normalized representation of $\varphi$ from Step (1a), 
and let $R_0 = \ord_p(\Res(F_0,G_0))$ be the ord-value of its resultant, computed in Step (1b).  
The Hadamard bound for the archimedean size of $\Res(F_0,G_0)$ is $(d+1)^d B^{2d}$, so
\begin{equation*}
R_0  \ \le \ d \log_p(d+1) + 2d \log_p(B) \ .
\end{equation*} 
Each time Step 2 or Step 3 is executed, the distance from $\zeta_G$ to the $\QQ_p$-rational type II point 
being considered increases by at least $1$, so by Theorem \ref{MainThm}, 
the algorithm terminates after at most $\frac{2}{d-1} R_0$ passes through Steps 2 and 3.  
At all intermediate stages, the coefficients of $F$ and $G$ remain in $\ZZ$;
by Theorem \ref{StabilityThm}, it suffices to compute them
modulo $p^{4 R_0}$, and as the algorithm proceeds, the required precision decreases.  
Step (3a) limits the number of residue classes considered in Step (3b) to at most $d+1$;
using Berlekamp's algorithm Step (3a) can be carried out in $\cO( d^3 \log(p)^3)$ bit operations.   
From these considerations one sees that Algorithm B runs in probabilistic polynomial time.

\section{The case $d = 1$} \label{d=1Section}

For completeness, in this section we consider $\ordRes_\varphi(\cdot)$ when $d = 1$, 
that is, when $\varphi(z) =  \frac{f_1z+f_0}{g_1z+g_0} \in K(z)$ with $f_1 g_0- f_0 g_1 \ne 0$.
It is no longer true that $\MinRes(\varphi)$ is a point or a segment of finite path-length:  
the reason for the difference is the simple fact that $1^2-1 = 0$, whereas $d^2 - d > 0$ when $d \ge 2$.  

As is well known, there are three cases to consider:
\begin{enumerate}
\item $\varphi(z) = z$;

\item $\varphi(z)$ has (exactly) two distinct fixed points, in which case there 
are a $\gamma \in \GL_2(K)$ and a $C \in K^{\times}$
with $|\,C| \le 1$ and $C \ne 1$ such that $\varphi^{\gamma}(z) = Cz$;

\item $\varphi(z)$ has a single fixed point of multiplicity $2$, in which case there are a  
$\gamma \in \GL_2(K)$ and a $0 \ne C \in K$ such that $\varphi^{\gamma}(z) = z + C$.
\end{enumerate}  
By standard computations in linear algebra, 
it is easy to distinguish between the cases, 
and to find a $\gamma$ which carries out the desired conjugacy:  
the second case occurs when the Jordan form of the matrix corresponding to $\varphi$ is 
$\left[ \begin{array}{cc} \lambda & 0 \\ 0 & \mu \end{array} \right]$ with $\lambda \ne \mu$, 
and the eigenvalues are ordered so that $|\lambda| \le |\mu|$;
the third case when it is $\left[ \begin{array}{cc} \lambda & 1 \\ 0 & \lambda \end{array} \right]$.
In the second case $C = \lambda/\mu$, in the third case $C = 1/\lambda$.
If $\varphi$ and the eigenvalues are rational over a subfield $H \subset K$, 
then $\gamma$ can be chosen to 
belong to $\GL_2(H)$. 

We will need some terminology.  Given points $x_0 \ne x_1 \in \PP^1(K)$, the {\em strong tube of radius $R$
around the path $[x_0,x_1]$} is the set
\begin{equation*}
T_{[x_0,x_1]}(R) \ = \ [x_0,x_1] \cup \{ z \in \HH_\Berk : \text{$\rho(z,x) \le R$ for some $x \in [x_0,x_1]$} \} \ .
\end{equation*}
If $z \in \PP^1_{\Berk}$  corresponds to a sequence of nested discs $\{D(a_i,r_i)\}_{i \ge 1}$
by Berkovich's classification theorem (see \cite{B-R}, p.5), 
we define $\diam_\infty(z) =  \lim_{i \rightarrow \infty} r_i$;  we put $\diam_\infty(\infty) = \infty$.   
The {\em horodisc of codiameter $R$, tangent to the point $\infty$}, is the set 
\begin{equation*}
H_{\infty}(R) \ = \ \{ \ z \in \PP^1_{\Berk} : \diam_{\infty}(z) \ge R \ \} \ .
\end{equation*}
The only type I point belonging to $H_{\infty}(R)$ is $\infty$;  
a point $\zeta_{D(a,r)}$ of type II or  III belongs to $H_{\infty}(R)$ if and only if $r \ge R$.  
For each $a \in K$, the intersection of the path $[a,\infty]$ with $H_{\infty}(R)$ 
is the ray $[\zeta_{D(a,R)},\infty]$. 
For each $S > R$, the point $\zeta_{D(0,S)}$ belongs to $H_{\infty}(R)$;  
if $a \in K$ and $|a| \le S$, the intersection of 
$[a,\zeta_{D(0,S)}]$ with $H_\infty(R)$ is 
\begin{equation*}
\{\ z \in [a,\zeta_{D(0,S)}] : \rho(\zeta_{D(0,S)},z) \le \log(S/R) \ \} \ .
\end{equation*} 
Thus $H_\infty(R)$ can be described informally as ``the set of points in $\PP^1_{\Berk}$ accessible
by moving the ray $[\zeta_{D(0,R)}, \infty]$ without stretching, keeping it anchored at $\infty$''.
For an arbitrary $x_0 \in \PP^1(K)$, a {\em horodisc tangent to $x_0$} is a set of the form
$\gamma(H_\infty(R))$ for some $R$, where $\gamma \in \GL_2(K)$ is such that $\gamma(\infty) = x_0$.

\begin{theorem} \label{d=1Thm}  Suppose $\varphi(z) \in K(z)$ has degree $d = 1$.  
The function $\ordRes_{\varphi}(\cdot)$ on type {\rm II} points 
extends to a function $\ordRes_{\varphi} :\PP^1_{\Berk} \rightarrow [0,\infty]$
which is piecewise affine and convex upwards on each path in $\PP^1_{\Berk}$,
with respect to the logarithmic path distance.
It is finite and continuous on $\HH_{\Berk}$ with respect to the strong topology, 
and achieves its minimum on a nonempty set $\MinRes(\varphi) \subset \PP^1_{\Berk}$. Furthermore     

$(1)$ If $\varphi(z) = z$, then $\ordRes_{\varphi}(\cdot) \equiv 0$
 and $\MinRes(\varphi) = \PP^1_{\Berk}$.

$(2)$ If $\varphi(z)$ has exactly two fixed points $x_0, x_1$, 
let $\gamma \in \GL_2(K)$ and $C \in K^{\times}$ with $|\,C| \le 1$, $C \ne 1$,
be such that $\varphi^{\gamma}(z) = Cz$.
The minimal value of $\ordRes_{\varphi}(\cdot)$ is $\ord(C)$,
and $\varphi$ has potential good reduction if and only if $|\,C| = 1$.  
When $|\,C| < 1$, or when $|\,C| = 1$ and $|\,C-1| = 1$, 
then $\MinRes(\varphi)$ is the path $[x_0,x_1]$.
When $|\,C| = 1$ and $|\,C-1| < 1$, put $R = \ord(C-1) > 0;$  
then $\MinRes(\varphi)$ is the strong tube $T_{[x_0,x_1]}(R)$.
The function $\ordRes_{\varphi}(\cdot)$ 
takes the value $\infty$ at each point of \,$\PP^1(K) \backslash \{x_0,x_1\}$,
and is continuous on $\PP^1_{\Berk} \backslash \{x_0,x_1\}$ relative to the strong topology.

$(3)$ If $\varphi(z)$ has one fixed point $x_0$, 
let $\gamma \in \GL_2(K)$ and $0 \ne C \in K$ be such that $\varphi^{\gamma}(z) = z + C$.
Then the minimal value of $\ordRes_{\varphi}(\cdot)$ is \,$0$
and $\varphi$ has potential good reduction.
Put $R = |\,C|$.  Then $\MinRes(\varphi)$ is the horodisc tangent to $x_0$ given by $\gamma(H_\infty(R))$.  
The function $\ordRes_{\varphi}(\cdot)$ takes the value $\infty$ 
at each point of \,$\PP^1(K) \backslash \{x_0\}$,
and is continuous on $\PP^1_{\Berk} \backslash \{x_0\}$ relative to the strong topology.
\end{theorem}  

\begin{proof}
The fact that $\ordRes_{\varphi}(\cdot)$ extends from type {\rm II} points 
to a function $\ordRes_{\varphi} :\PP^1_{\Berk} \rightarrow [0,\infty]$
which is piecewise affine and convex upwards on each path in $\PP^1_{\Berk}$ 
with respect to the logarithmic path distance,
and is finite and continuous on $\HH_{\Berk}$ with respect to the strong topology, 
follows by the same argument as in the proof Theorem \ref{MainThm}.  
Indeed, $\ordRes_{\varphi}(\cdot)$ is Lipschitz continuous on $\HH_\Berk$, 
with Lipschitz constant $1^2 + 1 = 2$.
To prove the remaining assertions, we will make explicit computations in each case.

When $\varphi(z) = z$, it is easy to see that $\varphi^{\gamma}(z) = z$ for each $\gamma \in \GL_2(K)$,
and the assertions in part (1) of the Theorem follow trivially.

Next assume $\varphi$ has exactly two distinct fixed points $x_0, x_1$, 
and let $\gamma \in \GL_2(K)$ be such that $\varphi^{\gamma}(z) = C z$ with $|\,C| \le 1$, $C \ne 1$. 
After relabeling $x_0, x_1$ if necessary,
we can assume that $\gamma(0) = x_0$ and $\gamma(\infty) = x_1$.  
Given $A \in K^{\times}$ and $B \in K$, 
put $\tau = \tau_{A,B} = \left[ \begin{array}{cc} A & B \\ 0 & 1 \end{array} \right]$.
As $A$ and $B$ vary, the points $\zeta_{D(B,|A|)} = \tau_{A,B}(\zeta_G)$ range over all type II 
points in $\HH_\Berk$.  Consider the representation 
$(F^\gamma(X,Y),G^\gamma(X,Y) = (CX,Y)$ for $\varphi^{\gamma}$.  
One sees easily that $\ordRes(\varphi^\gamma) = \ord(C)$ 
and $(F^{\gamma \tau},G^{\gamma \tau}) = (AC X + B(C-1)Y, AY)$, which gives 
\begin{eqnarray}
& & \ordRes_{\varphi^\gamma}(\zeta_{D(B,|A|)}) \notag \\
& & \qquad \qquad \ = \ \max\big(\ord(C), \ord(C) - 2 \ord(B) - 2 \ord(C-1) + 2 \ord(A) \big) \ .
        \label{CaseCz} 
\end{eqnarray} 

When $|\,C| < 1$, or when $|\,C| = |\,C-1| = 1$, formula (\ref{CaseCz}) simplifies to 
\begin{equation*}
\ordRes_{\varphi^\gamma}(\zeta_{D(B,|A|)}) \ = \ 
          \max\big(\ord(C), \ord(C) + 2 \,\ord(A)- 2\, \ord(B) \big) \ .
\end{equation*} 
When $B = 0$, then $\ordRes_{\varphi^\gamma}(\zeta_{D(0,|A|)}) = \ord(C)$ for all $A$,
so  $\ordRes_{\varphi^{\gamma}}(\cdot) \equiv \ord(C)$ on the path $[0,\infty]$.
Next suppose $B \ne 0$.  The path $[B,\infty]$ meets $[0,\infty]$ at $\zeta_{D(0,|B|)}$, 
and for $|A| \le |B|$ we see that 
$\ordRes_{\varphi^\gamma}(\zeta_{D(B,|A|)}) = \ord(C) - 2 \ord(A/B) > \ord(C)$.  
Thus $\ordRes_{\varphi}(\cdot)$ increases as one moves away from $[0,\infty]$,
and $\ordRes_{\varphi^{\gamma}}(B) = \infty$.    
It follows that $\MinRes(\varphi^{\gamma}) = [0,\infty]$ and that $\ordRes_{\varphi^\gamma}(x) = \infty$
for all $x \in \PP^1(K) \backslash \{0,\infty\}$.  By Proposition \ref{ExtensionCor}, 
$\ordRes_{\varphi^{\gamma}}(\cdot)$ is continuous on $\PP^1_{\Berk} \backslash \{0,\infty\}$ 
relative to the strong topology.

When $|\,C-1| < 1$, formula (\ref{CaseCz}) becomes  
\begin{equation*}
\ordRes_{\varphi^\gamma}(\zeta_{D(B,|A|)}) \ = \ 
          \max\big(0, -2\,\ord(C-1) + 2\, \ord(A) - 2 \,\ord(B) \big) \ .
\end{equation*} 
When $B = 0$, then $\ordRes_{\varphi^\gamma}(\zeta_{D(0,|A|)}) = 0$ for all $A$,
so  $\ordRes_{\varphi^{\gamma}}(\cdot) = 0$ on $[0,\infty]$.
When $B \ne 0$,  for $|A| \le |B|$ we see that 
$\ordRes_{\varphi^\gamma}(\zeta_{D(B,|A|)}) = 0$ if $\ord(A/B) \le \ord(C-1)$,
while $\ordRes_{\varphi^\gamma}(\zeta_{D(B,|A|)}) = -2\ord(C-1) + 2 \ord(A/B) > 0$ 
if $\ord(A/B) > \ord(C-1)$.  Putting $R = \ord(C-1)$ we see that   
$\MinRes(\varphi^{\gamma})$ is the strong tube $T_{[0,\infty]}(R)$
and that $\ordRes_{\varphi^\gamma}(x) = \infty$
for all $x \in \PP^1(K) \backslash \{0,\infty\}$.  By Proposition \ref{ExtensionCor}, 
$\ordRes_{\varphi^{\gamma}}(\cdot)$ is continuous on $\PP^1_{\Berk} \backslash \{0,\infty\}$ 
relative to the strong topology. 
Transferring these assertions back to $\varphi$ using formula (\ref{Equicontinuity}), 
we obtain part (2) of the Theorem.

Finally suppose $\varphi$ has exactly one fixed point $x_0$. 
Let $\gamma \in \GL_2(K)$ be such that $\varphi^{\gamma}(z) = z + C$ with $C \ne 0$; 
then $\gamma(\infty) = x_0$.  
Given $A \in K^{\times}$ and $B \in K$, 
let $\tau = \tau_{A,B}$ be as above.
Consider the representation 
$(F^\gamma(X,Y),G^\gamma(X,Y) = (X + CY,Y)$ for $\varphi^{\gamma}$.  
Then $\ordRes(\varphi^\gamma) = 0$ 
and $(F^{\gamma \tau},G^{\gamma \tau}) = (AX + C Y, AY)$, which gives 
\begin{equation}
\ordRes_{\varphi^\gamma}(\zeta_{D(B,|A|)}) 
\ = \ \max\big(0, 2 \, \ord(A/C) \big) \ . \label{Casez+C}      
\end{equation} 
Put $R = |C|$.  For each $B \in K$, formula (\ref{Casez+C}) shows that 
$\ordRes_{\varphi^\gamma}(\zeta_{D(B,|A|)}) = 0$ if $|A| \ge |C|$, 
while $\ordRes_{\varphi^\gamma}(\zeta_{D(B,|A|)}) = 2 \,\ord(A/C) > 0$ if $|A| < |C|$.
Thus $\MinRes(\varphi^{\gamma})$ is the horodisc $H_\infty(R)$,
and  $\ordRes_{\varphi^\gamma}(x) = \infty$
for all $x \in \PP^1(K) \backslash \{\infty\}$.  By Proposition \ref{ExtensionCor}, 
$\ordRes_{\varphi^{\gamma}}(\cdot)$ is continuous on $\PP^1_{\Berk} \backslash \{\infty\}$ 
relative to the strong topology. 
Transferring these assertions back to $\varphi$ using formula (\ref{Equicontinuity}), 
we obtain part (3) of the Theorem.
\end{proof}


\newpage
\begin{thebibliography}{999}





\bibitem{Berk} V. G. Berkovich, Spectral theory and analytic geometry over non-Archimedean fields,
Mathematical Surveys and Monographs 33, American Mathematical Society,
Providence, RI, 1990.

\bibitem{B-R} M. Baker and R. Rumely, Potential Theory and Dynamics on the Berkovich Projective Line,
AMS Surveys and Monographs 159, Providence, 2010. 






\bibitem{BIJL} R. L. Benedetto, P. Ingram, R. Jones, and A. Levy, 
{\em Critical orbits and attracting cycles in $p$-adic dynamics}, 
Online preprint arXiv:12011605v2 (September 2012). 

\bibitem{Berk1} V.~G.~Berkovich, 
Spectral theory and analytic geometry over non-archimedean fields,
AMS Mathematical Surveys and Monographs 33, Providence, 1990.

\bibitem{Ber} E. R. Berlekamp, {\em Factoring polynomials over large finite fields}, 
Math. Comp. 24 (1970), 712-735.

\bibitem{BvHKS} K. Belabas, M. van Hoeij, J. Kl\"{u}ners, and A. Steel,
Factoring polynomials over global fields, J. Th Nombres Bordeax 21 (2009), 15-29.


\bibitem{BM} N. Bruin and A. Molnar, {\em Minimal Models for Rational Functions 
in a Dynamical Setting}, Online preprint  arXiv:1204:4967v1, (April 2012).

\bibitem{CG} D. Cantor and D. Gordon, {\em Factoring polynomials over $p$-adic fields}, 
in ANTS-IV, LNCS 1838, Springer, Berlin (2000), 185-208.


\bibitem{CZ} D. Cantor and H. Zassenhaus, {\em A new algorithm for factoring polynomials over finite fields}, 
Math. Comp. 36 (1984), 587-592. 

\bibitem{Coh} H. Cohen, A Course in Computational Algebraic Number Theory, 
Graduate Texts in Mathematics 138, Springer, New York Berlin Heidelberg, fourth printing, 2000.


\bibitem{Fab} X. Faber, {\em Topology and Geometry of the Berkovich Ramification Locus I, II},
Online preprints  arXiv:1102:1432 and arXiv:1104:0943 (May 2011). 

\bibitem{F-RL2} C. Favre and J. Rivera-Letelier, 
{\em Equidistribution des points de petite hauteur}, Math. Ann. 335(2), 2006, 311-361;. 
Online preprint arXiv:math/0407471.

\bibitem{F-RL1} C. Favre and J. Rivera-Letelier, 
{\em Th\'eor\`eme de Brolin en dynamique $p$-adique},
C. R. Math. Acad. Sci Paris 339 (2004), 271--276.  

\bibitem{FRLErgodic} C. Favre and J. Rivera-Letelier,{em Th\'eorie ergodique des fractions rationelles sur un 
corps ultram\'etrique}, Proc. London Math. Soc. 100(1) (2010), 116-154.

\bibitem{KS} E. Kaltofen and V. Shoup, {\em Subquadratic time factoring of polynomials over finite fields},
pp. 398-406 in: Proc. 27th ACM Symp. Th. Comp., New York, 1995.

\bibitem{AKL} A. K. Lenstra, {\em Lattices and factorization of polynomials over algebraic number fields}, 
in LNCS 114, Springer, Berlin (1982), 32-39.

\bibitem{LLL} A. K. Lenstra, H. W. Lenstra Jr, and L. Lovasz, {\em Factoring polynomials with rational coefficients},
Math. Ann. 261, No. 4, (1982), 515-534.

\bibitem{Mig} M. Mignotte, {\em An inequality about factors of polynomials},
Math. Comp. 28 (1974), 1153-1157.

\bibitem{Pohst} M. Pohst, {\em Factoring polynomials over global fields I}, 
J. Symbolic Computation 39 (2005), 617-630.  

\bibitem{PohstOm} J. M. Oman\~{a}  and M. Pohst, {\em Factoring Polynomials over global fields II}, 
J. Symbolic Computation 40 (2005), 1325-1339.

\bibitem{Pauli1} S. Pauli, {\em Factoring polynomials over local fields},  J. Symb. Comp. 32 (2001), 533-547.

\bibitem{Pauli2} S. Pauli, {\em Factoring polynomials over local fields II}, in ANTS-IX, LCNS 6197, 
Springer, Berlin (2000), 301-315.


\bibitem{R-L1} J. Rivera-Letelier, {\em Espace hyperbolique $p$-adique et dynamique des fonctions rationelles},
Compositio Math. 138(2) (2003), 199-231. 

\bibitem{RR-GMR} R. Rumely, {\em The Geometry of the Minimal Resultant Locus}, in preparation.

\bibitem{Sil} J. Silverman, The Arithmetic of Dynamical Systems, GTM 241, 
Springer-Verlag, New York 2007. 


\end{thebibliography}
\end{document}